\documentclass[11pt]{article}
\usepackage[T1]{fontenc}
\usepackage[margin=1in]{geometry}
\usepackage{amsmath, amsthm, amssymb, enumitem}
\usepackage{mathtools}
\mathtoolsset{showonlyrefs}
\usepackage{verbatim}
\usepackage[abbrev,lite,nobysame]{amsrefs}
\usepackage{times}
\usepackage{orcidlink}
\usepackage{esint}
\usepackage[compact]{titlesec}
\usepackage{booktabs}

\usepackage{hyperref}

\makeatletter
\def\DOI@strip#1https://doi.org/#2\@nil{#2}
\newcommand{\PrintDOIlink}[1]{%
  \in@{https://doi.org/}{#1}%
  \ifin@
  \href{#1}{DOI\nobreakspace\nolinkurl{\DOI@strip#1\@nil}}%
  \else
  \url{#1}%
  \fi
}
\makeatother
\BibSpec{article}{%
  +{}  {\PrintAuthors}                {author}
  +{,} { \textit}                     {title}
  +{.} { }                            {part}
  +{:} { \textit}                     {subtitle}
  +{,} { \PrintContributions}         {contribution}
  +{.} { \PrintPartials}              {partial}
  +{,} { }                            {journal}
  +{}  { \textbf}                     {volume}
  +{}  { \PrintDatePV}                {date}
  +{,} { \issuetext}                  {number}
  +{,} { \eprintpages}                {pages}
  +{,} { }                            {status}
  +{,} { \PrintDOI}                   {doi}
  +{,} { available at \eprint}        {eprint}
  +{}  { \parenthesize}               {language}
  +{}  { \PrintTranslation}           {translation}
  +{;} { \PrintReprint}               {reprint}
  +{.} { }                            {note}
  +{,} { \PrintDOIlink}               {url}
  +{.} {}                             {transition}
  +{}  {\SentenceSpace \PrintReviews} {review}
}
\BibSpec{book}{%
  +{}  {\PrintPrimary}                {transition}
  +{,} { \textit}                     {title}
  +{.} { }                            {part}
  +{:} { \textit}                     {subtitle}
  +{,} { \PrintEdition}               {edition}
  +{}  { \PrintEditorsB}              {editor}
  +{,} { \PrintTranslatorsC}          {translator}
  +{,} { \PrintContributions}         {contribution}
  +{,} { }                            {series}
  +{,} { \voltext}                    {volume}
  +{,} { }                            {publisher}
  +{,} { }                            {organization}
  +{,} { }                            {address}
  +{,} { \PrintDateB}                 {date}
  +{,} { }                            {status}
  +{}  { \parenthesize}               {language}
  +{}  { \PrintTranslation}           {translation}
  +{;} { \PrintReprint}               {reprint}
  +{.} { }                            {note}
  +{,} { \PrintDOIlink}               {url}
  +{.} {}                             {transition}
  +{}  {\SentenceSpace \PrintReviews} {review}
}

\usepackage{stmaryrd, mathrsfs}

\newcommand{\CC}{\mathbf{C}}
\newcommand{\RR}{\mathbf{R}}
\newcommand{\NN}{\mathbf{N}}
\newcommand{\ZZ}{\mathbf{Z}}

\newcommand{\TT}{\mathbf{T}}

\newcommand{\EE}{\mathbb{E}}

\newcommand{\PP}{\mathbb{P}}

\newcommand{\Sph}{\mathbb{S}}

\newcommand{\mL}{\mathcal{L}}

\newcommand{\mQ}{\mathcal{Q}}
\newcommand{\mC}{\mathcal{C}}
\newcommand{\mW}{\mathcal{W}}
\newcommand{\mM}{\mathcal{M}}

\newcommand{\mF}{\mathcal{F}}
\newcommand{\mD}{\mathcal{D}}
\newcommand{\mX}{\mathcal{X}}

\newcommand{\mE}{\mathcal{E}}
\newcommand{\mS}{\mathcal{S}}
\newcommand{\mB}{\mathcal{B}}
\newcommand{\mA}{\mathcal{A}}
\newcommand{\mH}{\mathcal{H}}

\newcommand{\mK}{\mathcal{K}}

\newcommand{\mf}[1]{\mathfrak{#1}}

\newcommand{\ve}{\varepsilon}
\newcommand{\vr}{\varrho}

\newcommand{\lqm}{``}
\newcommand{\rqm}{'' }
\newcommand*{\ud}{\mathrm{\,d}}
\newcommand{\eqdef}{\overset{\mathrm{\scriptscriptstyle def}}{=}}
\newcommand{\supp}{\operatorname{supp} } 
\renewcommand{\div}{\operatorname{div}}
\newcommand{\curl}{\operatorname{curl}\,}

\newcommand{\Op}{\mathrm{Op}}

\makeatletter
\newcommand{\opnorm}{\@ifstar\@opnorms\@opnorm}
\newcommand{\@opnorms}[1]{%
  \left|\mkern-1.5mu\left|\mkern-1.5mu\left|
   #1
  \right|\mkern-1.5mu\right|\mkern-1.5mu\right|
}
\newcommand{\@opnorm}[2][]{%
  \mathopen{#1|\mkern-1.5mu#1|\mkern-1.5mu#1|}
  #2
  \mathclose{#1|\mkern-1.5mu#1|\mkern-1.5mu#1|}
}
\makeatother

\renewcommand{\leq}{\leqslant}
\renewcommand{\geq}{\geqslant}
\renewcommand{\le}{\leqslant}
\renewcommand{\ge}{\geqslant}

\numberwithin{equation}{section}
\newtheorem{theorem}{Theorem}[section]
\newtheorem{lemma}[theorem]{Lemma}
\newtheorem{proposition}[theorem]{Proposition}
\newtheorem{corollary}[theorem]{Corollary}

\theoremstyle{definition}
\newtheorem{definition}[theorem]{Definition}
\newtheorem{remark}[theorem]{Remark}
\newtheorem{assumption}[theorem]{Assumption}

\newcommand{\h}{\mathfrak{h}}

\newcommand{\R}{\RR}

\newcommand{\CM}{H_{\mathcal{C}\mathcal{M}}}
\newcommand{\Zstar}{\ZZ^2_*}
\newcommand{\Zstarp}{\ZZ^2_{*,+}}

\DeclareMathOperator{\ad}{ad}

\newcommand{\vf}[1]{{#1}^\sharp}
\newcommand{\Opbeta}{\Op^{\beta_\star}}
\newcommand{\Opinf}{\Op^{\infty}}

\DeclareMathOperator{\Ran}{Ran}

\title{Unique ergodicity of projective cocycles over the 2D stochastic Navier--Stokes equations}
\author{
  Sam Punshon-Smith\,\orcidlink{0000-0003-1827-220X} \\
  \small Department of Mathematics, Tulane University \\
  \small \href{mailto:spunshonsmith@tulane.edu}{\texttt{spunshonsmith@tulane.edu}}
  \and
  Tommaso Rosati\,\orcidlink{0000-0001-5255-6519} \\
  \small Department of Mathematics, Imperial College London \\
  \small \href{mailto:trosati@ic.ac.uk}{\texttt{trosati@ic.ac.uk}}
}

\begin{document}

\maketitle

\begin{abstract}
  We consider the linear cocycles generated by the linearized vorticity
  equation and by passive scalar advection--diffusion, both driven by the
  two-dimensional stochastic Navier--Stokes flow on the torus
  with non-degenerate additive forcing, and we prove uniqueness of the stationary
  measures for the projective process associated to such linear dynamics. The
  proof relies on a localized asymptotic strong coupling construction, and on
  the non-degeneracy of the Malliavin matrix of the projective
  process. Establishing this non-degeneracy for passive scalar advection poses additional challenges, and
  requires a proof (based on Cameron--Martin analyticity arguments) that generically the passive scalar is nowhere one-dimensional.
\end{abstract}

\setcounter{tocdepth}{2}
\tableofcontents

\section{Introduction}
\label{sec:introduction} 
This work concerns long-time properties of linear dynamics associated to the two-dimensional
stochastic Navier--Stokes equations
\begin{equation}\label{eq:sns_vorticity_intro}
  \partial_t w + u \cdot \nabla w = \Delta w + \curl \xi ,
  \qquad u = K * w , \qquad w_0 \in \mH ,
\end{equation}
where $w$ is the vorticity, $u = K*w$ the velocity recovered through the
Biot--Savart kernel, and $\xi$ a white-in-time,
spatially colored forcing, see
Assumption~\ref{ass:noise}, and $\mH$ is $L^2$ with mean zero. 
The equation is posed on the torus $\TT^2 = \RR^2 /
(2\pi\ZZ)^2$.
The linear dynamics that we are interested in are of the form
\begin{equation}\label{eq:linearized_abstract}
  \partial_t \zeta = \Delta \zeta - L^0_{w} \zeta ,
  \qquad \zeta_0 \in \mH \setminus \{0\} ,
\end{equation}
where, writing $B(f,g) \eqdef (K * f) \cdot \nabla g$, the operator $L^0_w$ is
one of the following:
\begin{subequations}\label{eq:two_operators_intro}
  \begin{align}
    L^0_w \zeta &\eqdef B(w, \zeta) + B(\zeta, w)
    && \textbf{(LNS)} , \label{eq:two_operators_intro_NS} \\[2pt]
    L^0_w \zeta &\eqdef B(w, \zeta)
    && \textbf{(PSA)} . \label{eq:two_operators_intro_PS}
  \end{align}
\end{subequations}
The first is the equation for the linearization
of~\eqref{eq:sns_vorticity_intro} along $w$. The second is the equation for the
advection-diffusion of a passive scalar by the divergence free velocity field $u$.

The central object of this work is the top Lyapunov exponent $\lambda_1$ associated
to~\eqref{eq:two_operators_intro}. If $S_t$ is the solution operator
to~\eqref{eq:two_operators_intro}, so that $\zeta_t = S_t \zeta_0$, then 
\begin{equation}\label{eq:lambda1-def}
  \lambda_1 = \lim_{t \to \infty} \frac1t \log \|S_t\|_{\mathrm{op}},
\end{equation}
with $\| \cdot \|_{\mathrm{op}}$ the operator norm on $\mH$. The limit exists by
the subadditive ergodic theorem~\cite{Kingman73}.

The linear dynamics~\eqref{eq:two_operators_intro} and their Lyapunov exponents
are fundamental in the study of fluid dynamics and turbulence. In
the case of linearized Navier--Stokes at large Reynolds number, the top Lyapunov
exponent is expected to become positive, implying Eulerian chaos for the
Navier--Stokes dynamics~\cite{BCCV97}, with precise conjectures on the rate at which it
explodes, both in 2D~\cite{ohkitani,constantin_foias_temam} and in
3D~\cite{ruelle}, where it is connected to the emergence of spontaneous stochasticity~\cite{malibayev_eyink}. However, numerical
verifications are not
conclusive and sometimes in disagreement with the predictions~\cite{boffetta_musacchio,mohan_fitzsimmons_moser}.

In the case of passive scalar advection, the Lyapunov exponent captures the
mixing properties of the advecting velocity field. For a large class of velocity
fields, including~\eqref{eq:sns_vorticity_intro}, recent results have proven
exponential mixing uniform in diffusivity \cite{BBPS-AOP-22}. This corresponds to a
uniform in diffusivity upper bound $\lambda_1(\kappa )\lesssim -1$ (here $\kappa$ is the diffusivity parameter that would be
present in front of the Laplacian in~\eqref{eq:two_operators_intro_PS} but in
this work is set to $\kappa=1$). A matching lower bound $\lambda_1(\kappa)
\gtrsim -1$ is equally expected and studied numerically~\cite{miles_doering}, but at
least for generic velocity fields remains yet unproven. The lower bound
is connected to the formation of a critical lengthscale $\ell(\kappa) \simeq
\sqrt{\kappa}$ in the passive scalar, which is referred to as the Batchelor
scale~\cites{batchelor, miles_doering}.
Similar questions appear in many other areas of fluids, including for
instance the fast dynamo problem in magnetohydrodynamics
\cite{coti_zelati_sorella_villringer, gilbert}. 

Overall, these open problems show our severe shortcomings, not only in obtaining
mathematically rigorous estimates on Lyapunov exponents, but even in the
corroboration of their numerical study. The present article wants to take a step
in this direction.

In previous works~\cites{HairerRosati25, HPRY24} together with Hairer and Yi, we introduced a method
to prove that there is no spectral collapse (meaning $\lambda_1 > -\infty$) with
some quantitative lower bounds on the exponent in terms of the diffusivity
coefficient, but at unit Reynolds number. Note that while the infinite
dimensional ergodic theorem guarantees the existence of such exponents~\cites{Ruelle82, Lian-Lu10, Blumenthal16}, nothing
prevents them from collapsing to $-\infty$: this is the case for
pathological examples~\cites{BessaCarvalho08, Rowan}. Proving that such
exponents are not only finite, but become positive is a major challenge
even in finite dimensions~\cites{BBPS-JEMS-22, BBPS-L96-2021, Baxendale89}, and
widely open for SPDEs.

The starting point of the analysis in~\cite{HPRY24}
is the classical observation that properties of Lyapunov exponents are captured by the
long-time statistics of the couple $z_t= (w_t, \pi_t) \in \mX = \mH \times
\mathbb{S}$, where $\pi_t= \zeta_t /
\|\zeta_t\|$ is the projective process (although for simplicity we consider it
with values in the sphere $\mathbb{S} = \{\varphi \in \mH \colon \| \varphi
\|=1\}$). The projective process is itself the solution to a non-local,
non-parabolic SPDE
\begin{equation}\label{eq:proj-spde}
  \partial_t \pi = G_{w} \pi - \langle \pi, G_{w} \pi \rangle \pi ,
  \qquad G_w \eqdef \Delta - L^0_w ,
\end{equation}
which is hard to analyze.
However, the joint process $z$ is Markov, and the main
achievement of~\cite{HPRY24} is to construct a Lyapunov functional for that
process in order to prove existence of stationary measures for $z$. As a
consequence, we could obtain a lower bound on $\lambda_1$
\begin{equation*}
  \lambda_1 \geq \sup_{\nu} \int_{\mX} \langle \pi, G_w \pi \rangle \ud\nu(w, \pi) ,
\end{equation*}
where the supremum is taken over all stationary measures. The achievement of this
work is to prove uniqueness of such stationary measures, and as a consequence
derive the Furstenberg--Khasminskii formula.
\begin{theorem}\label{thm:main_result_intro}
Let Assumption~\ref{ass:noise} hold, and let $L^0_w$ be given by either \eqref{eq:two_operators_intro_NS} or~\eqref{eq:two_operators_intro_PS}. Then the Markov process $z_t = (w_t, \pi_t)$ on $\mX = \mH \times \Sph$ admits at most one stationary probability measure. In particular the stationary measure $\nu$ constructed in \cite{HPRY24} is the unique one.
\end{theorem}

As a consequence we deduce a Furstenberg--Khasminskii type formula for the
Lyapunov exponent.

\begin{corollary}\label{cor:FK_formula}
With $\nu$ as in Theorem~\ref{thm:main_result_intro},
  \begin{equation}\label{eq:FK_k}
    \lambda_1 = \int_{\mX} \langle \pi, G_w \pi \rangle \ud\nu(w, \pi) .
  \end{equation}
\end{corollary}
Moreover, the Lyapunov exponent is obtained from any deterministic initial
condition. This point is somewhat subtle: in principle the
multiplicative ergodic theorem implies the result below only for $\nu$-almost every
initial data $(w_0, \pi_0)$. 
\begin{corollary}\label{cor:deterministic_datum}
For every deterministic $w_0 \in \mH, \zeta_0 \in \mH \setminus \{0\}$, the solution
of~\eqref{eq:linearized_abstract} satisfies $$\lim_{t \to \infty} \frac{1}{t} \log
\|\zeta_t\| = \lambda_1$$ almost surely, with $\lambda_1$ as in Corollary~\ref{cor:FK_formula}.
\end{corollary}

This work opens the path toward
an understanding of the quantitative convergence of finite time Lyapunov
exponents to their long time limits, and the
corroboration of numerical methods for their estimation. In
deterministic systems, the convergence rates of finite-time Lyapunov exponents to
their long-time limits can exhibit pathological behavior, for example if the
dynamic is chaotic but trapped close to a periodic orbit for arbitrarily long
times. Stochastic settings should avoid such pathological cases by naturally
selecting generic behaviors~\cites{crauel_flandoli, blumenthal_engel_neamtu}. We expect this to be the case also in
the present setting, however the question of the existence of a spectral gap, and
central limit theorems or large deviation principles for finite time Lyapunov
exponents about their long-time limits is very difficult. Even more difficult is
to quantify these convergence rates in the diffusivity of the system.
Establishing uniqueness of the stationary measure is a first significant step in
this direction.

There are several challenges in establishing uniqueness of the stationary
measure. First, the noise acts only on the vorticity dynamics. This is a
standard example in which one does not expect strong Feller to hold, and
the uniqueness of the stationary measures passes through the H\"ormander
condition (the approximate invertibility of the Malliavin matrix). No matter how
non-degenerate our driving noise is (we must choose a highly non-degenerate noise in
order to apply~\cite{HPRY24}), the argument for uniqueness requires asymptotic
coupling techniques, and our setting lies beyond the
essentially elliptic one~\cites{glatt_holtz_mattingly_richards, kulik_scheutzow,
butkovsky_kulik_scheutzow}.

Moreover, the projective process does not solve a parabolic equation, and
in order to analyze~\eqref{eq:proj-spde}, we must rely on the spectral median
techniques developed in~\cites{HairerRosati25, HPRY24}. One issue is that~\cite{HPRY24} does not
prove sufficient control on $\pi$ in order to guarantee that the Jacobian
$J_t = D_{z_0} z_t$ of the process $z$ has finite polynomial moments: it would require
exponential moments of $\|\nabla \pi \|^2$, which are currently not established.
The lack of moments prevents us from
applying black box results for asymptotic strong
Feller~\cites{HairerMattingly06, HairerMattingly08, HairerMattingly11}. Instead, we
construct by hand an (anticipative) asymptotic coupling. To overcome the moment
problem, we localize the coupling by controlling the equation only when the
process lives in a central region of the state space (and then make sure that
there is no particularly strong separation outside of that region). Eventually, we apply the criterion from Hairer, Mattingly, and Scheutzow~\cite{HMS11} for
uniqueness. Some of the technical steps are inspired by the work by
Dong and Peng~\cite{DongPeng24}. In~\cite{DongPeng24} the degeneracy lies in the
driving noise, which acts in a few directions only, while here it lies in the
bundle structure, since the projective fiber is not forced directly. A related 
result is that of Kuksin, Nersesyan and Shirikyan~\cite{KNS20}, who obtain
uniqueness of the stationary measure and exponential mixing for a class of
dissipative equations from approximate controllability of the linearization.
Their result applies to
bounded, decomposable noise on a compact state space. In our setting the forcing is
unbounded (however, it is Gaussian) and the state space is not compact, which is
roughly why we require a localization argument.

In order to construct the coupling we require the approximate
invertibility of the Malliavin matrix for a sphere-valued dynamic. This is in
general highly non-trivial and model dependent. In the case of
LNS~\eqref{eq:two_operators_intro_NS}, we essentially use the results
of~\cite{BedrossianPunshon-Smith-Chaos-2024y}, where a similar problem was
tackled to prove chaos of Galerkin truncations of Navier--Stokes. That proof has
a strong algebraic flavor and requires the study of eigenvalues of a certain
class of 
diagonal matrices. Instead, PSA~\eqref{eq:two_operators_intro_PS} is entirely
new and poses significant challenges. The main issue is that the transport
dynamics alone preserve all Casimirs, and therefore the coupling between the
vorticity and the scalar can not act transitively by itself. This is a somewhat unusual situation to be in: in all the works
the authors are familiar with, including the
classical~\cite{HairerMattingly06} and the case
of LNS, the approximate Malliavin invertibility is induced by the nonlinearity.
Here instead, we must make use of the Laplacian.

This causes several problems. The parabolic Lie algebra generated by bracketing with the Laplacian
consists of operators of arbitrarily high order, and because of the strength of
our noise, the processes we consider have only finite (if arbitrarily large)
Sobolev regularity. Therefore, while some (AI assisted) algebraic calculations show that
the full parabolic Lie algebra should be transitive, we are not able to use this
fact. Instead, we show that a portion of that Lie algebra suffices, as long as
the scalar is sufficiently generic, in the sense that it is not locally of the
(one-dimensional) form $\zeta(t,x) = f(t, \ell \cdot x)$ for some $\ell \in
\RR^2$ and $x$ in some open patch. To prove that our stochastic passive scalar
is nowhere one-dimensional we use analyticity arguments. We note that this
approach is indirectly related to problems in fluid dynamics concerning for example
functional dependence in the steady states of Euler~\cites{ConstantinDrivasGinsberg21, HamelNadirashvili19, CotiZelatiElgindiWidmayer23, ElgindiHuangSaidXie}, see
Remark~\ref{rem:steady_euler}.

Shortly before finalizing this work, a related article has
appeared, addressing the same problem \cite{LianLiuLu26}. While the question is
similar, the proof methods are different: for
instance~\cite{LianLiuLu26} do not construct an explicit control or use
explicitly the approximate invertibility of the Malliavin matrix.
The article~\cite{LianLiuLu26} does not apply to the passive scalar case.

Overall, we set up a framework for the study of uniqueness of
stationary measures to projective processes in infinite dimensions. We prove that
the tools introduced here and in~\cite{HPRY24} are sufficiently strong to tackle
physically relevant systems, and that the study of Lyapunov exponents raises
interesting analytical and algebraic questions for future studies.

\section*{AI Disclosure}

In preparing this paper the authors made extensive use of AI tools (LLMs) for exploratory discussion, for drafting
text, and for automated verification of intermediate arguments. In
particular, we acknowledge the proof of Proposition~\ref{cor:control} (there are many proofs, but this one seemed particularly
elegant), and several preliminary versions of the Lie algebra
calculations behind the proof of Proposition~\ref{prop:transitivity_PSA} (which we regard as a substantial contribution). All
other uses were to fill in standard arguments in a human-defined path. In
preparing the final version we carefully optimized and rewrote any AI generated
arguments by hand, so the proofs and the exposition as they appear here are our
own. We have verified every statement and every proof ourselves, and we are
solely responsible for the correctness of the results and for any errors that
remain.

\section*{Notation}

We let $\TT^2 = \RR^2 / 2\pi\ZZ^2$ denote the two-dimensional torus and write
$\NN = \{0, 1, 2, \ldots\}$, $\NN_* = \NN \setminus \{0\}$, $\ZZ^2_* = \ZZ^2
\setminus \{0\}$, and $\iota = \sqrt{-1}$. We define $\Zstarp \eqdef \{k =
(k_1,k_2) \in \ZZ^2_* : k_1 > 0\} \cup \{(0,k_2) : k_2 > 0\}$, so that $\ZZ^2_*
= \Zstarp \sqcup (-\Zstarp)$. For a vector $j = (j_1, j_2) \in \RR^2$ we write
$j^\perp \eqdef (j_2, -j_1)$, and analogously we define the perpendicular
gradient $\nabla^\perp \eqdef (\partial_2, -\partial_1)$. For a sufficiently
smooth vector field $u$ we write $\curl u \eqdef \partial_1 u_2 - \partial_2
u_1$, which in the notation just fixed is $-\nabla^\perp \cdot u$. Then we set
$e_k(x) = e^{\iota k \cdot x}$ for $k \in \ZZ^2_*$, and for any mean-zero function
$\varphi \in L^2(\TT^2)$ we write $\varphi = \sum_{k \in \ZZ^2_*}
\hat{\varphi}(k) e_k$ with \[ \hat{\varphi}(k) = \frac{1}{(2\pi)^2} \int_{\TT^2}
\varphi(x) e^{-\iota k \cdot x} \ud x . \] 
We write $L^2(\TT^2)= \mH$ for
the Hilbert space of real-valued, mean-zero, square-integrable scalar functions
on $\TT^2$, equipped with the inner product $\langle f, g \rangle = \int_{\TT^2}
f(x) g(x) \ud x$. We denote by $\|\cdot\|$ the norm in $\mH$. We also use the
Bessel potential spaces $H^s(\TT^2)$ defined through the norm 
\begin{equation*}
  \| \varphi \|_{H^s}^2 = \sum_{k \in \ZZ^2_*} |k|^{2s} |\hat{\varphi}(k)|^2.
\end{equation*}
For $N \in \NN_*$, we denote by $P_N$ the orthogonal projection of $\mH$ onto
the span of the modes $\{e_k : k \in \ZZ^2_*,\ |k| \le N\}$, and $Q_N \eqdef I -
P_N$. We use the notation $a \lesssim b$ if $a \le C b$ for a constant $C$.

\section{Preliminaries}\label{sec:preliminaries}

We work throughout with the vorticity formulation~\eqref{eq:sns_vorticity_intro}
of the 2D stochastic Navier--Stokes equations, started from $w(0, \cdot) =
w_0(\cdot)$. In the Fourier basis, where $\hat{w}_k$ are the Fourier
coefficients of the vorticity, the
nonlinear term takes the form:
\begin{equation}\label{eq:nonlinearity_fourier}
  (u \cdot \nabla w)_\ell = \frac{1}{2}\sum_{\ell = k+ j\in \Zstar} c_{j,k} \hat{w}_{k} \hat{w}_j ,
\end{equation}
where
\begin{equation}\label{eq:interaction_coefficients}
  c_{j,k} = \langle j^\perp, k \rangle ( |k|^{-2} - |j|^{-2} ) 
\end{equation}
are the interaction coefficients. Recall from Section~\ref{sec:introduction} the bilinear vorticity advection operator $B(f,g) = (K * f) \cdot \nabla g$, with $K = \nabla^\perp(-\Delta)^{-1}$ the Biot--Savart kernel.

The noise $\xi$  is given by
\begin{equation}\label{eq:forcing}
  \xi = \sum_{k \in \ZZ^2_*} \sigma_k \frac{\iota k^\perp }{|k|^2} e_k\dot{W}_k ,
\end{equation}
where $\{\sigma_k\}$ are non-zero real coefficients with $\sigma_{-k} =
\sigma_k$, and $\{W_k\}_{k \in \ZZ^2_*}$ are a collection of standard, complex-valued
Brownian motions with covariance $\langle W_k, W_l \rangle_t = t \delta_{k+l=0}$
on a filtered probability space $(\Omega, \mathcal{F}, \{\mathcal{F}_t\}_{t \ge
0}, \PP)$. Equivalently, $W_{-k} = \overline{W_k}$ with $\{W_k\}_{k \in
\Zstarp}$ independent, so that $\xi$ is real-valued. The scaling in the
definition of $\xi$ is chosen so that
\begin{equation*}
  \curl \xi = \sum_{k \in \ZZ^2_*} \sigma_k e_k \dot{W}_k.
\end{equation*}
We require a non-degeneracy condition on the noise.
\begin{assumption} \label{ass:noise}
  There exists an $ \alpha_\star > 10$ such that for some $c, C> 0$
  \begin{equation}
    c |k|^{-\alpha_\star} \leq \sigma_k \leq C|k|^{-\alpha_\star} .
  \end{equation}
\end{assumption}

Under this assumption, it is classical that the SNS equation~\eqref{eq:sns_vorticity_intro} possesses a unique ergodic stationary measure $\mu$ on $L^2(\TT^2)$ \cite{Flandoli94}.

The non-degeneracy in Assumption~\ref{ass:noise} is essential for the results of
our previous paper \cite{HPRY24}, which supplies the stationary measure of the
fiber process and the Lyapunov structure on the sphere. Currently, it is not
known how to obtain the latter without non-degeneracy. It is also used less
fundamentally in several places to invert the noise covariance on all of $\mH$
(Lemma~\ref{lem:vanishing_implies_orthogonality}) and build admissible controls
(Proposition~\ref{prop:accessibility}). However these particular uses are
out of convenience and we expect the results used for uniqueness to hold also for
hypoelliptic forcing such as in \cite{HairerMattingly06}. The noise exponent of
\cite{HPRY24} corresponds roughly to $2\alpha_\star + 2$ in terms of our
$\alpha_\star$, so Assumption~\ref{ass:noise} implies the hypotheses of~\cite{HPRY24}.

The state space $\mX = \mH \times \Sph$ of $z = (w,\pi)$, has the structure of a
trivial fiber bundle over the base $\mH$. We view the unit sphere fiber $\Sph$
as an embedded infinite-dimensional Hilbert manifold, so the tangent space to
$\Sph$ at a unit vector $\pi$ is its orthogonal complement in $\mH$:
\begin{equation}\label{eq:sphere_tangent}
  T_\pi \Sph \cong \pi^\perp \eqdef \{v \in \mH : \langle v, \pi
  \rangle = 0\} .
\end{equation}
Then, the tangent space at a point $z \in \mX$ is the direct sum:
\begin{equation}
  T_{z}\mX = T_w \mH \oplus T_\pi \Sph \cong
  \mH \oplus \pi^\perp .
\end{equation}
A tangent vector at $z$ is an ordered pair $\eta = (\delta w, \delta\pi)$ where
$\delta w \in \mH$ and $\delta\pi \in \pi^\perp$. 

Whenever useful, we will work with functions of higher regularity. Under
Assumption~\ref{ass:noise}, standard parabolic regularity theory guarantees that
for any $w_0, \zeta_0 \in \mH$ we have $w_t \in H^s$ for all $s < \alpha_\star$
and $\zeta_t \in H^s$ for all $s < \alpha_\star+1$ (PSA has better regularity
than LNS, but here we treat both models at once), for all $t >0$. Since this
work is concerned with the long time behavior of the pair, it makes no
difference whether we start in $\mX$ or with more regular initial data, and
sometimes it will be useful to fix the latter. Therefore, we also define
\begin{equation}\label{eq:X_star}
  \mX_\star \eqdef H^{s_\star} \times (\Sph \cap H^{\beta_\star}), \qquad s_\star \eqdef \alpha_\star - \frac 32 , \qquad \beta_\star \eqdef \alpha_\star - \frac 12 .
\end{equation}
The parameters $(s_\star, \beta_\star)$ need only satisfy the bounds above. This
choice is somewhat arbitrary, but fixed for convenience to match the Lyapunov
functional taken from~\cite{HPRY24}. We record the regularity of $z$ in the
following remark.
\begin{remark}\label{rem:smoothing}
Almost surely, for any $z_0\in \mX$, it holds that $z_t \in \mX_\star$ for all $t >0$ in both~\eqref{eq:two_operators_intro_NS} and~\eqref{eq:two_operators_intro_PS}.
\end{remark}

Next, recall from~\eqref{eq:proj-spde} the generator $G_w = \Delta - L^0_w$
of the linear flow from \eqref{eq:linearized_abstract} on $\mH$, under either choice~\eqref{eq:two_operators_intro_NS} or~\eqref{eq:two_operators_intro_PS} of $L^0_w$. The map $w \mapsto L^0_w$ is linear, so setting $H^k \eqdef L^0_{e_k}$ for $k \in \Zstar$ gives the mode expansion
\begin{equation}\label{eq:generator_modes}
  G_w = \Delta - \sum_{k \in \Zstar} \hat w(k)\, H^k ,
\end{equation}
through which the noise acts on the fiber. Explicitly, $H^k = B(e_k,\cdot) +
B(\cdot,e_k)$ for the linearized Navier--Stokes flow and $H^k = B(e_k,\cdot)$
for the passive scalar. Throughout, $H^k$ with a lattice index $k \in \Zstar$
denotes this operator and $H^s$ with a real exponent the Sobolev space. The
difference between the two will be clear from context. Let $\Phi_t: \mX \to \mX$
be the stochastic flow map of the joint process, such that $z_t = \Phi_t(z_0)$.
We write $\Phi_t(z_0, W)$ when the driving path is to be displayed. Solutions
are constructed pathwise, by a fixed point run on each realization of the noise,
so that
\begin{equation*}
  (z_0, W) \ \longmapsto\ \bigl( \Phi_t(z_0, W) \bigr)_{t \ge 0}
\end{equation*}
is jointly Borel measurable from $\mX \times C\bigl([0,\infty); H^{-\alpha}(\TT^2)\bigr)$, $\alpha > 1$, into $C([0,\infty); \mX)$ \cites{HPRY24, Flandoli94, KuksinShirikyan12}. Section~\ref{sec:proof_of_main_theorem} uses this to transport an absolute continuity statement about (initial data, driving path) to one about path laws.

Next, we consider the Jacobian $J_t = D \Phi_t $, which is the Fr\'echet derivative of the flow with respect to the initial data:
\begin{equation}
  J_t(z_0): T_{z_0}\mX \to T_{z_t}\mX .
\end{equation}
The evolution of a tangent vector $\eta_t = (\delta w_t, \delta\pi_t)$ under the Jacobian
\begin{equation*}
  \eta_t = J_t(z_0) \eta_0
\end{equation*}
is then governed by the following equation:
\begin{equation}\label{eq:grand_jacobian_system}
  \partial_t
  \eta =
  \begin{pmatrix}
    L^{ww}_{w} & 0 \\
    L_{z}^{\pi w} & L_{z}^{\pi \pi}
  \end{pmatrix}
  \eta .
\end{equation}
The three non-zero blocks are defined immediately below. These terms contain differential operators and therefore do not, in general, map the tangent space of $\mX$ into itself. We define these operators on a smooth core, and then in the remainder of the article we will extend their definition to larger domains by density and continuity arguments. We define the smooth core
\begin{equation}\label{eq:smooth_core}
  \mS \eqdef C^\infty(\TT^2) \cap \mH,
\end{equation}
and then we set for given $ z = (w, \pi) \in \mS \times (\mS \cap \Sph)$:
\begin{itemize}
  \item $L^{ww}_w \colon \mS \to \mS $
is the linearization of the base SNS dynamics on the vorticity component, \[ L^{ww}_w\delta w \eqdef \Delta\delta w - B(w, \delta w) - B(\delta w, w) , \] with $B$ the bilinear advection operator introduced above. This block is independent of $L^0_w$ in~\eqref{eq:linearized_abstract}.
  \item $L_{z}^{\pi w} \colon \mS \to  \mS$ acts by
    \begin{equation}\label{eq:coupling_operator}
      L_{z}^{\pi w}[\delta w]
      = - (DL^0_w)[\delta w] \pi
      + \bigl\langle \pi, (DL^0_w)[\delta w] \pi \bigr\rangle \pi ,
    \end{equation}
    where $(DL^0_w)[\delta w]$ is the Fr\'echet derivative of $L^0_w$ with respect to $w$, given by
    \begin{equation*}
      (DL^0_w)[\delta w] v
      =
      \begin{cases}
        B(\delta w, v) + B(v, \delta w)
        & \text{in case~\eqref{eq:two_operators_intro_NS}} ,\\[2pt]
        B(\delta w, v)
        & \text{in case~\eqref{eq:two_operators_intro_PS}} .
      \end{cases}
    \end{equation*}
  \item $L_{z}^{\pi \pi} \colon \mS \to
    \mS$ acts by
    \begin{equation}\label{eq:fiber_operator}
      L_{z}^{\pi \pi}\delta\pi
      = G_w\delta\pi
      - \langle \pi, G_w\delta\pi \rangle \pi
      - \langle \delta \pi, G_w \pi \rangle \pi
      - \langle \pi, G_w \pi \rangle \delta\pi .
    \end{equation}
\end{itemize}
Overall, we write for short
\begin{equation*}
  L_{z} \eqdef
  \begin{pmatrix}
    L^{ww}_{w} & 0 \\
    L_{z}^{\pi w} & L_{z}^{\pi\pi}
  \end{pmatrix} .
\end{equation*}
Due to this lower-triangular structure, the Jacobian $J_t$ has also lower-triangular form:
\begin{equation}\label{eq:J-matrix}
  J_t =
  \begin{pmatrix} J_t^{ww} & 0 \\ J_t^{\pi w} & J_t^{\pi\pi}
  \end{pmatrix} ,
\end{equation}
and it solves the equation
\begin{equation*}
  \partial_t J_t = L_{z_t} J_t .
\end{equation*}
\subsection{Standing estimates}
Many estimates of this paper will use the following
inequalities. For $f, g \in \mH$ and any $\ve \in (0,1]$,
\begin{equation}\label{eq:l4_toolkit}
  \|g\|_{L^4} \lesssim \|g\|^{1/2} \|g\|_{H^1}^{1/2} ,
  \qquad
  \|K * f\|_{L^4} \lesssim \|f\| ,
  \qquad
  \|K * f\|_{L^\infty} \lesssim_\ve \|f\|_{H^\ve}.
\end{equation}
For a state $z = (w,\pi) \in \mX$ we write
\begin{equation}\label{eq:gamma_r}
  \Gamma(z) \eqdef \|w\|_{H^1}^2 + \|\pi\|_{H^1}^2 ,
\end{equation}
abbreviating $\Gamma_r \eqdef \Gamma(z_r)$ along a trajectory and dropping the argument when the state is clear from the context. Since $|k| \ge 1$ for $k \in \ZZ^2_*$, the Poincar\'e inequality on mean-zero functions gives $\|\psi\|_{H^1} \ge \|\psi\|$ for every $\psi \in \mH$, so that $\Gamma(z) \ge \|\pi\|_{H^1}^2 \ge \|\pi\|^2 = 1$.

We will also use a super-Lyapunov functional for $u = K*w$. For $\rho > 1$ define \cite{HPRY24}*{Eq.~(7.10)}
\begin{equation}\label{eq:lyapunov_S4}
  V_\rho(u) \eqdef (1 + \|u\|_{H^{\beta_\star}}^2)^\rho \exp\big(c_\star \|u\|_{H^1}^2\big) ,
\end{equation}
where $\beta_\star = \alpha_\star-1/2$ {as in~\eqref{eq:X_star}} and $c_\star >
0$ are fixed constants depending only on the noise regularity $\alpha_\star$. It
enjoys the super-Lyapunov property (cf. \cite{HPRY24}*{Lemma~7.4}), namely that
for every $\gamma > 0$ there exists $C_\gamma > 0$ with
\begin{equation}\label{eq:super_lyapunov}
  \EE[V_\rho(u_t)] \le e^{-\gamma t} V_\rho(u_0) + C_\gamma , \qquad t \ge 0 .
\end{equation}
The second is an energy--enstrophy bound stated in H\"older spaces, so for a
non-integer $\alpha \in (0, \infty) \setminus \NN_*$ we write $\mC^\alpha =
\mC^\alpha(\TT^2)$ for the space of $\alpha$-H\"older continuous functions.
With $m = \lfloor \alpha \rfloor$ and $\theta = \alpha - m \in
(0,1)$, it consists of those $f$ for which 
\begin{equation}\label{eq:holder_def_S4}
  \|f\|_{\mC^\alpha} \eqdef \sum_{|j| \le m} \|\partial^j f\|_{L^\infty} + \sum_{|j| = m} \sup_{x \ne y} \frac{|\partial^j f(x) - \partial^j f(y)|}{|x-y|^\theta} < \infty.
\end{equation}
The following is the content of \cite[Lemma 6.4 and Lemma 7.5]{HPRY24}. 
\begin{proposition} \label{prop:energy_enstrophy}
  Fix $T > 0$. Then for $V_\rho$ as in~\eqref{eq:lyapunov_S4}, the following hold:
  \begin{enumerate}[label=(\roman*)]
    \item There exists a $q \geq 1$ such that for every $p \ge 1$ and every $\beta \in (1, \alpha_\star) \setminus \NN$
      there exists a $C = C(p, \beta, T) > 0$ with
      \begin{equation}\label{eq:ee_poly}
        \EE\Big[ \sup_{t \in [0,T]} \|w_t\|_{\mC^{\beta}}^p \Big] \le C(1 + \| w_0 \|_{\mC^{\beta}})^{pq} .
      \end{equation}
    \item For every $p \geq 1$ and
      $0 \le s \le t$ there exists $C = C(p, t) > 0$ such that, at $\rho = p$,
      \begin{equation}\label{eq:ee_fiber}
        \EE\big[ \|\pi_t\|_{H^1}^p \big]
        \le C \EE\big[ \|\pi_s\|_{H^1}^{2p} + 1 \big]^{1/2} V_\rho(u_0)^{1/2} .
      \end{equation}
  \end{enumerate}
\end{proposition}

\section{Non-degeneracy of the Malliavin matrix}
\label{sec:quantitative_nondegeneracy}

One of the key steps in establishing uniqueness of stationary measures for stochastic PDEs via asymptotic coupling is the non-degeneracy of the Malliavin matrix. This is the content of the present section. We defer the algebraic density arguments that are required in this step to the next Section~\ref{sec:transitivity}.

We denote with $Q\colon \mH \to \mH$ the covariance operator of the stochastic
forcing $\curl \xi$ introduced in~\eqref{eq:forcing}. This operator is diagonal
in the Fourier basis, with $Q e_k = \sigma_k^2  e_k$ and positive square root
$Q^{1/2} e_k = \sigma_k e_k$. Under Assumption~\ref{ass:noise}, $\sigma_k^2
\simeq |k|^{-2\alpha_\star}$ with $\alpha_\star > 10$, so $Q$ is trace-class on
$\mH$. The Cameron--Martin space of the colored noise is
$\mathrm{Range}(Q^{1/2}) = \bigl\{f \in \mH : \textstyle\sum_k
|\hat{f}_k|^2/\sigma_k^2 < \infty\bigr\}$. However, we will use controls that
act on the underlying space-time white noise rather than the colored noise. We
therefore define the Cameron--Martin space \[ \CM \eqdef L^2\bigl([0,\infty);
\mH\bigr) \] so that to any $h \in \CM$ we associate the perturbation
$Q^{1/2}h(s)$ of the colored noise at time $s$. Next we extend the definition of
the Jacobian to two time points, allowing us to start from some arbitrary time
$s>0$. For $0 \le s \le t$, it is the Fr\'echet derivative of the flow map
$\Phi_{s,t}\colon \mX \to \mX$ with respect to the initial condition at time
$s$: \[ J_{s,t} \eqdef D\Phi_{s,t}(z_s) \colon T_{z_s}\mX \to T_{z_t}\mX . \] It
satisfies $J_{s,s} = \mathrm{Id}$ and the cocycle identity $J_{s,t} = J_{r,t}
\circ J_{s,r}$ for $s \le r \le t$. As a function of $t$, it solves \[ \frac{\ud}{\ud t} J_{s,t} = L_{z_t} J_{s,t} , \] where
$L_{z_t}$ is the matrix appearing in~\eqref{eq:grand_jacobian_system} evaluated
along the trajectory $z_t$. Then we define the Malliavin derivative of
$\Phi_t(z_0) = z_t$ in the direction $h$:
\begin{equation}
  A_t \colon \CM \to T_{z_t} \mX, \qquad A_t h = \int_0^t J_{s,t} (Q^{1/2}h(s), 0) \ud s .
\end{equation}
The term $(Q^{1/2}h(s), 0)$ reflects that the noise acts only on the $w$ component. Next, we define the Malliavin matrix to be the operator
\begin{equation*}
  M_t (z_0): T_{z_t} \mX \to T_{z_t} \mX, \qquad M_t(z_0) = A_t A_t^* .
\end{equation*}
Whenever clear from context, we will drop the dependency on the initial condition $z_0 \in \mX$. For a tangent vector $\eta = (\delta w, \delta\pi) \in T_{z_t} \mX$, the quadratic form of the Malliavin matrix is given by:
\begin{equation}
  \langle M_t \eta, \eta \rangle = \|A_t^* \eta\|^2_{L^2([0,t]; \mH)} = \int_0^t \|Q^{1/2} \Pi_w (J_{s,t})^* \eta \|^2 \ud s ,
\end{equation}
where $\Pi_w$ is the projection onto the base component of the tangent space and
\begin{equation*}
  (A_t^*\eta)(s) = Q^{1/2}\Pi_w(J_{s,t})^*\eta .
\end{equation*}
In terms of the block structure of $J$ we rewrite:
\begin{equation}
  \langle M_t \eta, \eta \rangle = \int_0^t \left\| Q^{1/2} \left( (J_{s,t}^{ww})^* \delta w + (J_{s,t}^{\pi w})^* \delta\pi \right) \right\|^2 \ud s .
\end{equation}
The objective of this section is to identify (and later verify in Section~\ref{sec:transitivity}) algebraic conditions under which the Malliavin matrix can not be degenerate. To state the main result of the section, consider parameters $\alpha, T>0$ and $N \geq 1$, and define the sliced (because of the condition $\| \eta \| =1$) cone
\begin{equation}\label{eq:cone_SaN}
  S_{\alpha, N, T} \eqdef \left\{ \eta = (\delta w, \delta\pi) \in T_{z_T} \mathcal{X} :    \|P_N \eta\| \geq \alpha \| \eta \| ,  \| \eta \| =1 \right\} ,
\end{equation}
where $P_N$ is the orthogonal projection onto the first $N$ Fourier modes of the
tangent space $T_{z_T} \mathcal{X}$.
Furthermore define the following \lqm central region\rqm of the state space.
This will play a role throughout the paper, since we must localize our arguments
to situations in which we have good control on the projective process (this is a
consequence of not having finite moments for the Jacobian). 
\begin{definition}\label{def:central_region}
  For a fixed parameter $R > 0$, the {\em central region} $\mC_R$ is the sublevel set
  \begin{equation}\label{eq:central_region}
    \mC_R \eqdef \left\{ z = (w, \pi) \in \mX : \|w\|_{H^{s_\star}} + \|\pi\|_{H^1} \le R \right\} .
  \end{equation}  
\end{definition}

Here $s_\star$ is as in~\eqref{eq:X_star}. Note that $\mC_R \subseteq \mX$ is
compact. The main result of this section is the following non-degeneracy.

\begin{proposition}\label{prop:uniform_tail}
  For any $\ve, T, R, N > 0$ and $\alpha \in (0,1]$, define:
  \begin{equation}
    r(\ve, \alpha, T, R, N) \eqdef \sup_{z_0 \in \mC_R} \PP\left( \inf_{\eta \in S_{\alpha, N, T}} \langle M_T(z_0) \eta, \eta \rangle < \ve \right) .
  \end{equation}
  Then $\lim_{\ve \to 0} r(\ve, \alpha, T, R, N) = 0$.
\end{proposition}

For the base equation alone this is \cite{DongPeng24}*{Proposition~3.2}, and the argument below extends it to the bundle.
The proof of this result can be found at the end of Section~\ref{sec:uniform_tail}.
The key intermediate step toward Proposition~\ref{prop:uniform_tail} is the
pointwise (in the initial condition) almost sure non-degeneracy stated below. In
Section~\ref{sec:uniform_tail}, we use a compactness argument due to
\cite{DongPeng24} to obtain the result uniformly over initial data.

\begin{proposition}\label{prop:as_nondegeneracy}
  For any $\alpha \in (0,1]$, $T > 0$, $N \geq 1$, and any initial condition $z_0 \in \mathcal{X}$:
  \begin{equation}
    \PP\left( \inf_{\eta \in S_{\alpha, N, T}} \langle M_T(z_0) \eta, \eta \rangle > 0 \right) = 1 .
  \end{equation}
\end{proposition}
For the SNS equations alone, driven by degenerate noise, this is
\cite{DongPeng24}*{Proposition~3.1}, the proof below must control the projective
process.
The proof of this result can be found in Section~\ref{sec:proof-as-nondegeneracy}.

\subsection{Admissible operators and vector fields}\label{sec:admissible_operators}
Because of the finite Sobolev regularity of the forcing in
Assumption~\ref{ass:noise}, base vorticity $w_t$ (and consequently the projective
process $\pi_t$) has only finite Sobolev regularity, so we may only work with
operators that take derivatives up to a fixed order. This does not limit the
number of Lie brackets as such. Brackets among the generators below do not raise
the order and may be iterated freely. It is the brackets with the Laplacian that
raise the order, so only finitely many of these are allowed, as
Definition~\ref{def:bracket_regular} makes precise. However, all operators in
this section preserve the common smooth invariant domain $\mS = C^\infty(\TT^2)
\cap \mH$ of~\eqref{eq:smooth_core}. Given two Banach spaces $X$ and $Y$, we
denote by $\mL(X; Y)$ the space of bounded linear operators from $X$ to $Y$,
equipped with the operator norm
\begin{equation*}
  \| T \|_{\mL(X; Y)} \eqdef \sup_{x \in X \setminus \{0\}} \frac{\| T x \|_Y}{\| x \|_X} ,
\end{equation*}
and we abbreviate $\mL(X) \eqdef \mL(X; X)$.
\begin{definition}\label{def:admissible_operators}
  A linear operator $A \colon \mS \to \mS$ is \emph{admissible} if it is continuous in the Fr\'echet topology of $\mS$. (Differential operators with smooth coefficients, while possibly unbounded on $L^2$, are continuous on $\mS$).
  \begin{itemize}
    \item   We denote with $\Opinf$ the set of
      admissible operators on $\mS$.
    \item We denote with $\Opbeta \subseteq \Opinf$ the
      subset of admissible operators that extend (uniquely) to bounded linear operators $H^{\beta_\star} \to \mH$, with $\beta_\star$ as in~\eqref{eq:X_star}. We endow $\Opbeta$ with the operator norm topology of $\mL(H^{\beta_\star};\mH)$, and write $\|A\|_{\Opbeta} \eqdef \|A\|_{\mL(H^{\beta_\star};\mH)}$ for the corresponding norm.
  \end{itemize}
\end{definition}
The operators we build from the Fourier modes are complex valued. All the definitions here and below extend from real to complex scalars in the obvious way, as in Section~\ref{sec:transitivity}. By Schauder theory, the solution $\zeta$ to \eqref{eq:linearized_abstract} satisfies $\zeta \in H^{\beta}$ for every $\beta < \alpha_\star+1$. Therefore any $A \in \Opbeta$ has a well defined action on the solution $\zeta$ to \eqref{eq:linearized_abstract} and similarly on its projective component $\pi = \zeta / \|\zeta \| $, see Remark~\ref{rem:smoothing}. We also note that while $\Opinf$ is a Lie algebra, the same does not hold for $\Opbeta$.

The test operators we construct in the sequel are differential operators with
smooth Fourier coefficients. We work directly with the classes $\Opbeta$ rather
than abstract pseudodifferential operators, because the drift operator $G_w$ has
only finite Sobolev regularity $H^{s_\star}$, and the solutions $\pi_t$ live in
$H^{\beta_\star}$. For a differential operator $A$ of order $d$, brackets with
the first-order operators $H^k$ of~\eqref{eq:generator_modes} preserve the
differential order, whereas each bracket with the Laplacian $\Delta$ increases
the order by two. Writing $\ad(X) Y \eqdef [X, Y]$ and $\ad(X)^j$ for its $j$-fold
iteration, we find for $A \in \mL(H^d; \mH)$:
\begin{equation*}
  \ad(H^k)^j A \in \mL(H^{d}; \mH) \quad \text{for }j \ge 0 ,
  \quad
  \ad(\Delta)^m A \in \mL(H^{d+2m}; \mH) \subseteq \mL(H^{\beta_\star}; \mH)
  \quad \text{for } 2m \le \beta_\star - d .
\end{equation*}
Finally, for any admissible operator $A$ and any unit vector $\pi \in \Sph \cap \mS$, we define the projective action
\begin{equation}\label{eq:fundamental_vf}
  \vf{A}\pi \eqdef P_{\pi^\perp} A \pi
  = A\pi - \langle \pi, A\pi \rangle \pi
  \in \pi^\perp \cap \mS .
\end{equation}
 We view $\vf{A}$ as a vector field on $\Sph$, and in this sense consider the Lie bracket between two $\vf{A}, \vf{B}$. We record the relationship between the Lie algebra structure of $\Opinf$ and that of the induced vector fields on $\mathbb{S}$.

\begin{proposition}
  \label{prop:anti_homomorphism} For $A, B \in \Opinf$ it holds that
  \begin{equation}\label{eq:anti_hom}
    [\vf{A}, \vf{B}] = -\vf{[A,B]} ,
  \end{equation}
  where the left side is the Lie bracket of vector fields on $\Sph \cap \mS$ and $[A,B] = AB - BA$, is the standard commutator.
\end{proposition}

\begin{proof}
  Fix $\pi \in \Sph \cap \mS$ and note that $[\vf{A}, \vf{B}] = D \vf{B} \cdot \vf{A} - D \vf{A} \cdot \vf{B}$. Differentiating $\vf{A}\pi = A\pi - \langle \pi, A\pi\rangle\pi$ along $h \in \pi^\perp \cap \mS$ by the Leibniz rule gives
  \begin{equation}\label{eq:frechet_vf}
    D_\pi \vf{A} \cdot h = Ah
    - \langle h, A\pi \rangle \pi
    - \langle \pi, Ah \rangle \pi
    - \langle \pi, A\pi \rangle h .
  \end{equation}
  Inserting $h = \vf{B}\pi$, (and similarly with $A$ and $B$ flipped) one obtains
  \begin{align*}
    [\vf{A}, \vf{B}](\pi)
    & = (BA - AB)\pi + \langle \pi, (AB - BA)\pi\rangle\pi
    = -\vf{[A,B]}\pi .
  \end{align*}
\end{proof}
As we will see in the next section, the analysis of the parabolic H\"ormander
condition quickly reduces to the transitivity of certain Lie algebras of
admissible operators.

\subsection{A truncated parabolic H\"ormander condition}
\label{sec:hormander}

The aim of this subsection is to identify a suitable subset $\mf{g} \subseteq
\Opbeta$ such that if for some $\eta = (\delta w, \delta \pi) \in T_{z_T} \mX$
we have $ \langle M_T(z_0) \eta, \eta \rangle =0$, then we have $\delta w=0$ and
$ \langle \delta \pi, \vf{A} \pi \rangle = 0$ for all $A \in \mf{g}$. The choice
of $\mf{g}$ will differ substantially depending on whether we are in the setting
of~\eqref{eq:two_operators_intro_NS} or~\eqref{eq:two_operators_intro_PS}.
Omitting the issue of finite
regularity, the reader should think of both $\mf{g}$ and $\Opbeta$ as Lie
algebras. Because we work with finite regularity, we sometimes truncate the Lie
algebra at a level that is useful to us, which is why the eventual $\mf{g}$
might not be a Lie algebra (it turns out that it is one for the LNS case, and it
is not for PSA). Now, recall from~\eqref{eq:generator_modes} that
\begin{equation}\label{eq:Hj_def}
  H^k = -D_w G_w[e_k] \in \Opbeta,
\end{equation}
where the inclusion is verified in Lemma~\ref{lem:lns_algebra} and Lemma~\ref{lem:psa_algebra}.

\begin{lemma}\label{lem:vanishing_implies_orthogonality}
  Fix $T > 0, z_0 \in \mX, \eta = (\delta w, \delta\pi) \in T_{z_T}
  \mathcal{X}$. Suppose that $\langle M_T(z_0) \eta, \eta \rangle = 0$ and define
  \begin{equation}\label{eq:eta_def}
    \eta_s \eqdef (J^{\pi\pi}_{s,T})^* \delta\pi \in T_{\pi_s}\Sph , \qquad \forall s \in [0,T] .
  \end{equation}
  Then the following hold.
  \begin{enumerate}[label=(\roman*)]
    \item We have $\delta w =0$ and $(J^{\pi w}_{s,T})^* \delta\pi = 0 $, $\forall s \in [0,T]$.
      Hence, $\langle \eta, J_{s,T}(f, 0) \rangle = 0$ for all $f \in \mH, s \in [0,T]$.
    \item For every $s \in (0,T]$ and every $k \in \Zstar$ it holds that $
      \langle \eta_s, \vf{H^k}\pi_s\rangle = 0$.
  \end{enumerate}
\end{lemma}
While the initial data $\pi_0 \in \mathbb{S}$ is not necessarily smooth, at a
later time we have $z_s \in \mX_\star$, so that $ \vf{H^k}\pi_s \in
H^{\beta_\star -1} \subseteq \mH$, for all $s >0$ (see Remark~\ref{rem:smoothing}), and in particular the last
bracket is well defined. Such considerations are made silently throughout the
remainder of this section.
\begin{proof}
  \emph{Part~(i).} From the definition of the Malliavin matrix, if $\langle M_T \eta, \eta \rangle=0$, then $Q^{1/2} \Pi_w (J_{s,T})^* \eta = 0$ for all $s \in [0, T]$ since the integrand is continuous in time. Therefore by injectivity of $Q^{1/2}$, we conclude
  \[
    \Pi_w (J_{s,T})^* \eta = 0  \qquad \forall s \in [0,T].
  \]
  By the block structure~\eqref{eq:grand_jacobian_system}, this reads $(J^{ww}_{s,T})^* \delta w + (J^{\pi w}_{s,T})^* \delta\pi = 0$. Since at $s=T$ we have $J^{\pi w}_{s,T} =0$ we conclude that $\delta w =0$. Pairing with any $f \in \mH$ leads to
  \[
    \langle \delta w, J^{ww}_{s,T} f \rangle + \langle \delta\pi, J^{\pi w}_{s,T} f \rangle = \langle \eta, J_{s,T}(f, 0) \rangle = 0,
  \]
  which is the last assertion of part~(i).

  \emph{Part~(ii).}
  We observe that $J^{\pi w}_{s, t}$ starts from $J^{\pi w}_{s,s} = 0$ and satisfies
  \begin{equation}
    J^{\pi w}_{s,T} f = \int_s^T J^{\pi\pi}_{r,T} L^{\pi w}_{\pi_r} J^{ww}_{s,r} f \ud r ,
  \end{equation}
  which is the Duhamel formula for $\partial_t J^{\pi w}_{s,t} = L^{\pi \pi}_{z_t} J^{\pi w}_{s,t} + L^{\pi w}_{z_t} J^{w w}_{s, t}$. We insert this identity above and move $(J^{\pi\pi}_{r,T})^*$ onto $\delta\pi$ through the definition $\eta_r = (J^{\pi\pi}_{r,T})^* \delta\pi$. This gives
  \begin{equation}\label{eq:integral_vanishing}
    \int_s^T \langle \eta_r, L^{\pi w}_{\pi_r} J^{ww}_{s,r} f\rangle \ud r = 0 \qquad \forall f \in \mH,s \in [0,T] .
  \end{equation}
  We recover the integrand by differentiating \eqref{eq:integral_vanishing} in
  $s$. Fix $k \in \Zstar$ and take $f = e_k$, which is smooth, so that $s
  \mapsto J^{ww}_{s,r} e_k$ is differentiable with $\partial_s J^{ww}_{s,r} =
  -J^{ww}_{s,r} L^{ww}_{w_s}$. Here $e_k$ is complex valued, but this does not
  matter, as we can take separately the real and imaginary part. Using also
  $J^{ww}_{s,s} = I$, differentiation of \eqref{eq:integral_vanishing} at any
  $s \in (0,T]$ gives
  \begin{equation}
    \langle \eta_s, L^{\pi w}_{\pi_s} e_k\rangle = - \int_s^T \langle \eta_r, L^{\pi w}_{\pi_r} J^{ww}_{s,r} L^{ww}_{w_s} e_k\rangle \ud r .
  \end{equation}
  The vector $L^{ww}_{w_s} e_k$ lies in $\mH$ for $s > 0$, by
  Remark~\ref{rem:smoothing}, so the right side is the left side of
  \eqref{eq:integral_vanishing} evaluated at $f = L^{ww}_{w_s} e_k$, and it
  therefore vanishes by part~(i). We deduce $ \langle \eta_s, L^{\pi
  w}_{\pi_s} e_k\rangle = 0$ for all $s \in (0,T]$.
  Since  $L^{\pi w}_{\pi} e_k = -\vf{H^k}\pi$, the result is proven.
\end{proof}

We now extend the previous result to iterated Lie brackets of $H^k$. The extension asks three regularity conditions of the operator, and it applies to families generated from the $H^k$ by brackets that keep them.

\begin{definition}\label{def:bracket_regular}
  Fix $T > 0$ and $z_0 \in \mX$. An operator $A \in \Opbeta$ is bracket-regular if:
  \begin{enumerate}[label=(\roman*)]
    \item for every $k \in \Zstar$ the commutator $[H^k, A]$ is bounded on both
      of the scales we shall test it against:
      \begin{enumerate}[label=(i-\alph*), ref=(i-\alph*)]
        \item $[H^k, A] \in \Opbeta$ with
          $\sup_{k \in \ZZ^2_*}  |k|^{-2} \| [H^k, A] \|_{\Opbeta} < \infty$,
        \item $[H^k, A] \in \mL(H^{\beta_\star-2}; H^2)$ with
          $\sup_{k \in \ZZ^2_*} |k|^{-4} \| [H^k, A] \|_{\mL(H^{\beta_\star-2};\, H^2)} < \infty$,
      \end{enumerate}
    \item the map $s\mapsto [G_{w_s}, A]$ belongs to $C( (0, T];
      \mL(H^{\beta_\star-2}; H^2))$ almost surely,
    \item the commutator ${[L^0_{w}, A]}$ satisfies $ \|{[L^0_{w}, A]}
      \|_{\mL(H^{\beta_\star}; \mH)} \lesssim \| w \|_{H^{s_\star-2}}$.
  \end{enumerate}
  A subspace $\mf{g} \subseteq \Opbeta$ is \emph{bracket-generated} if $\mf{g} = \bigcup_{j \geq 0} \mf{g}_j$ for an increasing sequence of subspaces starting from $\mf{g}_0 = \mathrm{Span}(H^k \colon k \in \Zstar)$ and satisfying
  \begin{equation}\label{eq:reachable_chain}
    \mf{g}_{j+1} = \mf{g}_j + \mathrm{Span}(\mathcal{B}_j) ,
    \qquad
    \mathcal{B}_j \subseteq \bigl\{ [\Delta, A], \ [H^k, A] \colon A \in \mf{g}_j , \ k \in \Zstar \bigr\} ,
  \end{equation}
  for some countable sets $\mathcal{B}_j$ such that the generators $H^k$ and every element of every $\mathcal{B}_j$ are bracket-regular. Bracket-regular operators form a linear space and every element of a bracket-generated family is bracket-regular.
\end{definition}

The polynomial weights in~(i) are all that the brackets of either model supply, and it is the decay of $\sigma_k$ that absorbs them. Only an inclusion is asked of $\mathcal{B}_j$ in~\eqref{eq:reachable_chain}, so a particular bracket chain is free to take fewer brackets than are available to it, and in particular cannot take differential order that exceeds what bracket-regularity permits. This is what happens for the passive scalar, and it is the reason the family we use there is not a Lie algebra.

\begin{lemma}\label{lem:bracket_extraction}
  In the same setting of Lemma~\ref{lem:vanishing_implies_orthogonality}, and with $\eta_s$ given by \eqref{eq:eta_def}, let $A$ be bracket-regular and suppose that $ \langle \eta_s, \vf{A}\pi_s \rangle = 0$ for all $s \in (0,T]$, $\PP$-almost surely. Then we have:
  \begin{equation}
    \langle \eta_s, \vf{[G_{w_s}, A]}\pi_s\rangle = 0 , \qquad \langle \eta_s, \vf{[H^k, A]}\pi_s \rangle = 0 \qquad \text{for all } s \in (0,T] , \quad  \PP\text{-almost surely.}
  \end{equation}
\end{lemma}

\begin{proof}
  Using $J^{\pi\pi}_{u,u} = \mathrm{Id}$, the fiber equation $\partial_r \pi_r = \vf{G}_{w_r}\pi_r$ and the backward equation $\partial_r J^{\pi\pi}_{r,u} = -J^{\pi\pi}_{r,u} L^{\pi\pi}_{z_r}$ with $L^{\pi\pi}_{z_r} = D_\pi \vf{G}_{w_r}(\pi_r)$ from~\eqref{eq:fiber_operator}, the Leibniz rule gives
  \begin{equation}
    \partial_r \bigl( J^{\pi\pi}_{r,u} \vf{A}\pi_r \bigr)
    = J^{\pi\pi}_{r,u} \bigl( D_{\pi}\vf{A}\cdot \vf{G}_{w_r} - D_{\pi}\vf{G}_{w_r}\cdot \vf{A}\bigr)(\pi_r)
    = J^{\pi\pi}_{r,u} \, [\vf{G}_{w_r}, \vf{A}]\pi_r ,
  \end{equation}
  where the last equality is the definition of Lie bracket for vector fields $[\vf{X},\vf{Y}] = D\vf{Y}\cdot\vf{X} - D\vf{X}\cdot\vf{Y}$. Now by Proposition~\ref{prop:anti_homomorphism}, and integrating from $s$ to $u$, $[\vf{G}_{w_r}, \vf{A}] = -\vf{[G_{w_r}, A]}$, yields:
  \begin{equation}\label{eq:corrected_increment}
    \vf{A}\pi_{u} - J^{\pi\pi}_{s,u} \vf{A}\pi_s
    = \int_s^u J^{\pi\pi}_{r,u} \, [\vf{G}_{w_r}, \vf{A}]\pi_r \, \ud r
    = -\int_s^u J^{\pi\pi}_{r,u} \, \vf{[G_{w_r}, A]}\pi_r \, \ud r ,
    \qquad 0 < s \le u \le T .
  \end{equation}
  Testing both sides against $\eta_u$ makes the left side vanish, since $\langle \eta_u, J^{\pi\pi}_{s,u}\vf{A}\pi_s\rangle = \langle \eta_s, \vf{A}\pi_s\rangle$ by~\eqref{eq:eta_def} and both pairings vanish by hypothesis. We therefore conclude that
  \begin{equation}\label{eq:drift_vanishes_inline}
    \int_s^u \langle \eta_r, \vf{[G_{w_r}, A]}\pi_r\rangle \, \ud r = 0
    \qquad \forall 0 < s \le u \le T .
  \end{equation}
  Differentiating this quantity in $s$, which is allowed by our assumption on
  the commutator and by the continuity in $r$ of $\eta_r$ and $\pi_r$
  established below, delivers as desired
  $\langle \eta_s, \vf{[G_{w_s}, A]}\pi_s\rangle = 0$ for all $s \in (0, T]$.

  To conclude the second identity we use a (rough) Taylor expansion. Starting from \eqref{eq:corrected_increment} tested against $\eta_u$, and by expanding
  \begin{equation}
    w_r - w_s = \int_s^r b(w_q)\,\ud q + \sum_{k \in \Zstar} \sigma_k e_k \,\bigl( W_k(r) - W_k(s)\bigr) ,
    \qquad b(w) = \Delta w - B(w,w) ,
  \end{equation}
  we arrive at
  \begin{equation}\label{eq:corrected_expansion}
    \begin{aligned}
      0 & = \int_s^u \langle \eta_r, \vf{[G_{w_r}, A]}\pi_r\rangle  \ud r                                                                                                                                     \\
      & = (u - s)\langle \eta_s, \vf{[G_{w_s}, A]}\pi_s\rangle - \sum_{k \in \Zstar} \sigma_k \, \langle \eta_s,  \vf{[H^k, A]}\pi_s \rangle \, \int_s^u \bigl( W_k(r) - W_k(s)\bigr) \ud r
      + \mathcal E_{s,u}                                                                                                                                                                        \\
      & = -\sum_{k \in \Zstar} \sigma_k \, \langle \eta_s,  \vf{[H^k, A]}\pi_s \rangle \, \int_s^u \bigl( W_k(r) - W_k(s)\bigr) \ud r
      + \mathcal E_{s,u}.
    \end{aligned}
  \end{equation}
  where the last step uses the identity just proven and introduces the error term $\mE_{s, u}$. It is convenient to differentiate in $u$, which gives
  \begin{equation}\label{eq:diff_in_u}
    \sum_{k \in \Zstar} \sigma_k \, \langle \eta_s,  \vf{[H^k, A]}\pi_s \rangle \, \bigl( W_k(u) - W_k(s)\bigr)
    = \partial_u \mathcal E_{s,u} .
  \end{equation}
  We claim that $|\partial_u \mE_{s,u}| \lesssim |s-u|$. By~\eqref{eq:corrected_expansion} we may write $\mE_{s,u} = \int_s^u \rho(s,r)\,\ud r$ with
  \begin{equation*}
    \rho(s,r) \eqdef \langle \eta_r, \vf{[G_{w_r}, A]}\pi_r\rangle
    + \sum_{k \in \Zstar} \sigma_k \bigl( W_k(r) - W_k(s)\bigr)\langle \eta_s, \vf{[H^k, A]}\pi_s\rangle,
  \end{equation*}
  since $ \langle \eta_r, \vf{[G_{w_r}, A]}\pi_r\rangle= 0$. With $d(s,r) \eqdef \int_s^r b(w_u) \ud u$, the expansion of $w_r - w_s$ gives
  \begin{equation*}
    [G_{w_r}, A] = [G_{w_s}, A] - \sum_k \sigma_k(W_k(r)-W_k(s))[H^k, A] - [L^0_{d(s,r)}, A],
  \end{equation*}
  so that
  \begin{equation}\label{eq:rho_expansion}
    \begin{aligned}
      \rho(s,r) = & \langle \eta_r, \vf{[G_{w_s}, A]}\pi_r\rangle
      - \langle \eta_r, \vf{[L^0_{d(s,r)}, A]}\pi_r\rangle         \\
      & - \sum_k \sigma_k \bigl( W_k(r) - W_k(s)\bigr)
      \bigl( \langle \eta_r, \vf{[H^k, A]}\pi_r\rangle - \langle \eta_s, \vf{[H^k, A]}\pi_s\rangle\bigr).
    \end{aligned}
  \end{equation}
  Now Schauder theory guarantees that $[s, T] \ni r \mapsto \eta_r$ is Lipschitz in $H^{-2}$ (because the terminal data at time $T$ lies in $L^2$) and $[s,T]\ni r\mapsto \pi_r$ is Lipschitz in $H^{\beta_\star - 2}$. Therefore, we estimate the first term in \eqref{eq:rho_expansion} as
  \begin{equation*}
    \begin{aligned}
      |\langle \eta_r, \vf{[G_{w_s}, A]}\pi_r\rangle| & \leq     |\langle \eta_r - \eta_s, \vf{[G_{w_s}, A]}\pi_r\rangle| +     |\langle \eta_s, \vf{[G_{w_s}, A]}(\pi_r- \pi_s) \rangle| \lesssim_s |r -s|,
    \end{aligned}
  \end{equation*}
  by our assumption on $\vf{[G_{w_s}, A]}$.

  Similarly for the second term, condition~(iii) of Definition~\ref{def:bracket_regular} and the definition of $d$ give
  \begin{equation*}
    |\langle \eta_r, \vf{[L^0_{d(s,r)}, A]}\pi_r\rangle|
    \lesssim \| d(s, r)\|_{H^{s_\star-2}}
    \leq |s-r| \sup_{u \in [s, T]} \| b(w_u) \|_{H^{s_\star-2}}
    \lesssim_s |s-r|
  \end{equation*}
  since $w_u$ takes values in $H^{s_\star}$ uniformly in $[s, T]$ (but the estimate degenerates at $s=0$ because $w_0 \in L^2$) and $H^{s_\star -2}$ is an algebra. The last term in \eqref{eq:rho_expansion} pairs the noise against the bracket increments
  \[
    \Psi_k \eqdef \langle \eta_r, \vf{[H^k, A]}\pi_r\rangle - \langle \eta_s, \vf{[H^k, A]}\pi_s\rangle.
  \]
  By Assumption~\ref{ass:noise} and $\alpha_\star > 10$, $\mW \eqdef \sum_{k \in \Zstar} \sigma_k W_k e_k$ is a Wiener process with values in $H^6$. Hence by Cauchy--Schwarz
  \[
    \Bigl| \sum_{k \in \Zstar} \sigma_k \bigl( W_k(r) - W_k(s)\bigr) \Psi_k \Bigr|
    \leq \bigl\| \mW(r) - \mW(s) \bigr\|_{H^6}
    \Bigl( \sum_{k \in \Zstar} |k|^{-12} |\Psi_k|^2 \Bigr)^{1/2}
    \lesssim_s |r-s|,
  \]
  where for the last inequality we used a similar splitting as above with $[H^k, A]$ in place of $[G_{w_s}, A]$, so that
  \[
    |\Psi_k| \lesssim_s |r-s| \, \| [H^k, A] \|_{\mL(H^{\beta_\star-2};\, H^2)} \lesssim_s |r-s| \, |k|^4
  \]
  by condition~(i-b) of Definition~\ref{def:bracket_regular}, so that $\sum_k |k|^{-12}|\Psi_k|^2 \lesssim_s |r-s|^2 \sum_k |k|^{-4}$
  Altogether $|\rho(s,r)| \lesssim_s |r-s|$ with a random but almost surely finite constant, which is the claimed bound on $\partial_u \mE_{s,u}$.

  Now fix a rational $s \in (0,T)$ and set $c_k \eqdef \sigma_k \langle \eta_s, \vf{[H^k, A]}\pi_s\rangle$, so that the differentiated identity \eqref{eq:diff_in_u} reads
  \[
    \sum_{k \in \Zstar} c_k \bigl( W_k(u) - W_k(s)\bigr) = \partial_u\mE_{s,u}.
  \]
  Since $\sigma_k \neq 0$, it is enough to show that $c_k = 0$ for every $k \in \Zstar$.

  The assumption $\sup_k |k|^{-2} \|[H^k, A]\|_{\Opbeta} < \infty$ together with Assumption~\ref{ass:noise} gives
  \begin{equation}\label{eq:coeff_decay}
    |c_k| \lesssim |k|^{2-\alpha_\star} , \qquad \sum_{k \in \Zstar} |c_k|^2 < \infty .
  \end{equation}
  By reality $W_{-k} = \overline{W_k}$ and $c_{-k} = \overline{c_k}$, so we can split the sum into real and imaginary parts and pair conjugate indices to write it over the real index set $\mathcal K \eqdef \Zstarp \times \{\Re, \Im\}$,
  \begin{equation*}
    \sum_{k \in \Zstar} c_k \bigl( W_k(u) - W_k(s)\bigr) = \sum_{\kappa \in \mathcal K} c_\kappa\, \zeta_\kappa(s,u) ,
    \qquad
    \zeta_\kappa(s,u) \eqdef \beta_\kappa(u) - \beta_\kappa(s) ,
  \end{equation*}
  with $\{\beta_\kappa\}_{\kappa \in \mathcal K}$ independent standard real Brownian motions and $c_{(k,\Re)} = \sqrt 2\, \Re c_k$, $c_{(k,\Im)} = -\sqrt 2\, \Im c_k$ of the same size as the $c_k$, so that $c = 0$ forces every $c_k$ to vanish. Each $\zeta_\kappa$ is a Brownian increment with $\EE[\zeta_\kappa(s,u)^2] = \vr(u-s)$ for some $\vr > 0$. Then if we set for all $0 < h < e^{-1}$
  \begin{equation*}
    \phi(h) \eqdef \sqrt{2\vr\, h \log\log(1/h)},
  \end{equation*}
  we obtain from the differentiated identity and $|\partial_u\mE_{s,u}| \lesssim |s-u|$ that
  \begin{equation}\label{eq:scale_gap}
    \lim_{u \downarrow s} \frac{1}{\phi(u-s)} \sum_{\kappa \in \mathcal K} c_\kappa\, \zeta_\kappa(s,u) = 0 ,
  \end{equation}
  since $(u-s)/\phi(u-s) \to 0$.

  If $c = (c_\kappa)_{\kappa \in \mathcal K}$ were deterministic, the process $u \mapsto \sum_\kappa c_\kappa\, \zeta_\kappa(s,u)$ would be a one-dimensional Brownian motion with variance rate $\vr\|c\|_{\ell^2}^2$, and the law of the iterated logarithm gives that as $h \downarrow 0$, the limit points of
    \begin{equation*}
      \beta(h)\big/\sqrt{2h\log\log(1/h)}
    \end{equation*}
    are $[-1,1]$, which would force $c = 0$ in view of~\eqref{eq:scale_gap}. But $c$ is random, and through $\eta_s$ it depends on the increments $\zeta_\kappa(s,u)$, so we need this conclusion on the full vector of increments $(\zeta_{\kappa}(s,u))_{\kappa \in \mathcal K}$. This particular version is provided by Strassen's functional law of the iterated logarithm, which states that almost surely, the family
    \begin{equation*}
      \bigl(\zeta_\kappa(s,u)\bigr)_{\kappa \in \mathcal K} \big/ \phi(u-s) ,
      \qquad u \in (s, s+e^{-1}/2) ,
    \end{equation*}
    is relatively compact, with set of limit points as $u \downarrow s$ equal to the closed unit ball
    \begin{equation*}
      \overline \mB \eqdef \bigl\{ h \in \ell^2(\mathcal K) : \|h\|_{\ell^2(\mathcal K)} \le 1\bigr\} .
    \end{equation*}
    We use the Banach-space form of the theorem \cite[Theorem~8.5]{LedouxTalagrand}, transferred to small times by the time inversion $\beta_\kappa(t) \mapsto t\beta_\kappa(1/t)$. The vector of increments $(\zeta_\kappa(s,u))_{\kappa \in \mathcal K}$ does not lie in $\ell^2$, so the compactness takes place in the larger space $B \eqdef \ell^2(\mathcal K; |\kappa|^{-4})$, while the limit set is nonetheless the unit ball of $\ell^2$. Pairing with $c$ is continuous on $B$, because
    \begin{equation*}
      \sum_{\kappa \in \mathcal K} |\kappa|^{4} c_\kappa^2
      \ \lesssim\ \sum_{\kappa \in \mathcal K} |\kappa|^{8 - 2\alpha_\star}
      \ <\ \infty
    \end{equation*}
    by~\eqref{eq:coeff_decay}. As a consequence
    $\phi(u-s)^{-1}\sum_\kappa c_\kappa\, \zeta_\kappa(s,u)$ has set of limit points as $u \downarrow s$ equal to the closed interval
    \begin{equation*}
      \bigl\{ \langle h, c\rangle : h \in \overline{\mB} \bigr\}
      = \bigl[ -\|c\|_{\ell^2(\mathcal K)}, \|c\|_{\ell^2(\mathcal K)} \bigr] ,
    \end{equation*}
  and by~\eqref{eq:scale_gap} this set is $\{0\}$. Hence $c = 0$.

  Finally, for each fixed $k$ the map $s \mapsto \langle \eta_s, \vf{[H^k,
  A]}\pi_s\rangle$ is continuous on $(0,T]$, since $[H^k, A]$ is a fixed element
  of $\Opbeta$ and $s \mapsto \eta_s, \pi_s$ are continuous with values in $\mH$
  and $H^{\beta_\star}$ respectively. The claim therefore extends from rational
  $s$ to every $s \in (0,T]$, almost surely.

\end{proof}

We deduce the following corollary.
\begin{corollary}\label{cor:propagation}
  In the setting of Lemma~\ref{lem:bracket_extraction}, let $A$ be bracket-regular and $\langle \eta_s, \vf{A}\pi_s \rangle = 0$ for all $s \in (0,T]$, almost surely. Then, almost surely for all $s \in (0,T]$,
  \begin{equation}\label{eq:propagation}
    \langle \eta_s, \vf{[H^k, A]}\pi_s \rangle = 0
    \qquad \forall k \in \Zstar ,
  \end{equation}
  and, provided $[\Delta, A] \in \Opbeta$, also $\langle \eta_s, \vf{[\Delta, A]}\pi_s \rangle = 0$.
\end{corollary}

\begin{proof}
  The first identity is part of the conclusion of Lemma~\ref{lem:bracket_extraction}. For the second, observe that $B \mapsto \langle \eta_s, \vf{B}\pi_s\rangle$ is a continuous linear functional on $\Opbeta$, since $\eta_s \in \mH$ and $\pi_s \in H^{\beta_\star}$ with $\|\pi_s\| = 1$, we have
  \begin{equation*}
    \bigl| \langle \eta_s, \vf{B}\pi_s \rangle \bigr|
    \leq 2 \|\eta_s\| \, \|\pi_s\|_{H^{\beta_\star}} \, \|B\|_{\Opbeta} .
  \end{equation*}
  The series in~\eqref{eq:generator_modes} converges absolutely in the norm of $\mathcal L(H^{\beta_\star}; \mH)$, which is all the functional sees. Condition~(i-a) of Definition~\ref{def:bracket_regular} gives $\|[H^l, A]\|_{\Opbeta} \lesssim_A |l|^2$, so that
  \[
    \sum_l |\hat w_s(l)| \, \|[H^l, A]\|_{\Opbeta} \lesssim_A \|w_s\|_{H^{s_\star-2}}
  \]
  by Cauchy--Schwarz, using $s_\star - 2 > 3$. We may therefore evaluate the functional term by term, and each term vanishes by~\eqref{eq:propagation}. Since the left side of~\eqref{eq:generator_modes} is annihilated as well, so is $[\Delta, A]$.
\end{proof}

Orthogonality therefore propagates along the two operations $A \mapsto [H^k, A]$
and $A \mapsto [\Delta, A]$, for as long as the result stays bracket-regular.

\begin{proposition}\label{prop:harvest}
  Fix $T > 0$ and $z_0 \in \mX$, and let $\eta = (\delta w, \delta\pi) \in T_{z_T}\mX$ satisfy $\langle M_T(z_0)\eta, \eta\rangle = 0$, with $\eta_s$ as in~\eqref{eq:eta_def}. Then $\delta w = 0$ and, for every bracket-generated $\mf{g} \subseteq \Opbeta$,
  \begin{equation}\label{eq:harvest}
    \langle \eta_s, \vf{A}\pi_s \rangle = 0 ,
    \qquad \forall s \in (0,T] , \ A \in \mf{g} ,
  \end{equation}
  almost surely, with an exceptional null set that depends only on $T$, $z_0$, and
  $\mf{g}$,  but
  not on $\eta$.
\end{proposition}

\begin{proof}
  That $\delta w = 0$ is Lemma~\ref{lem:vanishing_implies_orthogonality}(i). The
  rest  follows by induction from
  Definition~\ref{def:bracket_regular} and Corollary~\ref{cor:propagation}.
\end{proof}

\subsection{Proof of the non-degeneracy for fixed initial data}\label{sec:proof-as-nondegeneracy}

We are now ready to prove the first main result of the section.

\begin{proof}[Proof of Proposition~\ref{prop:as_nondegeneracy}]
  The proof is by contradiction. Assume that the set
  \begin{equation}
    L \eqdef \left\{ \omega \in \Omega : \inf_{\eta \in S_{\alpha, N, T}} \langle M_T(z_0) \eta, \eta \rangle = 0 \right\}
  \end{equation}
  has $\PP(L) > 0$. Note that the fact that $L$ is measurable requires some
  care since the slice over which we are minimizing is random (depending on the
  tangent space): this issue can be solved by extending the quadratic form to
  the identity on the orthogonal of $T_{\pi_T} \mathbb{S}$ in $\mH$.

  Fix $\omega \in L$ and let us show that the infimum is attained. Fix a
  minimizing sequence $\eta_k \in S_{\alpha,N,T}$, so that $\langle
  M_T(z_0)\eta_k, \eta_k\rangle \to 0$, and consider the closed cone
  \begin{equation}\label{eq:cone_unsliced_as}
    \Lambda \eqdef \bigl\{ \eta = (\delta w, \delta\pi) \in \mH \times \mH :
    \| P_N \eta \| \geq \alpha \|\eta\| , \ \langle \delta\pi, \pi_T \rangle = 0 \bigr\} .
  \end{equation}
  Since $\Lambda$ is weakly sequentially closed and $P_N$ has finite rank, we
  have up to taking a subsequence $\eta_k \rightharpoonup \eta_*$ and $P_N\eta_k \to P_N\eta_*$
  strongly. Moreover
  \begin{equation*}
    \|\eta_*\| \le \liminf_{k \to \infty} \|\eta_k\| \le 1 ,
    \qquad
    \|P_N\eta_*\| = \lim_{k \to \infty} \|P_N \eta_k\| \geq \alpha > 0 ,
  \end{equation*}
  so that $\eta_* \neq 0$ and one can easily check $\eta_* \in T_{z_T}\mX$.
  For each $s \in [0,T]$ the map $\eta \mapsto Q^{1/2}\Pi_w (J_{s,T})^*\eta$
  is bounded, hence weakly continuous by
  Lemma~\ref{lem:fibre_dissipation}(iii). Writing $\hat\eta \eqdef
  \eta_*/\|\eta_*\|$ we have $\hat\eta \in S_{\alpha,N,T}$, and weak lower
  semicontinuity at each $s$ together with Fatou's lemma imply
  \begin{equation*}
    0 \le \|\eta_*\|^2 \langle M_T(z_0)\hat\eta, \hat\eta\rangle
    = \int_0^T \bigl\| Q^{1/2}\Pi_w (J_{s,T})^*\eta_* \bigr\|^2 \ud s
    \le \liminf_{k \to \infty} \langle M_T(z_0)\eta_k, \eta_k\rangle = 0 .
  \end{equation*}
  Hence for $\omega \in L$ we may fix a (random) $\eta = (\delta w, \delta\pi)
  \in T_{z_T}\mX$ with $\|\eta\| = 1$, $\|P_N\eta\| \geq \alpha$ and $\langle
  M_T(z_0)\eta,\eta\rangle = 0$. Moreover, since the exceptional null set of
  Proposition~\ref{prop:harvest} does not depend on the tangent vector, that
  proposition may be applied to this $\eta$.

  Now consider the two bracket-generated vector fields $\mf{g}_{\mathrm{LNS}}$
  for LNS~\eqref{eq:two_operators_intro_NS} and $\mf{g}_{\mathrm{PSA}}$ for
  PSA~\eqref{eq:two_operators_intro_PS}, both defined
  in~\eqref{eq:lie_algebras}. We write simply $\mf{g}$ when we do not need to define the
  model. Then by Proposition~\ref{prop:harvest} $\delta w = 0$ and
  \begin{equation}\label{eq:claim-alg}
    \langle \eta_s , \vf{A}\pi_s \rangle = 0 ,
    \qquad \forall s \in (0, T] , \ A \in \mf{g} ,
  \end{equation}
  where $\eta_s = (J^{\pi\pi}_{s,T})^* \delta\pi$ as in~\eqref{eq:eta_def}. By~\eqref{eq:claim-alg} the algebraic spanning of Theorem~\ref{thm:transitivity} gives $\delta \pi =0$ immediately in the case of LNS.
  We obtain the same conclusion in the case of PSA, but here we must take a bit more care. Consider a sequence of times $s_k < T, \lim_{k \to \infty }s_k = T$. Then for every $k$ there exists a set of full measure $\Omega_k$ on which
  \begin{equation*}
    \langle \eta_{s_k} , \vf{A} \pi_{s_k} \rangle = 0 , \quad \forall A \in \mf{g}_{\mathrm{PSA}} \qquad \Rightarrow \qquad \eta_{s_k} =0,
  \end{equation*}
  since for any $k \in \NN$ we have $\eta_{s_k} \in H^5$ by
  Lemma~\ref{lem:adjoint_smoothing} below, so it satisfies the regularity requirement
  of the theorem. In addition we know that $\eta_{s_k} \to \delta \pi$ as $k \to
  \infty$ weakly in $\mH$, again by Lemma~\ref{lem:adjoint_smoothing}. This
  implies that on the set $\cap_{k \in \NN} \Omega_k$, which has still full measure,
  we have $\delta \pi =0$.

  In both cases $\delta \pi = 0$ is a contradiction to the statement $\| P_N \eta \| \geq \alpha$, so the proof is complete.
\end{proof}

Before we conclude the section, we record a small parabolic smoothing result
that is used in the proof above. 

  \begin{lemma}\label{lem:adjoint_smoothing}
    Fix $T > 0$ and $z_0 \in \mX$. Almost surely, for every $0 < s < T$ the adjoint $(J^{\pi\pi}_{s,T})^*$ maps $\pi_T^\perp$ boundedly into $H^5 \cap \pi_s^\perp$. Moreover, for every $\delta\pi \in \pi_T^\perp$, writing $\eta_s = (J^{\pi\pi}_{s,T})^* \delta\pi$, we have $\eta_s \rightharpoonup \delta\pi$ weakly in $\mH$ as $s \uparrow T$.
  \end{lemma}
We state the result without complete proof, because it
is the consequence of usual Schauder estimates. The map $r \mapsto \eta_r$ solves the backward equation
    \begin{equation}\label{eq:adjoint_backward}
      \partial_r \eta_r = -\bigl(L^{\pi\pi}_{z_r}\bigr)^* \eta_r ,
      \qquad \eta_T = \delta\pi ,
    \end{equation}
which in the reversed time $\tau = T-r$, with $\tilde\eta_\tau \eqdef \eta_{T-\tau}$,
is the forward equation
    \begin{equation}\label{eq:adjoint_reversed}
      \partial_\tau \tilde\eta_\tau = \bigl(L^{\pi\pi}_{z_{T-\tau}}\bigr)^* \tilde\eta_\tau ,
      \qquad \tilde\eta_0 = \delta\pi .
    \end{equation}
Here we note that the adjoint operator is given by
    \begin{equation}\label{eq:fiber_operator_adjoint}
      \bigl(L^{\pi\pi}_z\bigr)^*\eta
      = G_w^*\eta
      - \langle \pi, \eta \rangle \bigl( G_w^*\pi + G_w\pi \bigr)
      - \langle \pi, G_w\pi \rangle \eta ,
    \end{equation}
with $G_w^* = \Delta - (L^0_w)^*$ and $B(w,\cdot)^* = -B(w,\cdot)$
because $u = K*w$ is divergence free, while for the stretching term
of~\eqref{eq:two_operators_intro_NS} the identity $K = \nabla^\perp(-\Delta)^{-1}$ and one
integration by parts give
    \begin{equation}\label{eq:stretching_adjoint}
      \bigl\langle B(\zeta, w), \eta \bigr\rangle
      = \bigl\langle (-\Delta)^{-1}\zeta, \curl (\eta \nabla w) \bigr\rangle ,
      \qquad\text{so that}\qquad
      B(\cdot, w)^* \eta = (-\Delta)^{-1} \curl (\eta \nabla w) .
    \end{equation}
Collecting the terms, \eqref{eq:adjoint_reversed} reads
    \begin{equation}\label{eq:adjoint_reversed_explicit}
      \partial_\tau \tilde\eta
      = \Delta \tilde\eta
      + u \cdot \nabla \tilde\eta
      - (-\Delta)^{-1}\curl\bigl( \tilde\eta \nabla w \bigr)
      - \langle \pi, \tilde\eta \rangle \bigl( G_w^*\pi + G_w\pi \bigr)
      - \langle \pi, G_w\pi \rangle \tilde\eta ,
    \end{equation}
  which is a perturbation of the heat equation and standard tools apply.

\subsection{Uniform tail estimate via compactness}
\label{sec:uniform_tail}

The almost sure non-degeneracy established in the preceding subsection is a
pointwise result for each initial condition $z_0$. Here we boost it to a uniform
estimate over initial data in  $\mC_R$,
following~\cite{DongPeng24}*{Appendix~B}, where the argument is attributed to~\cite{PZZ24}*{Proposition~3.5}.

\begin{lemma}\label{lem:base_difference}
  Let $w^1, w^2 \in C([0,T]; \mH) \cap L^2([0,T]; H^1)$ be such that $\tilde w \eqdef
  w^1 - w^2$ solves weakly
  \begin{equation}\label{eq:base_diff_eq}
    \partial_t \tilde w = \Delta \tilde w - B(\tilde w, w^1) - B(w^2, \tilde w) + g ,
    \qquad g \in L^2\big([0,T]; H^{-1}\big) .
  \end{equation}
  Then
  \begin{equation}\label{eq:base_diff_energy}
    \frac{\ud}{\ud t} \|\tilde w_t\|^2 + \|\nabla \tilde w_t\|^2
    \ \le\ C_0 \|w^1_t\| \|w^1_t\|_{H^1} \|\tilde w_t\|^2 + 8 \|g_t\|_{H^{-1}}^2
  \end{equation}
  and, with $\Lambda_T \eqdef C_0 \int_0^T \|w^1_r\| \|w^1_r\|_{H^1} \ud r$,
  \begin{equation}\label{eq:base_diff_gronwall}
    \sup_{t \le T} \|\tilde w_t\|^2 + \int_0^T \|\nabla \tilde w_r\|^2 \ud r
    \ \le\ 2 e^{\Lambda_T} \Big( \|\tilde w_0\|^2 + 8 \int_0^T \|g_r\|_{H^{-1}}^2 \ud r \Big) .
  \end{equation}
\end{lemma}
The proof is a standard energy estimate for 2D vorticity and we omit it. The projective process satisfies a similar estimate.

\begin{lemma}\label{lem:fibre_difference}
  Let $w^1, w^2 \in C([0,T]; \mH) \cap L^2([0,T]; H^1)$ be two base paths and let $\pi^1, \pi^2 \in C([0,T]; \Sph) \cap L^2([0,T]; H^1)$ solve the projective equation~\eqref{eq:proj-spde} over $w^1$ and over $w^2$ respectively. Then:
  \begin{enumerate}[label=(\roman*)]
    \item For $i \in \{1,2\}$ the Dirichlet quotient of $\pi^i$ obeys
      \begin{equation}\label{eq:fd_dirichlet}
        \frac{\ud}{\ud t} \|\pi^i_t\|_{H^1}^2
        \ \le\ C \|w^i_t\|_{H^1}^2  \|\pi^i_t\|_{H^1}^2   ,
      \end{equation}
      and consequently
      \begin{equation}\label{eq:fd_h1}
        \sup_{t \le T} \|\pi^i_t\|_{H^1}^2
        \ \le\  \|\pi^i_0\|_{H^1}^2 
        \exp\Bigl( C \int_0^T \|w^i_r\|_{H^1}^2 \ud r \Bigr) .
      \end{equation}
    \item Write $\tilde w \eqdef w^1 - w^2$ and $\tilde\pi \eqdef \pi^1 - \pi^2$. Then
      \begin{equation}\label{eq:fd_energy}
        \frac{\ud}{\ud t} \|\tilde\pi_t\|^2 + \|\nabla \tilde\pi_t\|^2
        \ \le\ C\bigl( \Gamma^1_t + \Gamma^2_t \bigr) \|\tilde\pi_t\|^2
        + C \|\pi^2_t\|_{H^1}^2 \|\tilde w_t\|^2 + C \|\tilde w_t\|_{H^1}^2 ,
      \end{equation}
      with $\Gamma^i$ the quantity~\eqref{eq:gamma_r} computed on $(w^i, \pi^i)$, and, with $A_T \eqdef C\int_0^T \bigl( \Gamma^1_r + \Gamma^2_r \bigr) \ud r$,
      \begin{equation}\label{eq:fd_gronwall}
        \sup_{t \le T} \|\tilde\pi_t\|^2 + \int_0^T \|\nabla \tilde\pi_r\|^2 \ud r
        \ \le\ 2 e^{A_T} \Bigl( \|\tilde\pi_0\|^2
          + C \int_0^T \bigl( \|\pi^2_r\|_{H^1}^2 \|\tilde w_r\|^2
        + \|\tilde w_r\|_{H^1}^2 \bigr) \ud r \Bigr) .
      \end{equation}
  \end{enumerate}
\end{lemma}
Part~(i) is the pathwise fiber Gr\"onwall of \cite{HPRY24}*{Lemma~6.4}, and
part~(ii) is a standard energy estimate as in Lemma~\ref{lem:base_difference}.
Next we study the continuity of the Jacobian in the initial data.

\begin{lemma}\label{lem:continuity_in_data}
  Let $\{z_0^{(n)}\}_{n \geq 1} \subset \mC_R$ be a sequence of initial
  conditions, and suppose $z_0^{(n)} \to z_0^{(0)}$ in $\mathcal{X}$. Denote by
  $z^{(n)}_t = (w_t^{(n)}, \pi_t^{(n)})$ the solution of the joint system
  starting from $z_0^{(n)}$, and by $J^{(n)}_{s,t}$ the associated Jacobian.
  Then, for any $0 \leq s \leq t < \infty$ we have that $z_t^{(n)} \to
  z_t^{(0)}$ in $\mX$ and, for any $\eta = (\varphi, 0)$ with $\varphi \in \mH$,
  $J_{s,t}^{(n)} \eta \to J_{s,t}^{(0)} \eta$ in $\mH \times \mH$.
\end{lemma}

\begin{proof}
  Throughout we fix a horizon $T>0$, work on $[0,T]$ and argue $\PP$-almost
  surely. The constants below may depend on $\omega$, $R$ and $T$ but never on
  $n$. Write $v^{(n)}\eqdef w^{(n)}-w^{(0)}$ and
  $\rho^{(n)}\eqdef\pi^{(n)}-\pi^{(0)}$ for the base and fiber differences.
  We also write
  $\Gamma^{(n)}_t\eqdef\|w_t^{(n)}\|_{H^1}^2+\|\pi_t^{(n)}\|_{H^1}^2$ for the
  weight~\eqref{eq:gamma_r} along the $n$-th trajectory, and we note that
  $\Gamma^{(n)}_t\ge\|\pi^{(n)}_t\|^2=1$. Since the embeddings $H^{s_\star}\hookrightarrow\mH$ and
  $H^1\hookrightarrow\mH$ are continuous, the hypothesis $z_0^{(n)}\to
  z_0^{(0)}$ in $\mX$ gives in particular
  $\|w_0^{(n)}-w_0^{(0)}\|\to0$ and $
  \|\pi_0^{(n)}-\pi_0^{(0)}\|\to0 $.
  Because the data lie in $\mC_R$, so that $\|w_0^{(n)}\|\le R$ and $\|\pi_0^{(n)}\|_{H^1}\le R$, there exists a finite random constant $M=M(\omega,R,T)$, independent of $n$, such that
  \begin{equation}\label{eq:cd_uniform}
    \begin{aligned}
      \sup_{n\ge0}\Big(\sup_{t\in[0,T]}\|w_t^{(n)}\|^2
      & +\int_0^T\Gamma^{(n)}_t\ud t\Big)\le M.
    \end{aligned}
  \end{equation}
  For $w$ this follows from an energy estimate and $\|w_0^{(n)}\| \le R$. For $\pi$
  from the bound on $w$ and $\|\pi_0^{(n)}\|_{H^1} \le R$.

  The convergence $z^{(n)}_t \to z_t^{(0)}$ follows from the continuous dependence
  of the two-dimensional vorticity equation on its
  data~\cite{ConstantinFoias88}, \cite{KuksinShirikyan12}*{Chapter~2}, in the
  form of Lemmas~\ref{lem:base_difference} and~\ref{lem:fibre_difference}(ii).
  So $w_t^{(n)}\to w_t^{(0)}$ and $\pi_t^{(n)}\to\pi_t^{(0)}$ in $\mH$,
  uniformly on $[0,T]$, which is the first assertion, and both differences also tend
  to zero in $L^2([0,T];H^1)$.

  Finally, we prove the convergence of the Jacobian. Fix a base vector $\eta =
  (\varphi, 0)$, $\varphi \in \mH$, and set $\eta_t\eqdef J_{s,t}^{(0)}\eta$,
  the solution of the limiting linear system $\partial_t\eta_t=L_{z_t^{(0)}}\eta_t$
  on $[s,T]$ with $\eta_s=\eta$. The difference $\zeta_t^{(n)} =
  J_{s,t}^{(n)}\eta-J_{s,t}^{(0)}\eta$ solves
  \begin{equation}\label{eq:cd_jdiff}
    \partial_t\zeta^{(n)}=L_{z_t^{(n)}}\zeta^{(n)}+\big(L_{z_t^{(n)}}-L_{z_t^{(0)}}\big)\eta_t ,
    \qquad\zeta_s^{(n)}=0 .
  \end{equation}
  Therefore, we find
  \begin{equation*}
    \frac12\frac{\ud}{\ud t}\|\zeta_t^{(n)}\|^2
    = \langle L_{z_t^{(n)}}\zeta_t^{(n)},\zeta_t^{(n)}\rangle
    + \big\langle \big(L_{z_t^{(n)}}-L_{z_t^{(0)}}\big)\eta_t,\zeta_t^{(n)}\big\rangle .
  \end{equation*}
  We estimate the two terms on the right separately. For the first term we
  must bound $\langle L_{z_t^{(n)}}\zeta_t^{(n)},\zeta_t^{(n)}\rangle$, and we
  split $\zeta_t^{(n)}=(\zeta^{w},\zeta^{\pi})$ into base and fiber components.
  The pairing then has three blocks, which we estimate one by one.

  The base block $L^{ww}_{w^{(n)}}\zeta^w = \Delta\zeta^w - B(w^{(n)},\zeta^w) -
  B(\zeta^w,w^{(n)})$ satisfies
  \begin{equation*}
    \langle L^{ww}_{w^{(n)}}\zeta^{w},\zeta^{w}\rangle
    \le - \frac14\|\nabla\zeta^{w}\|^2  + C\,\Gamma^{(n)}_t\,\|\zeta^{w}\|^2 .
  \end{equation*}
  The fiber block is estimated by Lemma~\ref{lem:fibre_dissipation}(ii) below:
  \begin{equation*}
    \bigl\langle L^{\pi\pi}_{z_t^{(n)}}\zeta^{\pi},\zeta^{\pi}\bigr\rangle
    \le -\frac14\|\nabla\zeta^{\pi}\|^2
    + C\,\Gamma^{(n)}_t\,\|\zeta^{\pi}\|^2 .
  \end{equation*}
  The coupling block~\eqref{eq:coupling_operator} is estimated through the same
  splitting as in~\eqref{eq:coupling_split} to obtain:
  \begin{equation*}
    \bigl| \bigl\langle \zeta^{\pi}, L^{\pi w}_{z_t^{(n)}}\zeta^{w} \bigr\rangle \bigr|
    \ \le\ \frac18\|\nabla\zeta^{w}\|^2 + C\,\Gamma^{(n)}_t\,\|\zeta_t^{(n)}\|^2 .
  \end{equation*}
  Together the three blocks give
  \begin{equation*}
    \langle L_{z_t^{(n)}}\zeta_t^{(n)},\zeta_t^{(n)}\rangle
    \le -\frac18\|\nabla\zeta_t^{(n)}\|^2
    + C\,\Gamma^{(n)}_t\,\|\zeta_t^{(n)}\|^2 .
  \end{equation*}

  The second term is the source $(L_{z_t^{(n)}}-L_{z_t^{(0)}})\eta_t$. We write
  $P^{\perp}_{\pi}\psi \eqdef \psi - \langle\pi,\psi\rangle\pi$, so
  that~\eqref{eq:coupling_operator} reads
  $L^{\pi w}_{\pi}[\delta w] = -P^{\perp}_{\pi}\bigl((DL^0)[\delta w]\pi\bigr)$
  and~\eqref{eq:fiber_operator} reads
  $L^{\pi\pi}_{z}\delta\pi = P^{\perp}_{\pi}(G_w\delta\pi)
  - \langle\delta\pi,G_w\pi\rangle\pi - \langle\pi,G_w\pi\rangle\delta\pi$. Two
  identities drive the three differences,
  \begin{equation}\label{eq:cd_diff_identities}
    \bigl(P^{\perp}_{\pi^{(n)}_t}-P^{\perp}_{\pi^{(0)}_t}\bigr)\psi
    = -\langle\pi^{(0)}_t,\psi\rangle\,\rho^{(n)}_t
      -\langle\rho^{(n)}_t,\psi\rangle\,\pi^{(n)}_t ,
    \qquad
    G_{w^{(n)}_t}-G_{w^{(0)}_t} = -L^0_{v^{(n)}_t} ,
  \end{equation}
  where the second holds because $w \mapsto L^0_w$ is linear. In the base,
  \begin{equation*}
    \bigl( L^{ww}_{w^{(n)}} - L^{ww}_{w^{(0)}} \bigr)\eta^{w}_t
    = - B\bigl(v_t^{(n)},\eta^{w}_t\bigr) - B\bigl(\eta^{w}_t,v_t^{(n)}\bigr) .
  \end{equation*}
  In the coupling, where we write $D_t \eqdef (DL^0)[\eta^{w}_t]$,
  \begin{equation*}
    \bigl( L^{\pi w}_{\pi^{(n)}_t} - L^{\pi w}_{\pi^{(0)}_t} \bigr)\eta^{w}_t
    = -\bigl(P^{\perp}_{\pi^{(n)}_t}-P^{\perp}_{\pi^{(0)}_t}\bigr)
        \bigl(D_t\pi^{(n)}_t\bigr)
      - P^{\perp}_{\pi^{(0)}_t}\bigl(D_t\rho^{(n)}_t\bigr) .
  \end{equation*}
  In the fiber, where $\varkappa_t \eqdef G_{w^{(n)}_t}\rho^{(n)}_t
  - L^0_{v^{(n)}_t}\pi^{(0)}_t$ is the difference of $G_{w}\pi$,
  \begin{equation*}
    \begin{aligned}
      \bigl( L^{\pi\pi}_{z^{(n)}_t} - L^{\pi\pi}_{z^{(0)}_t} \bigr)\eta^{\pi}_t
      = \ & \bigl(P^{\perp}_{\pi^{(n)}_t}-P^{\perp}_{\pi^{(0)}_t}\bigr)
            \bigl(G_{w^{(n)}_t}\eta^{\pi}_t\bigr)
          - P^{\perp}_{\pi^{(0)}_t}\bigl(L^0_{v^{(n)}_t}\eta^{\pi}_t\bigr)
      \\
      & - \bigl\langle \eta^{\pi}_t, \varkappa_t \bigr\rangle\,\pi^{(n)}_t
        - \bigl\langle \eta^{\pi}_t, G_{w^{(0)}_t}\pi^{(0)}_t \bigr\rangle\,
          \rho^{(n)}_t
      \\
      & - \Bigl( \bigl\langle \rho^{(n)}_t, G_{w^{(n)}_t}\pi^{(n)}_t \bigr\rangle
        + \bigl\langle \pi^{(0)}_t, \varkappa_t \bigr\rangle \Bigr)\,\eta^{\pi}_t .
    \end{aligned}
  \end{equation*}
 We estimate these quantities as follows, where
  $\Gamma_t \eqdef \Gamma^{(n)}_t + \Gamma^{(0)}_t \ge 1$. For the base,
  \begin{align*}
    \Bigl| \bigl\langle \bigl(L^{ww}_{w^{(n)}}-L^{ww}_{w^{(0)}}\bigr)\eta^{w}_t,
      \zeta^{w} \bigr\rangle \Bigr|
    &= \bigl| \langle (K*v^{(n)}_t)\cdot\nabla\zeta^{w},\eta^{w}_t \rangle
      + \langle (K*\eta^{w}_t)\cdot\nabla v^{(n)}_t,\zeta^{w} \rangle \bigr|
    \\
    &\lesssim \|v^{(n)}_t\|\,\|\nabla\zeta^{w}\|\,
        \|\eta^{w}_t\|^{1/2}\|\eta^{w}_t\|_{H^1}^{1/2}
      + \|\eta^{w}_t\|\,\|\nabla v^{(n)}_t\|\,
        \|\zeta^{w}\|^{1/2}\|\nabla\zeta^{w}\|^{1/2}
    \\
    &\le \frac18\|\nabla\zeta^{w}\|^2 + C\Gamma_t\|\zeta_t^{(n)}\|^2
      + C\|v^{(n)}_t\|^2\|\eta_t\|_{H^1}^2
      + C\|\nabla v^{(n)}_t\|^2\|\eta_t\|^2 .
  \end{align*}
   For the coupling, writing
  $\psi \eqdef P^{\perp}_{\pi^{(0)}_t}\zeta^{\pi}$,
  \begin{align*}
    \Bigl| \bigl\langle \bigl(L^{\pi w}_{\pi^{(n)}_t}-L^{\pi w}_{\pi^{(0)}_t}\bigr)
      \eta^{w}_t, \zeta^{\pi} \bigr\rangle \Bigr|
    &\le \bigl| \langle D_t\rho^{(n)}_t, \psi \rangle \bigr|
      + \bigl| \langle \pi^{(0)}_t, D_t\pi^{(n)}_t \rangle \bigr|\,\|\rho^{(n)}_t\|\,\|\zeta^{\pi}\|
      + \bigl| \langle \rho^{(n)}_t, D_t\pi^{(n)}_t \rangle \bigr|\,\|\zeta^{\pi}\|
    \\
    &\lesssim \bigl( \|\eta_t\|\,\|\nabla\rho^{(n)}_t\|
        + \|\rho^{(n)}_t\|\,\|\eta_t\|_{H^1} \bigr) \|\psi\|_{L^4}
      + \Gamma_t\,\|\eta_t\|\,\|\rho^{(n)}_t\|\,\|\zeta^{\pi}\|
    \\
    &\qquad + \Gamma_t^{1/2}\,\|\eta_t\|\,\|\rho^{(n)}_t\|_{H^1}\,\|\zeta^{\pi}\|
    \\
    &\le \frac18\|\nabla\zeta^{\pi}\|^2 + C\Gamma_t\|\zeta_t^{(n)}\|^2
      + C\|\rho^{(n)}_t\|^2\|\eta_t\|_{H^1}^2
      + C\|\nabla\rho^{(n)}_t\|^2\|\eta_t\|^2
    \\
    &\qquad + C\Gamma_t\|\rho^{(n)}_t\|^2\|\eta_t\|^2 ,
  \end{align*}
  For the fiber,
  \begin{align*}
    \Bigl| \bigl\langle \bigl(L^{\pi\pi}_{z^{(n)}_t}-L^{\pi\pi}_{z^{(0)}_t}\bigr)
      \eta^{\pi}_t, \zeta^{\pi} \bigr\rangle \Bigr|
    &\lesssim \bigl( \|\eta_t\|\,\|\nabla v^{(n)}_t\|
        + \|v^{(n)}_t\|\,\|\eta_t\|_{H^1} \bigr)\|\psi\|_{L^4}
      + \Gamma_t \bigl( \|v^{(n)}_t\| + \|\rho^{(n)}_t\| \bigr)
        \|\eta_t\|\,\|\zeta^{\pi}\|
    \\
    &\qquad + \|\nabla\rho^{(n)}_t\|\,\|\eta_t\|_{H^1}\,\|\zeta^{\pi}\|
    \\
    &\le \frac18\|\nabla\zeta^{\pi}\|^2 + C\Gamma_t\|\zeta_t^{(n)}\|^2
      + C\|v^{(n)}_t\|^2\|\eta_t\|_{H^1}^2
      + C\|\nabla v^{(n)}_t\|^2\|\eta_t\|^2
    \\
    &\qquad + C\Gamma_t\bigl(\|v^{(n)}_t\|^2+\|\rho^{(n)}_t\|^2\bigr)\|\eta_t\|^2
      + F^{(n)}_t\,\|\zeta_t^{(n)}\| ,
  \end{align*}
  where 
  \begin{equation}\label{eq:cd_Fdef}
    F^{(n)}_t \eqdef \|\nabla\rho^{(n)}_t\|\,\|\eta_t\|_{H^1} ,
    \qquad
    \int_s^T F^{(n)}_t \ud t
    \ \le\ C \Bigl( \int_s^T \|\nabla\rho^{(n)}_t\|^2 \ud t \Bigr)^{1/2}
      \Bigl( \int_s^T \|\eta_t\|_{H^1}^2 \ud t \Bigr)^{1/2} ,
  \end{equation}
  by Cauchy--Schwarz. Now we use $2ab \le ab^2 + a$ at $a = F^{(n)}_t$ and
  $b = \|\zeta^{(n)}_t\|$ to obtain

  \begin{equation}\label{eq:cd_jineq}
    \frac{\ud}{\ud t}\|\zeta_t^{(n)}\|^2\le \bigl(C\,\Gamma^{(n)}_t
    + F^{(n)}_t\bigr)\|\zeta_t^{(n)}\|^2+g_t^{(n)} ,
    \qquad\zeta_s^{(n)}=0 ,
  \end{equation}
  with
  \begin{equation*}
    \begin{aligned}
      g_t^{(n)} = & C\big(\|v_t^{(n)}\|^2+\|\rho_t^{(n)}\|^2\big)\|\eta_t\|_{H^1}^2
      +C\big(\|\nabla v_t^{(n)}\|^2+\|\nabla\rho_t^{(n)}\|^2\big)\|\eta_t\|^2
      \\
      & +C\,\Gamma^{(n)}_t\big(\|v_t^{(n)}\|^2+\|\rho_t^{(n)}\|^2\big)\|\eta_t\|^2
      + F^{(n)}_t .
    \end{aligned}
  \end{equation*}
  By the parabolic energy estimate for $\partial_t\eta_t=L_{z^{(0)}}\eta_t$ and
  Lemma~\ref{lem:fibre_dissipation}(iii) below, the limiting orbit satisfies
  $\sup_{t\in[s,T]}\|\eta_t\|\le e^{CM}\|\eta\|$ and
  $\int_s^T\|\eta_t\|_{H^1}^2\ud t<\infty$ almost surely, while $\int_s^T
  \Gamma^{(n)}_t\ud t\le M$ by~\eqref{eq:cd_uniform}, uniformly in $n$. The
  three parts of $g_t^{(n)}$ vanish in the limit when integrated, so that
  $\smallint_s^T g_t^{(n)}\ud t\to 0 .$ 
  Gr\"onwall's inequality applied to~\eqref{eq:cd_jineq} therefore gives
  \begin{equation}\label{eq:cd_jgronwall}
    \sup_{t\in[s,T]}\|\zeta_t^{(n)}\|^2
    \le e^{\int_s^T ( C\Gamma^{(n)}_t + F^{(n)}_t )\ud t}
      \int_s^T g_t^{(n)}\ud t\to 0 ,
  \end{equation}
  that is $J_{s,t}^{(n)}\eta\to J_{s,t}^{(0)}\eta$ in $\mH \times \mH$, uniformly on $[s,T]$. This completes the proof.
\end{proof}
Note that we only require pointwise continuity in the data, as opposed to the
operator convergence in \cite{HairerMattingly08}*{Section~5.3, Theorem~5.10} and
\cite{BBPS-AOP-22}*{Proposition~5.5}. We are now ready to prove the main result.

\begin{proof}[Proof of Proposition~\ref{prop:uniform_tail}]
  We argue by contradiction. The map $\ve\mapsto r(\ve,\alpha,T,R,N)$ is
  nondecreasing in $\ve$, so if it does not tend to $0$, there exist $c>0$ and a
  sequence $\ve_n\downarrow 0$ with $r(\ve_n, \alpha, T, R, N) \ge 2c$. Then let
  us choose
  $z_0^{(n)} \in \mC_R$ such that for every $n \in
  \NN_*$ we have
  \begin{equation}\label{eq:ut_hyp}
    \PP\Big( \inf_{\eta\in S_{\alpha,N,T}} \langle M_T(z_0^{(n)})\eta,\eta\rangle < \ve_n \Big) \ge c .
  \end{equation}
  The central region $\mC_R$ is a compact subset of $\mX$, so we may assume that there exists $z_0$ with
  \begin{equation*}
    z_0^{(n)} \to z_0^{(0)} \quad\text{in } \mX, \qquad z_0^{(0)}\in\mC_R .
  \end{equation*}
  For fixed initial data $z\in\mX$ and a fixed realization we set $m(z) \eqdef \inf_{\eta\in S_{\alpha,N,T}} \langle M_T(z)\eta,\eta\rangle \ge 0$,
  and abbreviate $m_n\eqdef m(z_0^{(n)})$ and $m_0\eqdef m(z_0^{(0)})$. Each
  $m_n$ is a nonnegative random variable by the measurability argument in the
  proof of Proposition~\ref{prop:as_nondegeneracy}, and we claim that
  \begin{equation}\label{eq:ut_lsc}
    m_0 \le \alpha^{-2}\liminf_{n\to\infty} m_n .
  \end{equation}
  To prove \eqref{eq:ut_lsc} we fix a subsequence along which $m_n$ converges to $\liminf_n m_n$ (however, to keep the notation clean we will not write the subsequence explicitly and instead work still with $n\in \NN$), and for each $n$ select a near-minimizer $\eta_n\in S_{\alpha,N,T}$ with
  \begin{equation*}
    \langle M_T(z_0^{(n)})\eta_n,\eta_n\rangle \le m_n + \frac1n .
  \end{equation*}
  Hence, along a further subsequence $\eta_n\rightharpoonup\eta_*$, and the
  extraction in the proof of Proposition~\ref{prop:as_nondegeneracy} gives
  $\|\eta_*\| \le 1$ and $\|P_N\eta_*\| \ge \alpha$, so that $\eta_* \ne 0$ and
  also $\eta_* \in T_{z_T}\mX$.
  Hence the normalized vector $\hat\eta_*\eqdef\eta_*/\|\eta_*\|$ belongs again
  to $S_{\alpha,N,T}$, with $\|\eta_*\|\ge\|P_N\eta_*\|\ge\alpha$.

  It remains to pass to the limit in the quadratic form along the near-minimizers. Fix $s \leq T$ and a base test function $\zeta = (f, 0)$ with $f \in \mH$. By Lemma~\ref{lem:continuity_in_data} we have that $J^{(n)}_{s,T}\zeta$ converge strongly and by assumption $\eta_n$ converge weakly, so that from
  \begin{equation*}
    \big\langle (J^{(n)}_{s,T})^*\eta_n, \zeta \big\rangle
    = \big\langle \eta_n, J^{(n)}_{s,T}\zeta \big\rangle ,
  \end{equation*}
  we deduce
  \begin{equation*}
    \Pi_w (J^{(n)}_{s,T})^*\eta_n \rightharpoonup \Pi_w (J^{(0)}_{s,T})^*\eta_*
    \quad \text{in } \mH
    \qquad \forall s\in[0,T].
  \end{equation*}
  Next, the functional $\xi\mapsto\|Q^{1/2}\xi\|^2$ is continuous and convex on $\mH$, hence weakly lower semicontinuous at each $s$, so that with Fatou's lemma
  \begin{align*}
    \langle M_T(z_0^{(0)})\eta_*,\eta_*\rangle
    & = \int_0^T \|Q^{1/2}\Pi_w (J^{(0)}_{s,T})^*\eta_*\|^2 \ud s                        \\
    & \le \liminf_{n\to\infty} \int_0^T \|Q^{1/2}\Pi_w (J^{(n)}_{s,T})^*\eta_n\|^2 \ud s
    = \liminf_{n\to\infty} \langle M_T(z_0^{(n)})\eta_n,\eta_n\rangle .
  \end{align*}
  Along the chosen subsequence the near-minimizing property forces $\langle
  M_T(z_0^{(n)})\eta_n,\eta_n\rangle\le m_n+\frac1n$, whose limit is $\liminf_n
  m_n$. Using $\hat\eta_*\in S_{\alpha,N,T}$ together with the bound
  $\|\eta_*\|\ge\alpha$, we conclude \eqref{eq:ut_lsc} from
  \begin{equation*}
    m_0 \le \langle M_T(z_0^{(0)})\hat\eta_*,\hat\eta_*\rangle
    = \frac{\langle M_T(z_0^{(0)})\eta_*,\eta_*\rangle}{\|\eta_*\|^2}
    \le \alpha^{-2}\liminf_{n\to\infty} m_n .
  \end{equation*}

  By \eqref{eq:ut_hyp} the events $A_n\eqdef\{m_n<\ve_n\}$ satisfy $\PP(A_n)\ge
  c$, so the reverse Fatou lemma gives $\PP(\limsup_n A_n) \ge c$. On $\limsup_n
  A_n$ the bound $m_n < \ve_n$ holds for infinitely many $n$, which forces
  $\liminf_n m_n = 0$ and hence $m_0 = 0$ by~\eqref{eq:ut_lsc}. This
  contradicts  Proposition~\ref{prop:as_nondegeneracy}.
\end{proof}

\subsection{Tikhonov contraction}
\label{sec:jacobian_moments}
The aim of this section is to establish certain estimates for a Tikhonov
regularization of the Malliavin matrix. Because we don't have any polynomial
moment on the Jacobian, such estimates are not quantitative (as would be the case in classical
approaches to asymptotic coupling) probabilistically, and rely on bounds on the
Jacobian that are available only when the process belongs to a central region
and require some form of localization. For this reason we define (the dependence on $T$ is not important so we omit it in the notation):
\begin{equation}\label{eq:good_event_EK}
  E_K \eqdef \left\{ \left( \int_0^T  \|w_r\|_{H^1}^4 + \|\pi_r\|_{H^1}^4  \ud r \right)^{\frac{1}{4}} \leq K \right\} , \qquad K, T > 0.
\end{equation}
Recall the definition of $P_N, Q_N$, which we use also for their
componentwise action on $\mH \times \mH$.

\begin{lemma}\label{lem:fibre_dissipation}
  Fix $T > 0$. Almost surely:
  \begin{enumerate}[label=(\roman*)]
    \item It holds that
      \begin{equation}\label{eq:rayleigh_bound}
        -\langle\pi_r,G_{w_r}\pi_r\rangle
        = \|\nabla\pi_r\|^2 + \langle\pi_r,B(\pi_r,w_r)\rangle
        \le C\Gamma_r .
      \end{equation}
    \item For every $\delta\pi \in \mH$ (not necessarily in $\pi_r^\perp$)
      \begin{equation}\label{eq:fiber_dissipation}
        \langle\delta\pi,L^{\pi\pi}_{z_r}\delta\pi\rangle
        \le -\frac14\|\nabla\delta\pi\|^2 + C\Gamma_r\|\delta\pi\|^2 .
      \end{equation}
    \item For all $0 \le s \le t \le T$ and every
      $\eta = (\delta w, \delta\pi) \in \mH \times \mH$, with $C = C(T)$
      \begin{equation}\label{eq:jacobian_pathwise}
        \sup_{r\in[s,t]}\|J_{s,r}\eta\|^2
        + \int_s^t \|J_{s,r}\eta\|_{H^1}^2 \ud r
        \le C \exp\Bigl( C\int_s^t \Gamma_r \ud r \Bigr) \|\eta\|^2 .
      \end{equation}
  \end{enumerate}
\end{lemma}

\begin{lemma}\label{lem:jacobian_moments}
  For any $T > 0$ and any $\delta, K, \tau_0 > 0$, there exists $N_0 = N_0(\delta,\tau_0, K, T)$, such that for all $N \geq N_0$, $s \geq \tau_0/2$, $t-s \geq \tau_0$, and any $\eta \in T_{z_s} \mathcal{X}$:
  \begin{equation}\label{eq:high_mode_decay_fiber}
    \|J_{s,t} Q_N \eta \| \leq \delta \| \eta \| \quad \text{ on } \quad E_K .
  \end{equation}
\end{lemma}

The proofs of all three of these statements are given in Appendix~\ref{app:jacobian_moments}.
We now combine the moment bounds on the Jacobian with the non-degeneracy of the
Malliavin matrix.

\begin{lemma}\label{lem:tikhonov_contraction}
  For any $\delta \in (0,1]$, $K,R, T > 0$, $p \geq 1$, and $N \in \NN$, there exists $\beta = \beta(\delta, p, N, K, R, T) > 0$ such that for all $z_0 \in \mC_R$ and $\eta \in T_{z_0} \mathcal{X}$ with $\|\eta\| \leq 1$ we have:
  \begin{equation}\label{eq:tikhonov_bound}
    \EE\left[ \| P_N \beta(\beta I + M_T )^{-1}  J_{0,T} \eta\|^p \cdot \mathbf{1}_{E_K} \right] \leq \delta^p .
  \end{equation}
\end{lemma}

\begin{proof}
  The proof follows \cite{HairerMattingly11}*{Lemma~5.14 and Corollary~5.15}, see also
  \cite{DongPeng24}*{Lemmas~4.1 and~4.2}, where the corresponding bound contains
  a Lyapunov weight in place of the localization on $E_K$. Let us write
  \begin{equation}\label{eq:resolvent_blue}
    R^\beta\eqdef\beta(\beta I+M_T)^{-1}.
  \end{equation}
  Since $M_T$ is self-adjoint and nonnegative, the operator $R^\beta$ is
  self-adjoint with $0\le R^\beta\le I$, and it satisfies the resolvent identity
  $(\beta I+M_T)R^\beta=\beta I$. For any $\alpha \in (0,1)$ recall from~\eqref{eq:cone_SaN} that
  $S_{\alpha,N,T}$ is the unit slice of the cone
  \begin{equation}\label{eq:cone_lambda}
    \Lambda_{\alpha, N, T}\eqdef\big\{\eta\in T_{z_T}\mathcal X:\ \|P_N\eta\|\ge\alpha\|\eta\|\big\}.
  \end{equation}
  Now, for $\ve>0$ consider the event
  \begin{equation}\label{eq:good_floor_event}
    A_\ve\eqdef\Big\{\inf_{\eta\in S_{\alpha,N,T}}\langle M_T(z_0)\eta,\eta\rangle\ge\ve\Big\},
  \end{equation}
  on which the Malliavin matrix is non-degenerate on the cone with floor $\ve$. We claim that on $A_\ve$ the operator norm of $P_NR^\beta$ satisfies
  \begin{equation}\label{eq:hm514}
    \|P_N R^\beta\|_{\mathrm{op}}\le\alpha\vee\sqrt{\beta/\ve}.
  \end{equation}
  Indeed, fix $\xi\in T_{z_T}\mathcal X$ with $\xi\neq0$ and write $\zeta\eqdef
  R^\beta\xi$, which is nonzero since $R^\beta$ is injective, and obeys
  $\|\zeta\|\le\|\xi\|$. If
  $\zeta\notin\Lambda_{\alpha,N,T}$ the cone condition fails for $\zeta$ and
  hence $\| P_N R^\beta \xi \| \leq \alpha \| \zeta \| \leq \alpha \| \xi \|$. If instead
  $\zeta\in\Lambda_{\alpha,N,T}$ then $\zeta/\|\zeta\|\in S_{\alpha,N,T}$, so on
  $A_\ve$ and by $(\beta I+M_T)R^\beta=\beta I$  we obtain as desired
  \begin{equation}\label{eq:hm514_cases}
    \begin{aligned}
      \ve\|\zeta\|^2
      & \le\langle M_T\zeta,\zeta\rangle\le\langle(\beta I+M_T)\zeta,\zeta\rangle=\beta\langle\xi,\zeta\rangle\le\beta\|\xi\|^2 .
    \end{aligned}
  \end{equation}
  Next, by Proposition~\ref{prop:uniform_tail} the probability of the degenerate event is small uniformly in the initial data,
  \begin{equation}\label{eq:tail_recall}
    \sup_{z_0 \in \mC_R}\PP(A_\ve^c)\le r(\ve,\alpha,T,R,N),\qquad r(\ve,\alpha,T,R,N)\to0\ \text{ as }\ve\to0.
  \end{equation}
  Combining~\eqref{eq:hm514} on $A_\ve$ with the trivial bound $\|P_N R^\beta\|_{\mathrm{op}}\le1$ on $A_\ve^c$, we obtain
  \begin{equation}\label{eq:hm515}
    \begin{aligned}
      \EE\big[\|P_N R^\beta J_{0,T} \eta\|^p\mathbf 1_{E_K}\big] & \leq C_K^p \EE\big[\|P_N R^\beta\|_{\mathrm{op}}^p\mathbf 1_{E_K}\big]  \le C_K^p \big(\alpha\vee\sqrt{\beta/\ve}\big)^p+ C_K^p \PP(A_\ve^c)         \\
      & \le C_K^p \big(\alpha\vee\sqrt{\beta/\ve}\big)^p+ C_K^p r(\ve,\alpha,T,R,N).
    \end{aligned}
  \end{equation}
  It remains to fix the parameters. We first choose $\alpha =
  \alpha(K,\delta,T)$ small enough that $C_K^p\alpha^p\le\delta^p/2$. We then choose
  $\ve$ small enough that
  $C_K^pr(\ve,\alpha,T,R,N)\le\delta^p/2$, which is possible because
  $r(\ve,\alpha,T,R,N)\to0$ as $\ve\to0$. Finally we set
  $\beta\eqdef\alpha^2\ve$, so that $\sqrt{\beta/\ve}=\alpha$. With these
  choices~\eqref{eq:hm515} becomes \[ \EE\big[\|P_N R^\beta
    J_{0,T}\eta\|^p\mathbf 1_{E_K}\big] \le C_K^p\alpha^p+C_K^pr(\ve,\alpha,T,R,N)
  \le\frac{\delta^p}{2}+\frac{\delta^p}{2}=\delta^p, \] which
  is~\eqref{eq:tikhonov_bound}.
\end{proof}

\section{Lie algebra transitivity}
\label{sec:transitivity}

In this section we prove the algebraic aspects of the invertibility of the
Malliavin matrix studied in Section~\ref{sec:quantitative_nondegeneracy}.
Let us introduce the
vector spaces of operators $\mf{g}_{\mathrm{LNS}}, \mf{g}_{\mathrm{PSA}}
\subseteq \Opbeta$ (throughout this section the space of operators $\Opbeta$ is
considered complexified) defined by:
\begin{equation}\label{eq:lie_algebras}
  \begin{aligned}
    \mf{g}_{\mathrm{LNS}} & = \mathrm{Lie}\left( H^k \colon k \in \Zstar \right) ,\footnotemark\\
    \mf{g}_{\mathrm{PSA}}& =\mathrm{Span}\Bigl(
      H^k, \ [\Delta, H^k], \  [[\Delta, H^k], H^{\ell}], \  \bigl[[[\Delta, H^k], H^{-k}], H^\ell\bigr] : k, \ell \in \Zstar
    \Bigr) .
  \end{aligned}
\end{equation}
\footnotetext{$\mathrm{Lie}(S)$ denotes the smallest linear
subspace containing $S$ and closed under commutation (Lie bracket) $[A, B] = AB -
BA$ .}
with $H^k$ from \eqref{eq:generator_modes}.
We note that in the case of PSA the natural Lie algebra to consider would be
\begin{equation*}
  \mathrm{Lie} \left(H^k, \ [\Delta, H^k], \ k \in \Zstar \right),
\end{equation*}
however, this is not a subset of $\Opbeta$, because it contains operators that
are essentially differential operators of arbitrary order, whereas in $\Opbeta$
we only allow for a fixed loss of regularity. For a proof that $
\mf{g}_{\mathrm{LNS}}, \mf{g}_{\mathrm{PSA}} \subseteq \Opbeta$ we refer to
Lemma~\ref{lem:lns_algebra} and Lemma~\ref{lem:psa_algebra} in the upcoming
sections. Given these two results it is immediate to verify the following lemma,
which is used for our applications in Section~\ref{sec:quantitative_nondegeneracy}.
\begin{lemma}\label{lem:reachable}
  Both $\mf{g}_{\mathrm{LNS}}$ and $\mf{g}_{\mathrm{PSA}}$ are bracket-generated in the sense of Definition~\ref{def:bracket_regular}.
\end{lemma}
\begin{proof}
  The only substantial aspect is to prove that the vector spaces are bracket
  regular. This follows from Lemma~\ref{lem:lns_algebra} in the case of LNS
  and Lemma~\ref{lem:psa_algebra} in the case of PSA. The remaining properties
  follow from the definition~\eqref{eq:lie_algebras}.
\end{proof}

We note that overall the discussion for SNS is substantially easier
than the PSA case.

\begin{theorem}\label{thm:transitivity}
  Fix any $T>0$ and $z_0 \in \mX$. Then
  \begin{itemize}
    \item In the case of LNS \eqref{eq:two_operators_intro_NS}, for any $\delta
      \pi \in T_{\pi_T} \mathbb{S}$ the following holds:
      \begin{equation*}
        \langle \delta\pi, \vf{A}\pi_T \rangle = 0
        \qquad \forall A \in \mf{g}_{\mathrm{LNS}} \quad \Rightarrow \quad \delta \pi = 0 .
      \end{equation*}
    \item In the case of PSA~\eqref{eq:two_operators_intro_PS}, there exists a set $\Omega_0 \subseteq \Omega$ (depending
      on the initial condition $z_0$ and the terminal time $T$) with $\PP(\Omega_0)=1$, such that the following holds for all events $\omega \in \Omega_0$. For any $\delta \pi \in T_{\pi_T} \mathbb{S}\cap H^{5}$ the following holds:
      \begin{equation*}
        \langle \delta\pi, \vf{A}\pi_T \rangle = 0
        \qquad \forall A \in \mf{g}_{\mathrm{PSA}} \quad \Rightarrow \quad \delta \pi = 0 .
      \end{equation*}
  \end{itemize}
\end{theorem}
Theorem~\ref{thm:transitivity} is proved for LNS at the end of Section~\ref{sec:tran-NS}, and for PSA in Section~\ref{sec:tran-PSA}, after Proposition~\ref{prop:transitivity_PSA} and Proposition~\ref{prop:psa_star_fixedtime}.

\begin{remark}
  The two cases are algebraically different. In the LNS case the $H^k$ alone are insufficient, and it is the iterated brackets between them that generate the missing directions. In the PSA case the $H^k$ instead span the exact, mean-zero Hamiltonian vector fields on $\TT^2$, and $[H^j, H^k]$ is proportional to $H^{j+k}$, so they close under bracketing and produce no new direction. This bracket relation is the classical Fourier presentation of the Poisson algebra of area-preserving vector fields \cites{ArnoldKhesin98, Zeitlin91}, and it reflects the conservation of the Casimirs
  \[
    \int_{\TT^2} f(\zeta_T) \ud x = \int_{\TT^2} f(\zeta_0) \ud x , \qquad f \in C(\RR) ,
  \]
  by the undiffused passive scalar. One must therefore bracket again with the drift, which reduces to the Laplacian by Corollary~\ref{cor:propagation} and gives rise to the truncated linear space $\mf{g}_{\mathrm{PSA}}$, which is notably not a Lie algebra.

  The proof for LNS follows much of the machinery laid out in~\cite{BedrossianPunshon-Smith-Chaos-2024y}, while the PSA case is new and rests on a rigidity property of the scalar.
\end{remark}

\subsection{Transitivity for the linearized Navier--Stokes dynamics}
\label{sec:tran-NS}

The transitivity of $\mf{g}_{\mathrm{LNS}}$ proved in this subsection is the infinite-dimensional counterpart of the Lie algebra analysis of \cite{BedrossianPunshon-Smith-Chaos-2024y}, which established the same spanning property for the Galerkin truncations of the stochastic Navier--Stokes equations. Notably much of the machinery developed there is applicable in the infinite dimensional setting and the argument below follows many of its steps, in particular the band decomposition, the diagonal operators obtained from $[H^k, H^{-k}]$, and the passage from basis pairs to elementary matrices.

Recall that in the case of linearized Navier--Stokes we have
$ H^k = B(e_k, \cdot) + B(\cdot, e_k)$, so that
\begin{equation}\label{eq:Hk_symbol}
  H^k = e_k\, h_k(D) ,
  \qquad
  h_k(l) = \langle l^\perp, k\rangle\Big(\frac{1}{|k|^2} - \frac{1}{|l|^2}\Big) ,
  \qquad
  |h_k(l)| \le \frac{|l|}{|k|} + \frac{|k|}{|l|} ,
\end{equation}
with the convention that every symbol vanishes at $l = 0$, so that its action in Fourier coordinates reads
\begin{equation*}
  (\widehat{H^k \psi})_l = h_k(l-k) \hat\psi_{l-k}, \qquad \forall l \in \Zstar .
\end{equation*}
In particular, $H^k$ is supported on band $k$, according to the definition that an operator $A$ is supported on band $k \in \Zstar$ if $A e_j = c_j e_{j+k}$ for all $j \in \Zstar$ (or equivalently if $A = \sum_{j \in \Zstar} c_j E^{j+k, j}$ with $E^{j+k, j}$ the rank one operator that maps $e_j \mapsto e_{j+k}$).

We start by observing that $\mf{g}_{\mathrm{LNS}}$ satisfies the assumption of Lemma~\ref{lem:bracket_extraction} and in particular that $\mf{g}_{\mathrm{LNS}} \subseteq \Opbeta$. To this aim, for $m \in \ZZ$, we write $a \in S^m$ (we say that $a$ is a symbol of order $m$) if $a \colon \Zstar \to \CC$ and it agrees, outside a bounded set, with a smooth function $\alpha$ on $\RR^2$ satisfying
\begin{equation*}
  |\nabla^i \alpha(l)| \le C_i\, |l|^{m - i} , \qquad i \ge 0 .
\end{equation*}
Orders add under products of symbols, and by the mean value theorem a
shift-difference $a(\cdot + l) - a$ lowers the order by one. Furthermore, we say that an
operator $A$ is supported on finitely many bands if it is of the form
\begin{equation*}
  A = \sum_{m \in F_A} e_m\, p_{A,m}(D),
\end{equation*}
for some finite subset $F_A \subseteq \Zstar$ and a collection of symbols $\{p_{A,m}\}_{m \in F_A}$.
  For brackets of such operators we find the following property.

  \begin{lemma}\label{lem:band_calculus}
    Let $A$ and $B$ be supported on finitely many bands, with coefficients in $S^r$ and in $S^s$ respectively. Then $[A, B]$ is supported on finitely many bands, with coefficients in $S^{r+s-1}$.
  \end{lemma}

  \begin{proof}
    Since $p(D)(e_n \psi) = e_n\, p(D+n)\psi$, two bands compose as $e_m\, p(D)\, e_n\, q(D) = e_{m+n}\, p(D+n)\, q(D)$, and therefore
    \begin{equation}\label{eq:band_bracket}
      [\, e_m\, p(D),\ e_n\, q(D) \,]
      \ =\ e_{m+n}\, \Big( \big(p(\cdot+n) - p\big)\, q \ -\ \big(q(\cdot+m) - q\big)\, p \Big)(D) ,
    \end{equation}
    and the product $p q$ terms cancel. Frequency supports add, so the bands
    stay finite, and the shift-difference lowers the order by one, so the first
    coefficient lies in $S^{r-1} S^{s}$ and the second in $S^{s-1} S^{r}$:
    both in $S^{r+s-1}$.
\end{proof}

\begin{lemma}\label{lem:lns_algebra}
  The symbols $h_k$ from~\eqref{eq:Hk_symbol} satisfy $h_k \in S^1$. Therefore each $A \in \mf{g}_{\mathrm{LNS}}$ is a finite sum of the form $A = \sum_{m \in F_A} e_m\, p_{A,m}(D)$ for some $p_{A,m} \in S^1$. In particular $|p_{A,m}(l)| \lesssim_A 1 + |l|$, so $A$ is a first-order operator, $A \in \Opbeta$, and $A$ is bracket-regular in the sense of Definition~\ref{def:bracket_regular}.
\end{lemma}

\begin{proof}
  It is straightforward to check that $h_k \in S^1$, and moreover finite sums of
  band-limited multipliers with symbols in $S^1$ form a Lie algebra, by
  Lemma~\ref{lem:band_calculus} at $r = s = 1$. Now, the representation of a
  general $A \in \mf{g}_{\mathrm{LNS}}$ follows by induction from the
  generators, and the fact that $A \in \Opbeta$ follows from
  \begin{equation*}
    \|e_m\, p(D)\varphi\|_{\mH}
    \ =\ \|p(D)\varphi\|_{\mH}
    \ \le\ \sup_{l} \frac{|p(l)|}{(1+|l|)^{\beta_\star}}\ \|\varphi\|_{H^{\beta_\star}} \lesssim  \|\varphi\|_{H^{\beta_\star}},
  \end{equation*}
  since $\beta_\star > 1$.

  We now fix $A$ and verify the three conditions of Definition~\ref{def:bracket_regular}. By~\eqref{eq:band_bracket} we have
  \begin{equation*}
    [H^k, A]
    \ =\ \sum_{m \in F_A} e_{k+m}\,
    \big(\, h_k(\cdot+m)\, p_{A,m} \ -\ p_{A,m}(\cdot+k)\, h_k \,\big)(D) ,
  \end{equation*}
  and on the lattice the two symbol products are bounded, up to a constant
  depending on $A$, by $(1+|l|)^2 + |k|^2$, since $|h_k(l)| \le |l|/|k| +
  |k|/|l|$ by~\eqref{eq:Hk_symbol} and $p_{A,m} \in S^1$. Dividing by
  $(1+|l|)^{\beta_\star}$ and taking the supremum over $l$ gives the desired
  \begin{equation*}
    \sup_{k \in \Zstar}\ |k|^{-2}\, \big\| [H^k, A] \big\|_{\Opbeta} \ <\ \infty.
  \end{equation*}
  The third hypothesis follows since $L^0_w =  \sum_j \hat w(j)\, H^j$ and
  \begin{equation*}
    \|[L^0_w, A]\|_{\mL(H^{\beta_\star};\mH)}
    \ \le\ \sum_j |\hat w(j)|\, \|[H^j, A]\|_{\Opbeta}
    \ \lesssim_A\ \sum_j |\hat w(j)|\, |j|^2
    \ \lesssim\ \|w\|_{H^{s_\star-2}}
  \end{equation*}
  by Cauchy--Schwarz, using $s_\star - 2 > 3$. For the second hypothesis write $[G_{w_s}, A] = [\Delta, A] - [L^0_{w_s}, A]$. Applying~\eqref{eq:band_bracket} with the symbol $-|l|^2$ of $\Delta$,
  \begin{equation*}
    [\Delta,\ e_m\, p(D)] \ =\ e_m\, r(D) ,
    \qquad
    r(l) \ =\ -\big( 2\, m \cdot l + |m|^2 \big)\, p(l) ,
  \end{equation*}
  a band-limited multiplier of order two, hence bounded from $H^{\beta_\star-2}$
  to $H^2$ since $\beta_\star \ge 6$, and constant in $s$. For the other term,
  we use as before $[L^0_{w_s}, A] = \sum_j \hat w_s(j)\, [H^j, A]$ together
  with $\|[H^j, A]\|_{\mL(H^{\beta_\star-2};\, H^2)} \lesssim_A |j|^4$, and the regularity of $ w$ (since $s_\star > 5$).
  The required continuity of $s \mapsto [G_{w_s}, A]$ in
  $\mL(H^{\beta_\star-2}; H^2)$ then follows, since $w \mapsto [L^0_w, A]$ is
  linear with the displayed bound and $s \mapsto w_s \in H^{s_\star}$ is almost
  surely continuous. This completes the proof.
\end{proof}

Now we move toward the proof of Theorem~\ref{thm:transitivity} in the case of
LNS. The parts of \cite{BedrossianPunshon-Smith-Chaos-2024y} that concern the Lie
  algebra transitivity adapt immediately to infinite dimensions, and in fact the
proofs simplify without the effect of boundary conditions. As in~\cite{BedrossianPunshon-Smith-Chaos-2024y} we start
by observing that the commutator
of two operators on bands $k_1$ and $k_2$ is supported on band $k_1 + k_2$.
Therefore, the commutator of $H^k$ and $H^{-k}$ is supported on band zero,
meaning that it is diagonal. We define the diagonal operators thus
obtained
\begin{equation}
  \mathbb{D}^k \eqdef [H^k, H^{-k}],\quad k \in \Zstar .
\end{equation}
Next we set the diagonal subalgebra $\h$ as the linear span of these diagonal generators:
\begin{equation*}
  \h \eqdef \mathrm{span} \{\mathbb{D}^k : k \in \Zstar\} \subseteq \mf{g}_{\mathrm{LNS}} .
\end{equation*}
While $H^k$ and $H^{-k}$ are unbounded, their commutator is compact due to a
cancellation. Indeed in Fourier coordinates
$(\widehat{\mathbb{D}^k \psi})_l = \mathbb{D}^k_l \hat{\psi}_l$ with eigenvalues
\begin{equation}\label{eq:Dkl-def}
  \begin{aligned}
    \mathbb{D}^k_l
    = \langle l^\perp, k \rangle^2
    &\left( \frac{1}{|k|^2} - \frac{1}{|l|^2} \right)
    \left( \frac{1}{|l-k|^2} - \frac{1}{|l+k|^2} \right) ,
  \end{aligned}
\end{equation}
see \cite{BedrossianPunshon-Smith-Chaos-2024y}*{Remark~4.4}. Since $\mathbb{D}^k_l = O(|l|^{-1})$ as $|l| \to \infty$, each $\mathbb{D}^k$ is a compact operator on $\mH$.

Now, let $\mathcal{K}(\mH)$ denote the set of compact operators on $\mH$ (and recall that in this section $\mH$ is the \emph{complex} Hilbert space of complex-valued mean zero $L^2$ functions), and let
\begin{equation*}
  \mathcal{I} \eqdef \mathrm{Ideal}_{\mf{g}_{\mathrm{LNS}}}(\h)
\end{equation*}
be the Lie algebra ideal generated by $\h$ within $\mf{g}_{\mathrm{LNS}}$. Namely, the smallest subspace of $\mf{g}_{\mathrm{LNS}}$ containing $\h$ and stable under bracketing $[\mf{g}_{\mathrm{LNS}}, \cdot\,]$. Its elements are finite linear combinations of iterated brackets $\ad(A_1)\cdots\ad(A_r)\, \mathbb{D}$ with $A_i \in \mf{g}_{\mathrm{LNS}}$ and $\mathbb{D} \in \h$, and where $\mathrm{ad}(A)B = [A, B]$.

\begin{proposition} \label{prop:lie_algebra_dense}
  For $\mathcal{I}$ as above it holds that $\overline{\mathcal{I}}^{\|\cdot\|_{\mathrm{op}}} = \mathcal{K}(\mH)$.
\end{proposition}
We start with the inclusion $\overline{\mathcal{I}}^{\|\cdot\|_{\mathrm{op}}} \subseteq \mathcal{K}(\mH)$, which is the consequence of the following lemma.

\begin{lemma}\label{lem:symbol_ideal}
  Every element of $\mathcal{I}$ is supported on finitely many bands, with coefficients in $S^{-1}$, and is compact. In particular $\overline{\mathcal{I}}^{\|\cdot\|_{\mathrm{op}}} \subseteq \mathcal{K}(\mH)$.
\end{lemma}

\begin{proof}
  Write $\mathcal{J}$ for the space of (compact) operators supported on finitely many bands whose coefficients lie in $S^{-1}$. As mentioned above, $D^k_l = O(|l|^{-1})$ as $|l| \to \infty$ and therefore $\h \subseteq \mathcal{J}$. On the other hand, the coefficients of any $B \in \mathcal{J}$ lie in $S^{-1}$ and those of any $A \in \mf{g}_{\mathrm{LNS}}$ in $S^1$ by Lemma~\ref{lem:lns_algebra}, so Lemma~\ref{lem:band_calculus} gives $[A, B] \in \mathcal{J}$. Thus $\mathcal{J}$ contains $\h$ and is stable under bracketing with $\mf{g}_{\mathrm{LNS}}$, and therefore $\mathcal{I} \subseteq \mathcal{J}$. Since the elements of $\mathcal{J}$ are compact operators, and compact operators are closed in the operator norm, we obtain the last inclusion.
\end{proof}

Now we move to proving the reverse inclusion of
Proposition~\ref{prop:lie_algebra_dense}. For simplicity, in the remainder of
this discussion we recall
from~\eqref{eq:interaction_coefficients} the interaction coefficients and write
$$c_{j,k} = \langle j^\perp, k \rangle ( |k|^{-2} - |j|^{-2} ) = h_k(j).$$ Our
aim is to isolate, inside $\overline{\mathcal{I}}^{\|\cdot\|}$, a sufficiently
large set of rank-one and rank-two operators. For $p, q \in \Zstar$, recall that the
elementary matrix $E^{p,q}$ acts by $E^{p,q} e_j = \delta_{q,j}\, e_p$.
Moreover, we define the basis pair as the trace-zero operator
\begin{equation*}
  M^{p,q} \eqdef c_{q,p-q}\, E^{p,q} - c_{p,p-q}\, E^{-q,-p} .
\end{equation*}
Oddness of $c_{j,k}$ in $j$ gives $M^{-q,-p} = M^{p,q}$, and $M^{p,-p} = 0$ since $c_{\pm p, 2p} = 0$. The basis pairs are eigen-operators of $\ad(\mathbb{D})$ for any $\mathbb{D} \in \h$. Since every $\mathbb{D} \in \h$ satisfies $\mathbb{D}_{-l} = -\mathbb{D}_l$, a direct calculation gives
\begin{equation}\label{eq:pair_eigen}
  [\mathbb{D}, M^{p,q}] = (\mathbb{D}_p - \mathbb{D}_q)\, M^{p,q} ,
\end{equation}
since $\mathbb{D}_{-q} - \mathbb{D}_{-p} = \mathbb{D}_p - \mathbb{D}_q$. Each generator is a sum of basis pairs along its band. Writing
\begin{equation} \label{eq:Rk_def}
  R_k \eqdef \{ j \in \Zstar \setminus \{-k\} \,:\, 2j + k \in \Zstarp \} ,
\end{equation}
which contains exactly one of $j$ and $-j-k$ for every such pair, because $j \mapsto -j-k$ replaces $2j+k$ by $-(2j+k)$ and $\Zstar = \Zstarp \sqcup (-\Zstarp)$, the case $2j + k = 0$ being omitted since its basis pair vanishes, the identity $c_{-j-k,k} = -c_{j+k,k}$ yields
\begin{equation}\label{eq:Hk_pairs}
  H^k = \sum_{j \in R_k} M^{j+k, j} .
\end{equation}
To isolate a single basis pair from~\eqref{eq:Hk_pairs} we use diagonal
operators whose adjoint action separates the pairs. This approach is that of \cite{BedrossianPunshon-Smith-Chaos-2024y}, applied to the truncated setting. While no single generator $\mathbb{D}^m$ is sufficient to separate them, a generic finite linear combination of them is.

\begin{lemma}\label{lem:distinct}
  There is a finite set $G \subset \Zstar$ such that for every $k \in
  \Zstar$ some $\mathbb{D} = \sum_{m \in G} \alpha_m \mathbb{D}^m$ is distinct
  for band $k$, that is, its eigenvalue differences $\lambda^k_j =
  \mathbb{D}_{j+k} - \mathbb{D}_j$ satisfy $\lambda^k_j \neq 0$ for all $j \in
  R_k$ and $\lambda^k_{j_1} \neq \lambda^k_{j_2}$ whenever $j_1 \neq j_2$ in
  $R_k$.
\end{lemma}
  \begin{proof}
    Let $G \subset \Zstar$ be finite, chosen below, and let $\mathbb{D} = \sum_{m \in G} \alpha_m \mathbb{D}^m$ with $\alpha \in \RR^{G}$. By the antisymmetry $\mathbb{D}^m_{-l} = -\mathbb{D}^m_l$, the eigenvalue differences and their pairwise differences are linear in $\alpha$,
    \begin{align}
      \lambda^k_j
      &= \sum_{m \in G} \alpha_m W_{j,k}(m), \label{eq:lin_nondeg}\\
      \lambda^k_{j_1} - \lambda^k_{j_2} 
      &= \sum_{m \in G} \alpha_m \bigl( W_{j_1,k} - W_{j_2,k} \bigr)(m), \label{eq:lin_sep}
    \end{align}
    where
    \begin{equation}\label{eq:W_def}
      W_{j,k}(m) \eqdef \mathbb{D}^m_{j+k} - \mathbb{D}^m_{j}
    \end{equation}
    is a rational function of $m$, and all the frequencies involved lie in $\Zstar$ since $R_k$ excludes $-k$.

    Reading $W_{j,k}$ as a vector of $\RR^G$ through $m \mapsto W_{j,k}(m)$, the right sides of~\eqref{eq:lin_nondeg} and~\eqref{eq:lin_sep} are the pairings $\langle \alpha, W_{j,k}\rangle$ and $\langle \alpha, W_{j_1,k} - W_{j_2,k}\rangle$, so that $\mathbb{D}$ fails to be distinct for band $k$ exactly when $\alpha$ lies in
    \begin{equation*}
      N_k \eqdef
      \bigcup_{j \in R_k} \bigl( W_{j,k} \bigr)^\perp
      \ \cup
      \bigcup_{\substack{j_1, j_2 \in R_k \\ j_1 \neq j_2}} \bigl( W_{j_1,k} - W_{j_2,k} \bigr)^\perp ,
      \qquad
      W^\perp \eqdef \bigl\{ \alpha \in \RR^G : \langle \alpha, W \rangle = 0 \bigr\} .
    \end{equation*}
    Each $W^\perp$ is a hyperplane as soon as $W$ is not identically zero on $G$, and $R_k$ is countable, so in that case $N_k$ is a countable union of hyperplanes and hence has zero Lebesgue measure. Every $\alpha \in \RR^G \setminus N_k$ then gives a $\mathbb{D}$ that is distinct for band $k$. It remains to find a $G$ on which none of $W_{j,k}$ and $W_{j_1,k} - W_{j_2,k}$ vanishes identically.

    For~\eqref{eq:lin_nondeg}, we evaluate $\mathbb{D}^m_l$ along rays $m = t\omega$, $|\omega| = 1$
    \begin{equation}\label{eq:D_asymptotic}
      \mathbb{D}^{t\omega}_l
      = -\frac{4}{t}\, \frac{\langle l^\perp, \omega\rangle^2 (l \cdot \omega)}{|l|^2}
      + O(t^{-3})
      = -\frac{|l|}{t}\, \big( \cos\phi_l - \cos 3\phi_l \big) + O(t^{-3}) ,
      \qquad \phi_l \eqdef \angle(l, \omega) ,
    \end{equation}
    so that, using $|l| \cos\phi_l = \langle l, \omega \rangle$,
    \begin{equation*}
      \lim_{t \to \infty}\, t W_{j,k}(t\omega)
      = -\langle k, \omega \rangle
      + |j+k| \cos 3\phi_{j+k} + |j| \cos 3\phi_{-j}
      \ \not\equiv\ 0 ,
    \end{equation*}
    the right side not vanishing identically because the first harmonic $-\langle k, \omega \rangle$ is non-zero. Therefore $W_{j,k}(m)$ is a non-zero rational function of $m$.

    That first harmonic does not depend on $j$, so it cancels in $W_{j_1,k} - W_{j_2,k}$ and the ray asymptotics give nothing in this case. We appeal instead to \cite{BedrossianPunshon-Smith-Chaos-2024y}*{Proposition~5.1}, in the untruncated form proved in \cite{BedrossianPunshon-Smith-Chaos-2024y}*{Section~5.1}, which concludes that (using the anti-symmetry $D^m_{-i}=-D^m_i$)
    \[
    W_{j_1,k}(m)   - W_{j_2,k}(m) = \mathbb{D}^m_{j_1+k} + \mathbb{D}^m_{-j_1} + \mathbb{D}^m_{-j_2-k} + \mathbb{D}^m_{j_2}
    \] 
    does not vanish identically in $m$ whenever
    \begin{equation*}
      (j_1 + k) + (-j_1) = k \neq 0 , \quad
      (j_1 + k) + (-j_2 - k) = j_1 - j_2 \neq 0 , \quad
      (j_1 + k) + j_2 \neq 0,
    \end{equation*}
    which hold since $k \neq 0$, since $j_1 \neq j_2$, and since $R_k$ contains only one of $j_1$ and $-j_1-k$, so that $j_2 \neq -j_1-k$.

    Finally, we choose $G$. Writing $W_{j_1,k} - W_{j_2,k}$ over a common denominator, we have
    \begin{equation*}
      \bigl( W_{j_1,k} - W_{j_2,k} \bigr)(m) = \frac{P(m)}{|m|^2 \prod_{i=1,2} |m - j_i|^2\, |m + j_i|^2\, |m - j_i - k|^2\, |m + j_i + k|^2}
    \end{equation*}
    with $\deg_{m_1} P \vee \deg_{m_2} P \le 17$. We then take $G$ of product form, $G = A \times B \subset \Zstar$ with $A, B \subset \ZZ$ finite and $0 \notin G$. On such a set a polynomial with partial degrees $\le d_1, d_2$ that is not identically zero is non-zero on at least $(|A| - d_1)(|B| - d_2)$ points. So for $|A|, |B| \ge 22$,
    \begin{equation*}
      | m \in G : P(m) \neq 0 | \ \ge\ (|A| - 17)(|B| - 17) \ \ge\ 25
      \ >\ 8 \ \ge\ | G \cap \{ \pm j_i,\, \pm(j_i + k) : i = 1, 2 \} |,
    \end{equation*}
    the excluded points being those at which the denominator vanishes, $0 \notin G$ aside. Some $m \in G$ therefore has $P(m) \neq 0$ and is not excluded. At such an $m$ the difference $W_{j_1,k} - W_{j_2,k}$ is non-zero. The same argument applies to $W_{j,k}$, which has two terms instead of four and so partial degrees $\le 9$ and at most four excluded points. In neither case do $j_1, j_2, j$ or $k$ enter as exponents, so the degree bounds and the number of excluded points are the same for every band, and the same $G$ works across all bands.
  \end{proof}

\begin{lemma}\label{lem:elementary}
  $\overline{\mathcal{I}}^{\|\cdot\|}$ contains $E^{p,q}$ for all $p \neq q \in \Zstar$, and $E^{p,p} - E^{q,q}$ for all $p, q \in \Zstar$.
\end{lemma}

\begin{proof}
  We first show that $\overline{\mathcal{I}}^{\|\cdot\|}$ contains $M^{p,q}$ for all $p, q \in \Zstar$ with $p + q \neq 0$ and $p \neq q$. Since $M^{-q,-p} = M^{p,q}$ we may assume $q \in R_k$, $k \eqdef p - q$. Let $\mathbb{D} \in \h$ be distinct for band $k$ (Lemma~\ref{lem:distinct}) and abbreviate $\lambda_j = \lambda^k_j$. Since the coefficients of each $\mathbb{D}^m$ lie in $S^{-1}$ (Lemma~\ref{lem:symbol_ideal}), the mean value theorem gives $\sup_j |j|^2 |\lambda_j| < \infty$. For a polynomial $\mathfrak{p}$, the identities \eqref{eq:pair_eigen} and~\eqref{eq:Hk_pairs} give
  \begin{equation}\label{eq:peel_family}
    \mathfrak{p}\big(\ad(\mathbb{D})\big)\, \ad(\mathbb{D})^2\, H^k
    \ =\ \sum_{j \in R_k} \mathfrak{p}(\lambda_j)\, \lambda_j^{2}\, M^{j+k,j}
    \ \eqdef\ T_{\mathfrak{p}} \ \in\ \mathcal{I} ,
  \end{equation}
  the membership because $\mathbb{D} \in \h$ and $\mathcal{I}$ is an ideal containing $\h$. The sum converges absolutely in operator norm, since
  \begin{equation}\label{eq:peel_envelope}
    \bigl\| \mathfrak{p}(\lambda_j)\, \lambda_j^2\, M^{j+k,j} \bigr\|_{\mathrm{op}}
    \ \lesssim_{\mathbb{D}, \mathfrak{p}, k}\ |j|^{-3} ,
    \qquad
    \|M^{j+k,j}\|_{\mathrm{op}} \le \max\bigl( |c_{j,k}|, |c_{j+k,k}| \bigr) \lesssim_k |j| ,
    \quad
    |\lambda_j| \lesssim_{\mathbb{D}} |j|^{-2} ,
  \end{equation}
  which is summable over $\ZZ^2$.

  The set $K \eqdef \{\lambda_j : j \in R_k\} \cup \{0\}$ is compact, with $0$ its only accumulation point, and by distinctness the values $\lambda_j$, $j \in R_k$, are pairwise distinct and non-zero. Hence $\lambda_{q}$ is an isolated point of $K$, so $\lambda \mapsto \lambda_q^{-2}\, \mathbf{1}_{\{\lambda_{q}\}}(\lambda)$ is continuous on $K$. By Stone--Weierstrass choose polynomials $\mathfrak{p}_N$ converging to it uniformly on $K$. Then $\mathfrak{p}_N(\lambda_j)\, \lambda_j^2 \to \delta_{jq}$ for each $j$, while the terms of $T_{\mathfrak{p}_N}$ are dominated, uniformly in $N$, by the summable envelope above, so $T_{\mathfrak{p}_N} \to M^{p,q}$ in operator norm and $M^{p,q} \in \overline{\mathcal{I}}^{\|\cdot\|}$.

  We next pass from basis pairs to elementary matrices. By \cite{BedrossianPunshon-Smith-Chaos-2024y}*{Proposition~4.11}, for $p \neq q \in \Zstar$ with $p + q \neq 0$ there are $k, k' \in \Zstar$ with
  \begin{equation}\label{eq:bps_det}
    \det
    \begin{pmatrix}
      c_{k,p-k}\, c_{q,k-q} & c_{p,p-k}\, c_{k,q-k} \\
      c_{k',p-k'}\, c_{q,k'-q} & c_{p,p-k'}\, c_{k',q-k'}
    \end{pmatrix} \neq 0 .
  \end{equation}

  Now let $p + q \neq 0$. For $k \notin \{\pm p, \pm q\}$ a direct computation with $E^{a,b} E^{c,d} = \delta_{b,c} E^{a,d}$ gives
  \begin{equation}\label{eq:pair_bracket}
    [M^{p,k}, M^{k,q}]
    = c_{k,p-k}\, c_{q,k-q}\, E^{p,q}
    +\ c_{p,p-k}\, c_{k,q-k}\, E^{-q,-p} ,
  \end{equation}
    the sign of the second term coming from $c_{k,k-q} = -c_{k,q-k}$. Any $k, k'$ satisfying~\eqref{eq:bps_det} must lie outside $\{\pm p, \pm q\}$, since at $k \in \{p, q\}$ the determinant has a pole and at $k \in \{-p, -q\}$ we have $c_{\pm p, 2p} = 0$ or $c_{\pm q, 2q} = 0$. The first step gives $M^{a,b} \in \overline{\mathcal{I}}$ whenever $a \neq b$ and $a + b \neq 0$, which for the four pairs $M^{p,k}, M^{k,q}, M^{p,k'}, M^{k',q}$ asks exactly that $k, k' \notin \{\pm p, \pm q\}$. All four therefore lie in $\overline{\mathcal{I}}$, and~\eqref{eq:bps_det} allows us to solve the $2\times 2$ system that the two brackets give for $E^{p,q}$ and $E^{-q,-p}$.

    For $q = -p$, pick any $s \notin \{0, p, -p\}$. Then $E^{p,s}, E^{s,-p} \in \overline{\mathcal{I}}$ by the previous step, and
    \begin{equation*}
      E^{p,-p} = [E^{p,s},\, E^{s,-p}] \in \overline{\mathcal{I}} ,
    \end{equation*}
  since $\overline{\mathcal{I}}$ is closed under brackets of bounded operators. Finally $[E^{p,q}, E^{q,p}] = E^{p,p} - E^{q,q}$.
\end{proof}

\begin{proof}[Proof of Proposition~\ref{prop:lie_algebra_dense}]
  The inclusion $\overline{\mathcal{I}} \subseteq \mathcal{K}(\mH)$ is proven in Lemma~\ref{lem:symbol_ideal}. For the converse, by Lemma~\ref{lem:elementary} and
  \begin{equation*}
    E^{p,p} = \lim_{n \to \infty}
    \Big( E^{p,p} - \frac{1}{n} \sum_{i \in \Lambda_n} E^{i,i} \Big) ,
    \qquad \Lambda_n \subset \Zstar \setminus \{p\},\ |\Lambda_n| = n ,
  \end{equation*}
  the closure $\overline{\mathcal{I}}$ contains every elementary matrix $E^{p,q}$. Their span contains all operators with finitely many non-zero Fourier entries, whose closure is $\mathcal{K}(\mH)$. Here, as everywhere in this section, spans and the ideal $\mathcal{I}$ are taken over $\CC$, matching the complex Hilbert space $\mH$, so the span of the elementary matrices is dense in the whole of $\mathcal{K}(\mH)$ and not merely in a real form of it.
\end{proof}

With the density theorem in hand, the transitivity of the linearized Navier--Stokes dynamics follows.

\begin{proof}[Proof of Theorem~\ref{thm:transitivity}, the case of LNS]
  The map
  \begin{equation*}
    A  \mapsto \langle \delta\pi, \vf{A}\pi_T \rangle
    = \langle \delta\pi, A \pi_T \rangle ,
  \end{equation*}
  where the projection orthogonal to $\pi$ was dropped since $\delta \pi \perp \pi_T$, is linear in $A$ and continuous in the operator norm of $\mL(\mH; \mH)$. By hypothesis the quantity vanishes on $\mf{g}_{\mathrm{LNS}} \supseteq \mathcal{I}$. Hence it vanishes also for all $A \in \overline{\mathcal{I}}^{\|\cdot\|_{\mathrm{op}}} = \mathcal{K}(\mH)$ by Proposition~\ref{prop:lie_algebra_dense}. Applying this to the rank-one operators $h \mapsto \langle \pi_T, h \rangle\, u$ with $u \in \mH$ gives $\langle \delta\pi, u \rangle = 0$ for every $u$, so $\delta\pi = 0$.
\end{proof}

\subsection{Transitivity of the passive scalar dynamics}\label{sec:tran-PSA}

The aim of this section is the proof of Theorem~\ref{thm:transitivity} in the case of passive scalar advection. Therefore, in this section we consider
\begin{equation}\label{eq:Lk_def}
  H^{k} = -[\partial_{z_k}, G_w]
  = B(e_k, \cdot) .
\end{equation}
Since the velocity field with vorticity $e_k$ is the shear $K * e_k = i k^\perp
e_k / |k|^2$, we rewrite $H^k$ as
\begin{equation}\label{eq:psa_shear}
  H^k = \frac{i}{|k|^2}\, e_k\, (k^\perp \cdot \nabla) .
\end{equation}
In terms of Fourier coefficients we have:
\begin{equation}\label{eq:psa_generators}
  \Delta e_j = -|j|^2 e_j ,
  \qquad
  H^k e_j = c^{\mathrm{ps}}_{j,k} e_{j+k} ,
  \qquad
  c^{\mathrm{ps}}_{j,k} \eqdef \frac{\langle j^\perp, k\rangle}{|k|^2} ,
\end{equation}
and a direct calculation from $|j|^2 - |j+k|^2 = -2 j\cdot k - |k|^2$ gives
\begin{equation}\label{eq:DeltaL_bracket}
  [\Delta, H^k] e_j = P_k(j) c^{\mathrm{ps}}_{j,k} e_{j+k},
  \qquad
  P_k(j) \eqdef -2 j\cdot k - |k|^2 .
\end{equation}
Moreover, we note that $\mathrm{Lie}(H^k \colon k \in \Zstar) = \mathrm{Span}(
H^k \colon k \in \Zstar)$, since
\begin{equation}\label{eq:psa_bracket_closed}
  [H^{k}, H^{j}]
  = \frac{\langle j^\perp, k \rangle |k+j|^2}{|k|^2 |j|^2}
  H^{{k+j}} ,
  \qquad\text{and}\qquad
  [H^{k}, H^{-k}] = 0 ,
\end{equation}
which is the main difficulty of PSA over LNS.
We proceed by obtaining the analogue of Lemma~\ref{lem:lns_algebra}.

\begin{lemma}\label{lem:psa_algebra}
  In the case of PSA the generators are band-limited Fourier multipliers,
  \begin{equation}\label{eq:psa_Hk_symbol}
    H^k = e_k\, g_k(D) ,
    \qquad
    g_k(l) = \frac{\langle l^\perp, k\rangle}{|k|^2} ,
    \qquad
    |g_k(l)| \le \frac{|l|}{|k|} ,
  \end{equation}
  so that $g_k \in S^1$. Moreover
  \begin{equation}\label{eq:psa_A_symbols}
    \begin{aligned}
      [\Delta, H^k] & = e_k\, q_k(D) ,
      \quad &q_k(l) && = P_k(l)\, g_k(l) , \\
      \bigl[[\Delta, H^k], H^{-k}\bigr] & = d_k(D) ,
      \quad &d_k(l) && = -\frac{2\langle l^\perp, k\rangle^2}{|k|^2} ,
    \end{aligned}
  \end{equation}
  with $P_k$ as in~\eqref{eq:DeltaL_bracket}, and $q_k, d_k \in S^2$. More generally every $A \in \mf{g}_{\mathrm{PSA}}$ is a finite sum $A = \sum_{m \in F_A} e_m\, p_{A,m}(D)$ with $p_{A,m} \in S^2$. Any such $A$ belongs to $\Opbeta$ and is bracket-regular in the sense of Definition~\ref{def:bracket_regular}.
\end{lemma}

\begin{proof}
  The identity~\eqref{eq:psa_Hk_symbol} is~\eqref{eq:psa_generators} read as a multiplier, and $g_k$ is linear in $l$, hence in $S^1$. The first identity in~\eqref{eq:psa_A_symbols} is~\eqref{eq:DeltaL_bracket} and the second is the content of Lemma~\ref{lem:E_eigenvalue}, which also shows that $[[\Delta, H^k], H^{-k}]$ is diagonal. Both $q_k$ and $d_k$ are quadratic polynomials in $l$, so they lie in $S^2$. That a general $A \in \mf{g}_{\mathrm{PSA}}$ has this form follows from Lemma~\ref{lem:band_calculus} at $r = 2$ and $s = 1$. Explicitly, bracketing a band-limited $S^2$ multiplier with $H^\ell = e_\ell\, g_\ell(D)$ produces
  \begin{equation*}
    [\, e_m\, p(D),\ e_\ell\, g_\ell(D) \,]
    = e_{m+\ell}\Bigl( \bigl(p(\cdot+\ell) - p\bigr) g_\ell
    - \bigl(g_\ell(\cdot+m) - g_\ell\bigr) p \Bigr)(D) ,
  \end{equation*}
  where $g_\ell(\cdot+m) - g_\ell$ is in fact constant, since $g_\ell$ is linear. That such an operator belongs to $\Opbeta$ follows as in Lemma~\ref{lem:lns_algebra}.

  We now fix $A = \sum_{m \in F_A} e_m\, p_{A,m}(D)$ and verify the three conditions of Definition~\ref{def:bracket_regular}. For the first, \eqref{eq:band_bracket} gives
  \begin{equation*}
    [H^k, A] = \sum_{m \in F_A} e_{k+m}
    \Bigl( \bigl(g_k(\cdot+m) - g_k\bigr) p_{A,m}
    - \bigl(p_{A,m}(\cdot+k) - p_{A,m}\bigr) g_k \Bigr)(D) .
  \end{equation*}
  Since $g_k$ is linear, $g_k(l+m) - g_k(l) = \langle m^\perp, k\rangle/|k|^2$ is a constant of modulus at most $|m|/|k|$, so the first product is bounded by $C_A (1+|l|)^2$, uniformly in $k$. For the second, $p_{A,m}$ is quadratic, so $|p_{A,m}(l+k) - p_{A,m}(l)| \lesssim_A |k|(1+|l|) + |k|^2$, and multiplying by $|g_k(l)| \le |l|/|k|$ gives an upper bound of order $ (1+|l|)^2 + |k|(1+|l|)$. Altogether the symbol of $[H^k, A]$ is bounded by $C_A(1+|k|)(1+|l|)^2$, whence
  \begin{equation}\label{eq:psa_bracket_growth}
    \|[H^k, A]\|_{\Opbeta} \lesssim_A |k| ,
  \end{equation}
  using $\beta_\star \ge 2$ once more. In particular $\sup_k |k|^{-2}\|[H^k,
  A]\|_{\Opbeta} < \infty$. The remaining two hypotheses are verified exactly as
  in Lemma~\ref{lem:lns_algebra}: only the exponents change. We omit the
  proof
  for the sake of brevity.
\end{proof}

We now move to the main aim of this section, which is to prove that $\{ A
\zeta \colon A \in \mf{g}_{\mathrm{PSA}}\}$ is dense in $L^2$, provided that
$\zeta$ is sufficiently generic. The genericity that we impose is that $\zeta$
is not one-dimensional on any open patch of the domain.
A function $\zeta \colon \TT^2 \to \RR$ is
\emph{one-dimensional} on an open set $W \subseteq \TT^2$ if $\zeta(x) =
f(\ell \cdot x)$ for all $x \in W$, for some $f \colon \RR \to \RR$ and some
$\ell \in \RR^2 \setminus \{0\}$. We say that $\zeta$ is
nowhere one-dimensional if it is one-dimensional on no open set.
We also say that $\zeta$ is globally one-dimensional if $\zeta(x) = f(\ell \cdot
x)$ on all of $\TT^2$ for some $\ell \in \Zstar$ and some $2\pi$-periodic
function $f$. We note that $\zeta$ is globally one-dimensional if and only if $\mathrm{supp}(\hat\zeta)$ lies on a
single line $\RR\ell$ through the origin (the local notion has no such
characterization).

In this setting the main results of the present section are two. First we prove the algebraic and analytic fact that $\mf{g}_{\mathrm{PSA}}$ acts transitively on states that are nowhere one-dimensional. Then we prove that at any fixed time, the solution to the passive scalar advection equation driven by our stochastic vorticity is nowhere one-dimensional, $\PP$--almost surely.

\begin{proposition}
  \label{prop:transitivity_PSA} Let $s > 4$ and let $\zeta \in H^s(\TT^2)$ be nowhere one-dimensional. If $v \in H^s(\TT^2)$ satisfies $\langle v, A\zeta\rangle = 0$ for all $A \in \mf{g}_{\mathrm{PSA}}$, then $v = 0$.
\end{proposition}

\begin{proposition}\label{prop:psa_star_fixedtime}
  Fix $T > 0$ and let $(\zeta_t)_{t \geq 0}$ be the solution to \eqref{eq:linearized_abstract} in the case of PSA \eqref{eq:two_operators_intro_PS} with $\zeta_0 \neq 0, \zeta_0 \in \mH$. Then $$\PP\bigl(\zeta_T \text{ is one-dimensional on some open set} \bigr) = 0 .$$
\end{proposition}
The proof of Proposition~\ref{prop:transitivity_PSA} can be found at the end of Section~\ref{sec:reduc_noncoll}. The proof of Proposition~\ref{prop:psa_star_fixedtime} can be found at the end of Section~\ref{sec:as_fixed_t}. Before we proceed with those proofs however, let us complete the proof of Theorem~\ref{thm:transitivity} in the case of PSA.

\begin{proof}[Proof of Theorem~\ref{thm:transitivity}, in the case of PSA]
  We fix $\Omega_0$ the set of full probability on which $\pi_T$ is nowhere
  one-dimensional, which exists by Proposition~\ref{prop:psa_star_fixedtime}.
  Then the claim follows by
  Proposition~\ref{prop:transitivity_PSA}, given the regularity of $\pi_T \in
  H^{\beta_\star}$ and the assumption $\delta \pi \in H^5$.
\end{proof}

\subsubsection{Transitivity for nowhere one-dimensional states}
\label{sec:reduc_noncoll}
We first record the result of bracketing with the Laplacian,
which leads to explicit degenerate elliptic operators.  Here it is convenient
to note that for every polynomial
$p$, since $\nabla (e_k u) = e_k (\nabla + ik) u$
\begin{equation}\label{eq:psa_shift}
  p(\nabla)\, e_k = e_k\, p(\nabla + ik).
\end{equation}

\begin{lemma}\label{lem:E_eigenvalue}
  The following hold:
  \begin{enumerate}[label=(\roman*)]
    \item For every $k \in \Zstar$,
      \begin{align}
        [\Delta, H^k]
        &= -|k|^2 H^k - \frac{2}{|k|^2}\, e_k\, (k \cdot \nabla)(k^\perp \cdot \nabla) ,
        \label{eq:psa_first_bracket}\\
        \bigl[[\Delta, H^k], H^{-k}\bigr]
        &= \frac{2}{|k|^2}\, (k^\perp \cdot \nabla)^2 .
        \label{eq:psa_double_bracket}
      \end{align}
    \item Hence $\mf{g}_{\mathrm{PSA}}$ contains the
      operators $ \partial_a \partial_b$, and every bracket $[\partial_a
      \partial_b, H^\ell]$, $\ell \in \Zstar$ for all $a,b \in \{1,2\}$.
  \end{enumerate}
\end{lemma}
\begin{proof}
  By~\eqref{eq:psa_shift} we find that  $[\Delta, e_k\, p(\nabla)] = e_k\,
  (2ik\cdot\nabla - |k|^2)\, p(\nabla)$. Together with $p(\nabla) = -\tfrac{i}{|k|^2}
  (k^\perp\cdot\nabla)$ this is~\eqref{eq:psa_first_bracket}. In the second
  bracket the $H^k$ term of~\eqref{eq:psa_first_bracket} drops
  by~\eqref{eq:psa_bracket_closed}, and in the remaining bracket only
  $(k\cdot\nabla)$ sees the modulation $e_{-k}$. Again by~\eqref{eq:psa_shift},
  with $H^{-k} = \tfrac{i}{|k|^2} e_{-k}(k^\perp\cdot\nabla)$, we have
  \begin{equation*}
    \Bigl[ \frac{2}{|k|^2} e_k (k\cdot\nabla)(k^\perp\cdot\nabla),\
    \frac{i}{|k|^2} e_{-k} (k^\perp\cdot\nabla) \Bigr]
    = \frac{2i}{|k|^4}
    \bigl( (k\cdot\nabla - i|k|^2) - (k\cdot\nabla) \bigr) (k^\perp\cdot\nabla)^2
    = \frac{2}{|k|^2}\, (k^\perp\cdot\nabla)^2 .
  \end{equation*}
  For the second assertion, we note that
  \begin{equation}\label{eq:square_span}
    \partial_1^2 = (k^\perp\cdot\nabla)^2\big|_{k = (0,-1)} , \qquad
    \partial_2^2 = (k^\perp\cdot\nabla)^2\big|_{k = (1,0)} , \qquad
    2\,\partial_1\partial_2 = (\partial_1 + \partial_2)^2 - \partial_1^2 - \partial_2^2 ,
  \end{equation}
  with $(\partial_1 + \partial_2)^2 = (k^\perp\cdot\nabla)^2|_{k = (1,-1)}$,
  which concludes the proof.
\end{proof}
Now we can recast the entire problem in terms of partial differential operators.

\begin{lemma}\label{lem:psa_duality} The following hold:
  \begin{enumerate}[label=(\roman*)]
    \item Let $v, \varphi \colon \TT^2 \to \RR$ be $C^1$. Then, for every $k \in \Zstar$,
      \begin{equation}\label{eq:curl_coeff}
        \langle v, H^k \varphi\rangle
        = -\frac{1}{|k|^2}\, \widehat{\curl (v \nabla \varphi)}(-k) ,
        \qquad
        \curl (v\nabla \varphi) = -\nabla^\perp v \cdot \nabla \varphi .
      \end{equation}
    \item Let $\zeta, v \in H^s(\TT^2)$ for $s > 4$. Under the
      hypothesis of Proposition~\ref{prop:transitivity_PSA} it holds that
      \begin{equation}\label{eq:psa_hessian}
        \nabla^\perp v \cdot \nabla \zeta = 0
        \qquad\text{and}\qquad
        \nabla^2 \zeta\ \nabla^\perp v = 0
        \qquad \text{on } \TT^2 .
      \end{equation}
  \end{enumerate}
\end{lemma}

\begin{proof}
  For (i), with $F = v \nabla \varphi$, and by using~\eqref{eq:psa_shear},
  we obtain
  \begin{equation*}
    \langle v, H^k \varphi\rangle
    = -\frac{i}{|k|^2} \int_{\TT^2} e_k\, v\, (k^\perp \cdot \nabla \varphi)\, \ud x
    = -\frac{i}{|k|^2}\, k^\perp \cdot \widehat{F}(-k)
    = -\frac{1}{|k|^2}\, \widehat{\curl F}(-k) ,
  \end{equation*}
  using $\widehat{\curl F}(m) = -i\, m^\perp \cdot \widehat F(m)$ in the
  last step. The Leibniz expansion of $\curl(v
  \nabla \varphi)$ leaves $-\nabla^\perp v \cdot \nabla \varphi$, since
  $\curl \nabla \varphi = 0$.

  For (ii), the first claim follows directly from part (i) with $\varphi =
  \zeta$, since vanishing of the coefficients~\eqref{eq:curl_coeff} over all
  $k \in \Zstar$ is equivalent to $\curl F \equiv 0$.
  For the
  second claim, constant-coefficient operators commute with
  $(\ell^\perp\cdot\nabla)$, so in the bracket $[\partial_a\partial_b,
  H^\ell]$ we obtain
  \begin{equation}\label{eq:dd_bracket}
    [\partial_a \partial_b,\, H^\ell]
    = -\ell_a \ell_b\, H^\ell
    + i \ell_a\, H^\ell \partial_b + i \ell_b\, H^\ell \partial_a ,
    \qquad
    H^\ell \partial_c = \frac{i}{|\ell|^2}\, e_\ell\, \partial_c\,
    (\ell^\perp \cdot \nabla) .
  \end{equation}
  The left side lies in $\mf{g}_{\mathrm{PSA}}$ by
  Lemma~\ref{lem:E_eigenvalue}, so $\langle v, [\partial_a\partial_b,
  H^\ell]\zeta\rangle = 0$. Combined with $\langle v, H^\ell \zeta\rangle =
  0$, we deduce that the numbers $u_c \eqdef \langle v, H^\ell \partial_c \zeta\rangle$
  satisfy
  \begin{equation*}
    \ell_a u_b + \ell_b u_a = 0 , \qquad a, b \in \{1, 2\} ,
  \end{equation*}
  that is $2\ell_1 u_1 = 0$, $2\ell_2 u_2 = 0$ and $\ell_1 u_2 + \ell_2 u_1 = 0$. Since $\ell \neq 0$ these force $u_1 = u_2 = 0$. As $H^\ell \partial_c \zeta = H^\ell (\partial_c \zeta)$ and $\partial_c \zeta \in C^2$, part (i) applied with $\varphi = \partial_c \zeta$ turns this into $\nabla^\perp v \cdot \nabla (\partial_c \zeta) = 0$ on $\TT^2$, which is the $c$-th component of the second equation in~\eqref{eq:psa_hessian}.
\end{proof}

\begin{lemma}\label{lem:psa_rigidity}
  Let $v, \zeta \in C^2(\TT^2)$ satisfy $\nabla^\perp v \cdot \nabla\zeta = 0$ and $\nabla^2\zeta\ \nabla^\perp v = 0$ on $\TT^2$. If $v$ is not constant, then $\zeta$ is one-dimensional on some open set.
\end{lemma}

\begin{proof}
  Pick a ball on which $\nabla v \neq 0$. If $\nabla\zeta$ vanishes identically on one of its open subsets, then $\zeta$ is constant and hence also one-dimensional on that open set. Otherwise shrink to a ball $U$ of radius smaller than $1$ on which both $\nabla v \neq 0$ and $\nabla \zeta \neq 0$. Such a ball is convex and simply connected, because we are on the torus.

    On $U$ the first hypothesis makes $\nabla\zeta$ parallel to $\nabla v$, and therefore $\nabla^\perp\zeta$ parallel to $\nabla^\perp v$, so that the second hypothesis becomes a condition on $\zeta$ alone,
    \begin{equation*}
      \nabla^2\zeta\, \nabla^\perp\zeta = 0 \qquad \text{on } U .
    \end{equation*}
    We write $\nabla\zeta = \rho\, n$ with $\rho \eqdef |\nabla\zeta| > 0$ and $n \eqdef (\cos\omega, \sin\omega)$, so that $\nabla^\perp\zeta = \rho\, n^\perp$ and the condition above reads $\nabla^2\zeta\, n^\perp = 0$. The angle $\omega$ has a $C^1$ branch on $U$, because $\nabla\zeta$ does not vanish there. Then $\partial_j n = n^\perp \partial_j \omega$, so that
    \begin{equation*}
      \nabla^2\zeta = n \otimes \nabla\rho + \rho\, n^\perp \otimes \nabla\omega.
    \end{equation*}
    Since $n$ and $n^\perp$ are independent, the vanishing of $\nabla^2\zeta\, n^\perp$ is equivalent to
    \begin{equation}\label{eq:psa_rigidity_eq1}
      \partial_{n^\perp}\rho = 0 , \qquad \partial_{n^\perp}\omega = 0
      \qquad \text{on } U.
    \end{equation}

    Next, since $\nabla^2\zeta$ is symmetric, \eqref{eq:psa_rigidity_eq1} gives
    \[
      0 = \partial_{n^\perp}\rho
      = \langle n, \nabla^2\zeta\, n^\perp\rangle
      = \langle n^\perp, \nabla^2\zeta\, n\rangle
      = \rho\, \partial_n\omega.
    \]
  Hence $\partial_n\omega = 0$ as well and therefore $\nabla\omega$ vanishes entirely on $U$. This means that $\omega$ is constant and the direction $n$ is a constant unit vector on $U$. Since $\nabla\zeta = \rho\, n$, the scalar is constant along $n^\perp$, and because $U$ is convex this gives $\zeta(x) = f(n \cdot x)$ on $U$, so that $\zeta$ is one-dimensional there.
\end{proof}

\begin{remark}\label{rem:steady_euler}
  Lemma~\ref{lem:psa_rigidity} has a natural geometric interpretation. On any open set where $\nabla v, \nabla \zeta \neq 0$, the Poisson bracket condition $\nabla^\perp v \cdot \nabla \zeta = 0$ gives $\nabla v \parallel \nabla \zeta$, which reduces the second hypothesis to
  \begin{equation*}
    \nabla^2\zeta \, \nabla^\perp\zeta = 0 .
  \end{equation*}
  This states that the directional derivative of $\nabla\zeta$ along its own level curves vanishes. Geometrically, this forces the level curves to be straight lines along which $|\nabla\zeta|$ is constant. By symmetry of the Hessian $\nabla^2\zeta$, this direction cannot rotate across neighboring level curves either, forcing $\zeta$ to be a one-dimensional profile
  \begin{equation*}
    \zeta(x) = f(n \cdot x)
  \end{equation*}
  locally for a constant unit vector $n \in \mathbb{S}^1$.

  In contrast to the Hartman--Nirenberg cylinder theorem \cite{HartmanNirenberg59}, where $\det(\nabla^2\zeta) = 0$ requires global completeness to deduce a cylindrical graph, the explicit alignment of $\ker(\nabla^2\zeta)$ with the level curves yields local 1D rigidity on any convex set without global assumptions. Analogous rigidity mechanisms appear in 2D hydrodynamics, where the steady Euler relation $\{\psi, \omega\} = 0$ (the setting $\omega = \Delta\psi$) paired with geometric or monotonicity constraints forces steady flows to reduce to 1D shear flows \cites{HamelNadirashvili19, ConstantinDrivasGinsberg21, ElgindiHuangSaidXie, CotiZelatiElgindiWidmayer23}.
\end{remark}

\begin{proof}[Proof of Proposition~\ref{prop:transitivity_PSA}]
  We apply Lemma~\ref{lem:psa_duality}(ii), which in turn allows us to apply
  Lemma~\ref{lem:psa_rigidity}. Since $\zeta$ is nowhere one-dimensional,
  that lemma forces $v$ to be constant, and since its mean is zero we deduce
  that $v \equiv 0$.
\end{proof}
\subsubsection{Almost sure absence of one-dimensional states}
\label{sec:as_fixed_t}

This section proves Proposition~\ref{prop:psa_star_fixedtime}. The first step
is to rewrite the condition of being one-dimensional as a polynomial
functional of the scalar. Since the scalar is driven by some noise that is
Gaussian at small times, we want to harness this to prove that
one-dimensionality cannot happen with probability one.
For an open box $B \subset \TT^2$ and $\zeta \in H^1(\TT^2)$ set
\begin{equation}\label{eq:structure_tensor}
  \mathcal{S}_B(\zeta) \eqdef \int_B \nabla\zeta(x)\otimes\nabla\zeta(x)\ud x ,
  \qquad
  D_B(\zeta) \eqdef \det \mathcal{S}_B(\zeta) \geq 0 .
\end{equation}

\begin{lemma}\label{lem:tensor}
  Let $B \subset \TT^2$ be open and convex, $\zeta \in C^1(\TT^2)$. Then $D_B(\zeta) = 0$ if and only if $\zeta$ is one-dimensional on $B$.
\end{lemma}

\begin{proof}
  For a unit vector $v$ one has \[ \langle v, \mathcal{S}_B(\zeta)\, v \rangle = \int_B \langle v, \nabla\zeta\rangle^2 \ud x \] and $\mathcal{S}_B(\zeta)$ is symmetric positive semidefinite, so $D_B(\zeta) = 0$ if and only if $\langle v, \nabla\zeta\rangle$ vanishes identically on $B$ for some unit $v$, the integrand being continuous and nonnegative. If $\zeta$ is one-dimensional on $B$ this holds with $v = \ell^\perp/|\ell|$, since $\nabla\zeta = f'(\ell\cdot x)\,\ell$. Conversely a vanishing derivative in the direction $v$ throughout the convex set $B$ makes $\zeta$ a function of the orthogonal coordinate alone, $\zeta(x) = f(\ell \cdot x)$ on $B$ with $\ell = v^\perp$.
\end{proof}

Let $\mB$ be the countable family of open balls in $\TT^2$ with rational
centers and rational radius smaller than $1$, so that every $B \in \mB$ is
convex (because we are on the torus). We have that $\zeta_T$ is one-dimensional on some open set if and
only if $D_B(\zeta_T) = 0$ for some $B \in \mB$, by Lemma~\ref{lem:tensor}. To
prove our main result, it suffices therefore to prove that for any $T>0$
\begin{equation}\label{eq:target_star}
  \PP\bigl(D_B(\zeta_T) = 0\bigr) = 0,
  \qquad \forall B \in \mB .
\end{equation}
To do so, we proceed in two steps. First we use a Cameron--Martin zero-one law to conclude that
\begin{equation*}
  \PP(D_B(\zeta_T) = 0 ) \in \{0, 1\}.
\end{equation*}
This result will rely on the analyticity of the solution map with respect to Cameron--Martin controls. The second step is then to solve a control problem to prove that $\PP(D_B(\zeta_T) > 0) > 0$, from which we immediately deduce \eqref{eq:target_star}.

Throughout the remainder of this section we fix, without loss of generality, the underlying probability space $(\Omega, \mF, \PP)$ to be the Wiener space, which we now make precise. Our driving noise $W = \sum_{k\in\Zstar} e_k W_k$ is space-time white. Therefore, fixing an arbitrary $r_\star > 1$ and time horizon $T>0$, it takes values in
\begin{equation*}
  \Omega = C\bigl([0,T]; H^{-r_\star}(\TT^2)\bigr).
\end{equation*}
We then let $\PP$ be the law of $W$ on $\Omega$ and $\mF$ be the Borel sigma algebra. The Cameron--Martin space of $(\Omega, \PP)$ is then the space of controls
\begin{equation*}
  \mf{H} = \Bigl\{ h \in \Omega : h(0) = 0 ,\ \dot h \in L^2\bigl([0,T]; \mH\bigr) \Bigr\} ,
  \qquad \| h \|_{\mf{H}}^2 = \int_0^T \| \dot h(s) \|_{\mH}^2\, \ud s .
\end{equation*}
The space $\mf{H}$ is not the space $\CM = L^2([0,\infty); \mH)$ of the other
sections of this work (the latter is the space of all time derivatives of the former), for convenience in the calculations below. Let us fix an orthonormal basis $\{h_n\}_{n\geq1}$ of smooth
functions for $\mf{H}$ and write $\mf{H}^n =
\operatorname{span}(h_1,\dots,h_n)$ and $\mf{H}^\infty = \bigcup_n \mf{H}^n$,
a dense linear subspace.
\begin{definition}\label{def:cm_analytic}
  A measurable functional $F \colon \Omega \to \RR$ is Cameron--Martin analytic if there exists an $\Omega_0$ with $\PP(\Omega_0)=1$ such that for every $\omega \in \Omega_0$ and every $n \in \NN_*$ and $h_i$ the basis elements of $\mf{H}$ as above, the map $(x_1,\dots,x_n) \mapsto F(\omega + \sum_i x_i h_i)$ is analytic on $\R^n$.
\end{definition}
\begin{lemma}\label{lem:analytic}
  Fix $T > 0$ and any open $B \subseteq \TT^2$. The map $\Omega \to \RR$
  defined by $\omega \mapsto D_B( \zeta_T (\omega))$, is Cameron--Martin
  analytic (recall that $\Omega$ depends on $T$).
\end{lemma}
The proof passes through the analytic dependence of the vorticity equation on
its controls, and it can be found in Appendix~\ref{app:analyticity}. Note that
this result is roughly \cite[Theorem~2.1]{Kuksin82} (or in any case a somewhat
expected result), only that there the authors
establish analyticity of Navier--Stokes, and we have to push this forward to
analyticity of the passive scalar.

Next, the zero set of Cameron--Martin analytic functions satisfies a zero-one
law. This is the analytic analogue of the zero-one law
of~\cite[Proposition~5.10.10]{Bogachev1998} for polynomial functions and it
rests on the
Cameron--Martin zero--one law
\cite[Theorem~2.5.2]{Bogachev1998}. The proof is identical after replacing polynomials by analytic functions.
\begin{lemma}\label{lem:dichotomy}
  If $F$ is Cameron--Martin analytic, then $\PP(F = 0) \in \{0,1\}$.
\end{lemma}

\begin{proof}
  Write $N = \{F = 0\}$ and fix $n \in \NN$. We can decompose
  an event $\omega \in \Omega$ as
  \begin{equation*}
    \omega = \sum_{i \le n} \xi_i h_i + \omega_n^\perp ,
    \qquad
    \xi_i \eqdef \int_0^T \langle \dot h_i(r), \ud W_r \rangle ,
  \end{equation*}
  which factorizes the Wiener measure as $\PP = \gamma_n \otimes \PP^\perp$,
  with $\xi \sim \gamma_n$ a standard Gaussian vector independent of the
  law $\PP^\perp$ of the remainder $\omega_n^\perp$. By Cameron--Martin analyticity,
  for $\PP^\perp$-almost every $\omega_n^\perp$ the map $ y \mapsto
  F(\sum_{i \le n} y_i h_i + \omega_n^\perp)$ is analytic on $\RR^n$. A real
  analytic function on $\RR^n$ either vanishes identically or its zero set
  has zero Lebesgue measure, so $\PP(N \mid
  \omega_n^\perp) = \mathbf 1_{A^{(n)}}$, where $A^{(n)}$ is the event that
  the map vanishes identically. Since $\PP(N) = \PP(A^{(n)})$ for every $n$,
  and the events $A^{(n)}$ decrease to
  \begin{equation*}
    A^* \eqdef \bigl\{ F(\omega + h) = 0 \ \text{for all } h \in \mf{H}^\infty \bigr\} ,
  \end{equation*}
  we obtain $\PP(N) = \PP(A^*)$. The set of $h \in \mf{H}$ for which $\mathbf
  1_{A^* + h} = \mathbf 1_{A^*}$ is almost surely closed, because the map $h
  \mapsto \mathbf 1_{A^* + h}$ is continuous from $\mf{H}$ to $L^1(\PP)$ by
  \cite[Theorem~2.4.8]{Bogachev1998}, and it contains $\mf{H}^\infty$, since
  $A^* + h = A^*$ exactly for $h \in \mf{H}^\infty$. It is therefore all of
  $\mf{H}$, and the Cameron--Martin zero--one law
  \cite[Theorem~2.5.2]{Bogachev1998} gives $\PP(N) = \PP(A^*) \in \{0,1\}$.
  This concludes the proof.
\end{proof}

\begin{proof}[Proof of Proposition~\ref{prop:psa_star_fixedtime}]
  In view of Lemma~\ref{lem:dichotomy}, applied to $D_B(\zeta_T)$ (which is
  analytic by Lemma~\ref{lem:analytic}),  and
  Lemma~\ref{lem:tensor}, the proposition follows if we can prove that
  \begin{equation*}
    \PP(D_B(\zeta_T) > 0 ) > 0
  \end{equation*}
  for any open convex $B$. This is a control problem, and by the upcoming Lemma~\ref{lem:tube} it is enough
  to exhibit a control $g \in \CM$ whose skeleton satisfies
  $D_B(\zeta^g_T) > 0$.

  To construct such a control we proceed through somewhat standard steps.
  First, we fix a $\delta \in (0, T/2)$ and steer the vorticity from $w_0$ to
  $0$ over $[0,\delta]$. Note that the initial data is only in
  $\mH$. Let $w^{\mathrm{f}}$ denote the unforced solution started from $w_0$,
  which is smooth at any positive time. On $[0, \delta/2]$ take $g = 0$, so
  the skeleton vorticity is $w^{\mathrm{f}}$. On $[\delta/2, \delta]$
  prescribe the path $\bar w_t \eqdef \psi(t)\, w^{\mathrm{f}}_t$, where
  $\psi$ is smooth with $\psi(\delta/2) = 1$ and $\psi(\delta) = 0$. The
  control this requires is
  \begin{equation}\label{eq:steer_control}
    Q^{1/2} g_t = \psi'(t)\, w^{\mathrm{f}}_t
    + \bigl(\psi(t)^2 - \psi(t)\bigr) B(w^{\mathrm{f}}_t, w^{\mathrm{f}}_t) ,
  \end{equation}
  whose right side is smooth in space, so that $g \in \CM$: the symbol
  $\sigma_k^{-1}$ of $Q^{-1/2}$ grows polynomially by
  Assumption~\ref{ass:noise}, while the Fourier coefficients of a smooth field
  decay faster than any polynomial. Write $\zeta^g_\delta$ for the scalar at
  the end of that window: it is non-zero by backward uniqueness, and it lies in
  $H^1$ by parabolic smoothing.

  {Now let us first suppose that $\zeta^g_\delta$ is not globally
    one-dimensional. Recall from the discussion preceding
    Proposition~\ref{prop:transitivity_PSA} that this says precisely that
    $\mathrm{supp}(\hat\zeta^g_\delta)$ lies on no line $\RR\ell$ with $\ell \in
    \Zstar$. In this case, we extend $g$ by zero on $[\delta,T]$, so that the skeleton vorticity
    stays at $0$ and the scalar evolves by the heat semigroup. The evolution
    $\overline{\zeta}_t = P_{t-\delta}\zeta^g_\delta$ is also not globally
  one-dimensional} for any $t \geq 0$, since Fourier modes are preserved. In
  addition, for any $t > 0 $, the function $x \mapsto \overline{\zeta}_t(x)$
  is analytic on $\TT^2$. Suppose that $D_B(\overline{\zeta}_T) =0$, then by
  Lemma~\ref{lem:tensor}, we would have that $\overline{\zeta}_T$ is
  one-dimensional on $B$, meaning that for some unit vector $v
  \in \RR^2$ we have $\langle v,  \nabla
  \overline{\zeta}_T(x) \rangle = 0$ for all $x \in B$. However, the latter is
  an analytic function on $\TT^2$. Therefore, if it vanishes on an open set it
  would vanish on the entire torus, contradicting the fact that
  $\overline{\zeta}_T$ is not globally one-dimensional. Therefore
  $D_B(\overline{\zeta}_T) = D_B(\zeta^g_T) > 0$ for any open $B$, as desired.

  It remains to treat the case in which $\zeta^g_\delta$ is
  globally one-dimensional, that is $\supp(\hat\zeta^g_\delta)
  \subseteq \RR\ell$ for some $\ell \in \Zstar$. For $k \in \Zstar$ define
  $w_k = e_k + e_{-k}$ and $u_k = K * w_k$, and fix a parameter $\ve
  \in (0,1)$. We extend $g$ over $[\delta,T]$ so that the skeleton vorticity
  becomes the fixed shear $\ve w_k$. Fix some $\chi$ smooth and non-decreasing, $\chi = 0$ near
  $\delta$ and $\chi = 1$ on $[\delta', T]$ for some $\delta' \in (\delta,
  T)$, we prescribe the path $\bar w_s \eqdef \ve\, \chi(s)\, w_k$. Shears are
  steady solutions of the Euler equations, $B(w_k, w_k) = 0$, so the control
  this requires is $Q^{1/2} g_s = \ve\bigl(\chi'(s) + \chi(s)|k|^2\bigr) w_k$,
  which indeed leads to a Cameron--Martin control $g$. Under this control the
  scalar solves
  \begin{equation*}
    \partial_s \overline{\zeta} + \ve\, \chi(s)\,u_k\cdot\nabla\overline{\zeta}
    = \Delta\overline{\zeta}, \qquad \overline{\zeta}_\delta = \zeta^g_\delta.
  \end{equation*}
  Now to conclude,
  we claim that there exist a time $t \in (\delta', T)$, an $\ve \in (0, 1)$ and two modes $m_1 \in \RR \ell, m_2 \nparallel \RR \ell$ such that
  \begin{equation} \label{eq:aim-noncoll}
    \hat{\overline{\zeta}}_t (m_1), \ \hat{\overline{\zeta}}_t(m_2) \neq 0,
  \end{equation}
  which would immediately imply that $\overline{\zeta}_t$ is not
  one-dimensional, so we could conclude as above.

  To prove~\eqref{eq:aim-noncoll} fix $t \in (\delta', T)$. For $u_k = K*(e_k + e_{-k}) = \frac{\iota k^\perp}{|k|^2}(e_k - e_{-k})$ we have
  \begin{equation*}
    u_k\cdot\nabla e_m = - \frac{\langle k^\perp, m\rangle}{|k|^2} \bigl(e_{m+k} - e_{m-k}\bigr),
  \end{equation*}
  so $u_k\cdot\nabla$ couples $m$ to $m\pm k$ with coefficient $\mp
  {\langle k^\perp, m\rangle/|k|^2}$. {Here
  $\langle k^\perp, m\rangle \neq 0$ if and only if $k\nparallel m$.} Now pick
  a mode $m_1$ such that {$\hat\zeta^g_\delta(m_1)\neq0$} (which
    is possible since {$\zeta^g_\delta\neq0$}, and necessarily
  satisfies $m_1\in\R\ell$). Choose any $k\nparallel \ell \RR$ and set
  $m_2\eqdef m_1+k$. At first order in $\ve$ the only source feeding $m_2$ is
  $m_1$, as the other potential source $m_2+k = m_1+2k$ also lies off
  $\R\ell$. Therefore, if we define
  $\Phi(\ve)\eqdef\hat{\overline{\zeta}}_t(m_2)$, then $\Phi(0)=e^{-|m_2|^2
  {(t - \delta)}}{\hat\zeta^g_\delta(m_2)}=0$. Moreover, $\Phi$ is differentiable
  at $0$ with
  \begin{equation*}
    \Phi'(0)={\frac{\langle k^\perp, m_1\rangle}{|k|^2}}\,{\hat\zeta^g_\delta(m_1)}
    \int_{\delta}^{t} {\chi(s)}\, e^{-|m_2|^2(t-s)} e^{-|m_1|^2 {(s - \delta)}}\ud s \neq 0 ,
  \end{equation*}
  {where the integral is positive because the integrand is non-negative and $\chi = 1$ on $[\delta', t]$}. Hence, there exists a $\ve_0>0$ with $\hat{\overline{\zeta}}_t(m_2)\neq0$ for all $0<\ve\leq\ve_0$. The same holds for $m_1$, where $\hat{\overline{\zeta}}_t(m_1)=e^{-|m_1|^2 {(t - \delta)}}{\hat\zeta^g_\delta(m_1)}+O(\ve)$. This completes the proof.

\end{proof}

\section{Weak irreducibility}
\label{sec:bundle_tracking_main}
In this section we prove that the Markov process $z= (w, \pi)$ is weakly
irreducible, in the sense that it reaches any neighborhood of a selected target
space with some positive probability. Let us define
\begin{equation}\label{eq:shell_spaces}
  V_1 \eqdef \Ran P_1 = \mH \cap \operatorname{span}_\CC\{e_k : |k| = 1\},
  \quad
  V_2 \eqdef \mH \cap \operatorname{span}_\CC\{e_k : |k|^2 = 2\},
\end{equation}
where $P_1$ is the orthogonal projection of $\mH$ onto the first Fourier shell. The eigenvalues of $-\Delta$ on $V_1$ and on $V_2$ are respectively
\begin{equation}\label{eq:eigenvalues}
  \lambda_{V_1} = 1, \qquad \lambda_{V_2} = 2, \qquad\text{so}\qquad \lambda_{V_2} - \lambda_{V_1} = 1 .
\end{equation}
We write $\Sph(V_1) \eqdef \Sph \cap V_1$, and for $\pi \in \Sph$ with $P_1\pi \ne 0$ we set
\begin{equation}\label{eq:shell_direction}
  \pi^\le \eqdef P_1\pi, \qquad \pi^> \eqdef Q_1\pi,
\end{equation}
where $Q_1 \eqdef 1 - P_1$. For $\ve > 0$ the low-mode cone is
\begin{equation}\label{eq:cone_def}
  \mK_\ve \eqdef \big\{\phi \in \mH \setminus\{0\} : \|Q_1\phi\| \le \ve\|P_1\phi\|\big\} ,
\end{equation}
so that a unit vector $\pi \in \mK_\ve$ satisfies $\|\pi^>\| \le \ve$. Fix the target direction
\begin{equation}\label{eq:target_point}
  \pi_\star \eqdef \frac{1}{\sqrt2\pi}\cos(x_1) \in \Sph(V_1) .
\end{equation}
Then the main result of this section is the following.

\begin{proposition}
\label{prop:accessibility} There exists a constant $R_\star > 0$ such that the
following holds. For any $z_0 \in \mX$ and $\delta_\star > 0$, there exists a
deterministic time $T_0 = T_0(z_0, \delta_\star) > 0$ (which can be chosen to be
a measurable function of $z_0 \in \mX$) such that, for every $T
\ge T_0$ and every $R \ge R_\star$,
  \begin{equation}\label{eq:accessibility_conclusion}
     \PP\left( \|z_{T} - (0,\pi_\star)\| < \delta_\star \text{ and } z_{T} \in \mC_R \right) > 0 .
  \end{equation}
\end{proposition}
The proof follows from solving a control problem and can be found in Section~\ref{sec:proof_accessibility}. Therefore we start the section
by writing the typical Cameron--Martin control problem in our setting.

\subsection{A Stroock--Varadhan control problem}
\label{sec:tube_controls}

In this section we briefly address a classical control problem, that is used
throughout the work. In short, we prove that if we fix any Cameron--Martin
control, then with some probability our dynamics will stay close to that
controlled path. This is the classical Stroock--Varadhan support theorem
\cite{StroockVaradhan72} adapted to the infinite dimensional setting.

For any $g \in \CM$ write $w^g$ for the solution of
\eqref{eq:sns_vorticity_intro} with $\curl \xi$ replaced by $Q^{1/2}g$, and $\zeta^g$
for the scalar~\eqref{eq:linearized_abstract}  transported by it. An energy
estimate delivers immediately
  \begin{equation}\label{eq:skeleton_energy}
    \sup_{t \le T} \|w^g_t\|^2 + \int_0^T \|\nabla w^g_t\|^2 \ud t
    \ \le\ C \bigl( \|w_0\|^2 + \|g\|_{\CM}^2 \bigr) .
  \end{equation}
In this setting, we obtain the following result.
  \begin{lemma}\label{lem:tube}
Fix $T > 0$, data $w_0, \zeta_0 \in \mH, g \in \CM$, and $2 < s < \theta <
\alpha_\star - 1,\tau \in (0,T]$. Then write
    \begin{equation}\label{eq:tube_event}
      V \eqdef \sum_{k \in \Zstar} \sigma_k e_k W_k ,
      \qquad
      H \eqdef \int_0^\cdot Q^{1/2}g_r \ud r ,
      \qquad
      E \eqdef  C([0,T]; H^\theta)  .
    \end{equation}
For every $\rho > 0$ there is an $r > 0$ with $\PP(\|V - H\|_E < r) > 0$ and
    \begin{equation}\label{eq:tube}
      \bigl\{ \|V - H\|_E < r \bigr\} \ \subseteq\ \Bigl\{
        \sup_{t \le T} \bigl( \|w_t - w^g_t\|
              +  \|\zeta_t - \zeta^g_t\| \bigr)
        + \sup_{t \in [\tau,T]} \bigl( \|w_t - w^g_t\|_{H^s}+ \|\zeta_t - \zeta^g_t\|_{H^1} \bigr)
      < \rho \Bigr\} .
    \end{equation}
  \end{lemma}

  \begin{proof}
That $\PP\bigl( \|V - H\|_E < r \bigr) > 0$ for any $r>0$ follows
immediately from the fact that $g$ is a Cameron--Martin shift.

For the remainder of the proof, let us  assume $\theta \le s + 2$, since this
only weakens the norm of $E$ and so enlarges the event on the left
of~\eqref{eq:tube}. Let us write $d \eqdef w - w^g$ and split $d = z + v$,
where 
    \begin{equation}\label{eq:tube_z}
      z_t \eqdef (V_t - H_t) + \int_0^t \Delta e^{(t-r)\Delta}(V_r - H_r)\ud r .
    \end{equation}
In the range $0 < \theta - s \le 2$ we have
    \begin{equation}\label{eq:tube_semigroup}
      \bigl\|\Delta e^{\varrho\Delta}\bigr\|_{H^\theta \to H^{s}}
      \lesssim \varrho^{-1 + (\theta-s)/2} ,
      \qquad \varrho \in (0, T] ,
    \end{equation}
which is integrable on $[0, T]$ because $s < \theta$ and so
    \begin{equation}\label{eq:tube_zbound}
      \sup_{t \le T} \|z_t\|_{H^{s}} \lesssim \|V - H\|_E.
    \end{equation}
Now $v$ solves
    \begin{equation}\label{eq:tube_v}
      \partial_t v = \Delta v - B(d, w) - B(w^g, d).
    \end{equation}
Substituting $w = w^g + d$ and $d = v + z$ in~\eqref{eq:tube_v} and using
$\langle v, B(a, v)\rangle$ we  obtain an energy estimate that is independent of
$w$ and using~\eqref{eq:skeleton_energy} and that $z \in H^s \subseteq C^1$
(since $s >2$), and using Gr\"onwall, we obtain
    \begin{equation*}
      \sup_{t \le T} \|v_t\|^2 + \int_0^T \|\nabla v_t\|^2 \ud t
      \ \le\ C(w^g) \|V - H\|_E^2 .
    \end{equation*}
This proves the first estimate in~\eqref{eq:tube}. The other terms are handled
similarly, using that at a later time $\tau > 0$ the solutions have become smooth.
  \end{proof}

\subsection{The reduced first shell system}\label{sec:shell_reduced}
We start our work toward the main result of this section, by proving that the
dynamics reduced to the first Fourier shell is irreducible. To complete the full
proof of the result, we will then show that with some probability the full
dynamics will stay close to the one of the first shell, at least after some time.

More precisely, in this section we consider the evolution of the reduced processes
\begin{equation}\label{eq:shell-reduced}
  \partial_t \varpi = P_1 G_w \varpi - \langle P_1 G_w \varpi ,\varpi \rangle \varpi ,
  \qquad \varpi_0 \in \mathbb{S}(V_1) .
\end{equation}
An element $\zeta \in V_1$ is determined by the two complex Fourier coefficients $\hat\zeta(1,0)$ and $\hat\zeta(0,1)$. We therefore introduce the map $\mF_1$
  \begin{equation*}
    \mF_1 \colon V_1 \to \CC^2,
    \qquad \mF_1 \zeta \eqdef 2\sqrt2\pi\big(\hat\zeta(1,0), \hat\zeta(0,1)\big),
  \end{equation*}
which in particular maps $\mF_1 \pi_\star = (1,0) $. Rewrite \eqref{eq:shell-reduced} in terms of the Fourier coefficients of $w$. We define
\begin{equation}\label{eq:shell_controls}
  p \eqdef \hat{w} (1,1), \qquad q \eqdef \hat{w} (1,-1) .
\end{equation}

\begin{proposition}\label{cor:control}
  The following two hold with $T_\star = 8 \pi$:
  \begin{enumerate}[label=(\roman*)]
    \item Let $\varpi$ be as in \eqref{eq:shell-reduced}, and set $(a, b) = \mF_1
\varpi$. Then \eqref{eq:shell-reduced} for $\varpi$ reduces to
      \begin{equation}\label{eq:reduced_projective}
        \text{(PSA)}\quad
        \left\{
        \begin{aligned}
          \dot a &= - \frac12 q b + \frac12 p \bar b,\\
          \dot b &= \frac12 \bar q a - \frac12 p \bar a,
        \end{aligned}
        \right.
        \qquad
        \text{(LNS)}\quad
        \left\{
        \begin{aligned}
          \dot a &= \frac12 q b - \frac12 p \bar b,\\
          \dot b &= - \frac12 \bar q a + \frac12 p \bar a .
        \end{aligned}
        \right.
      \end{equation}
depending on whether the evolution is governed by \eqref{eq:two_operators_intro_PS} or \eqref{eq:two_operators_intro_NS}.
    \item With $N_{p,q}$ the generator of that system,
defined in~\eqref{eq:reduced_undamped} below, for any $\delta \in (0,1)$ and any $(a_0, b_0), (a_{\star}, b_{\star}) \in \{ \psi \in \CC^2  \colon  |\psi|=1\}$ there exists a control $q = q_\delta \in C^{\infty}_c ((0, T_\star); \CC)$ (with compact support, and $\|q\|_{L^\infty} \le 2$) such that the solution to
      \begin{equation*}
        \partial_t (a, b) = N_{0, q} \cdot (a, b),
      \end{equation*}
with initial condition $(a_0, b_0)$, satisfies $|(a_{T_\star}, b_{T_\star}) - (a_{\star}, b_{\star})| \leq \delta$.
  \end{enumerate}
\end{proposition}

\begin{proof}
\emph{Part~(i).} We prove the reduction by evaluating the first shell Fourier coefficients of the operators involved. The Laplacian does not contribute to the projective dynamics on a single shell, so we can rewrite the dynamics of $\varpi$ as:
\begin{equation*}
  \partial_t \varpi = C_w\varpi - \langle C_w\varpi, \varpi\rangle\varpi, \qquad C_w \eqdef -P_1 L^0_w P_1 .
\end{equation*}
Moreover, the quantity $\langle C_w\varpi, \varpi\rangle$ vanishes for $\varpi \in V_1$. For the passive scalar~\eqref{eq:two_operators_intro_PS} this is immediate since $K*w$ is divergence free. For linearized Navier--Stokes~\eqref{eq:two_operators_intro_NS} the second term of $L^0_w = B(w,\cdot) + B(\cdot,w)$ is not antisymmetric. In this case, the vanishing rests instead on the fact that on the first shell
\begin{equation}\label{eq:V1_swap}
  P_1 B(\varpi, w) = -2 P_1 B(w,\varpi) ,
  \qquad \varpi \in V_1 ,
\end{equation}
which we verify by the Fourier computation carried out below. Therefore, the dynamics further reduces to the linear equation
\begin{equation*}
  \partial_t \varpi = C_w\varpi .
\end{equation*}
It therefore suffices to compute $C_w$ on $V_1$.

From its definition, $B(e_j,e_k) = -\langle j^\perp,k\rangle|j|^{-2}e_{j+k}$. Moreover $C_w$ depends on $w$ solely through
\begin{equation*}
  p = \hat w(1,1), \qquad q = \hat w(1,-1), \qquad \hat w(-1,-1) = \bar p, \qquad \hat w(-1,1) = \bar q ,
\end{equation*}
as introduced in~\eqref{eq:shell_controls}. If we fix the output direction $n = j+k = (1,0)$ the two surviving pairs $(j,k)$ are $\big((1,1),(0,-1)\big)$ and $\big((1,-1),(0,1)\big)$, with $\langle j^\perp,k\rangle = 1$ and $-1$ and $|j|^2 = 2$, so
\begin{equation*}
  \widehat{B(w,\varpi)}(1,0)
    = -\frac12 p\hat{\varpi}(0,-1) + \frac12 q\hat{\varpi}(0,1)
    = \frac{1}{2\sqrt2\pi}\Big(\frac12 qb - \frac12 p\bar b\Big) .
\end{equation*}
The component $n = (0,1)$ is analogous, with the pairs $\big((-1,1),(1,0)\big)$ and $\big((1,1),(-1,0)\big)$. Multiplying by $2\sqrt2\pi$ we obtain, in the coordinates $(a,b)$,
\begin{equation}\label{eq:PB_components}
  \mF_1\big(P_1 B(w,\varpi)\big)
    = \Big(\frac12 qb - \frac12 p\bar b, \frac12 p\bar a - \frac12 \bar qa\Big) .
\end{equation}
Since for the passive scalar~\eqref{eq:two_operators_intro_PS} we have $L^0_w = B(w,\cdot)$ and hence $C_w = -P_1 B(w,\cdot)P_1$, we obtain from~\eqref{eq:PB_components}
\begin{equation*}
  \dot a = -\frac12 qb + \frac12 p\bar b,
  \qquad
  \dot b = \frac12 \bar qa - \frac12 p\bar a ,
\end{equation*}
as desired. For the linearized Navier--Stokes operator~\eqref{eq:two_operators_intro_NS} we have $L^0_w = B(w,\cdot) + B(\cdot,w)$. Repeating the computation with the two arguments exchanged, the same pairs contribute to $P_1 B(\varpi, w)$, now with the first shell mode in the denominator ($|j|^2 = 1$ in place of $2$) and with $\langle k^\perp, j\rangle = -\langle j^\perp, k\rangle$, which is precisely~\eqref{eq:V1_swap}. Therefore $P_1 L^0_w P_1 = P_1 B(w,\cdot)P_1 - 2P_1 B(w,\cdot)P_1 = -P_1 B(w,\cdot)P_1$, so that $C_w = P_1 B(w,\cdot)P_1$ is the negative of the passive scalar coupling, which proves part~(i).

We now define the generator $N_{p,q} \colon \CC^2 \to \CC^2$ by
\begin{equation}\label{eq:reduced_undamped}
  N_{p,q} (a,b) \eqdef
  \varsigma\frac12\begin{pmatrix} q b - p\bar b \\[2pt] -\bar q a + p \bar a\end{pmatrix},
  \qquad \varsigma = +1 \ \text{(LNS)}, \qquad \varsigma = -1 \ \text{(PSA)} .
\end{equation}
Since $\varsigma^2 = 1$ and $N_{p,q}$ is linear in $(p,q)$,
replacing the control $(p,q)$ by $(\varsigma p, \varsigma q)$ carries one model
into the other, so everything below is proved once for both. If we identify $\CC^2 \simeq \RR^4 $ with
a four dimensional real space with $a = a_1 + \iota a_2, b = b_1 + \iota b_2$,
then $N_{p,q}$ is a skew symmetric matrix. Indeed, if we consider the real inner
product \[ \langle (a,b),(a',b')\rangle = \Re(\bar a a' + \bar b b') = a_1 a_1'
+ a_2 a_2' + b_1 b_1' + b_2 b_2' , \] and write $N_{p,q}$ in its real
coordinates: \[ N_{p,q} = \varsigma\frac12\begin{pmatrix} 0 & 0 & q_1 - p_1 &
-(q_2 + p_2) \\
0 & 0 & q_2 - p_2 & q_1 + p_1 \\
p_1 - q_1 & p_2 - q_2 & 0 & 0 \\
q_2 + p_2 & -(q_1 + p_1) & 0 & 0
  \end{pmatrix} ,
\] then we find that $\langle (a,b), N_{p,q} (a, b) \rangle = 0$, or equivalently $N_{p,q}^\top = -N_{p,q}$.

{\emph{Part~(ii).}} We first steer the reduced system through a
piecewise constant control. Set $ v_1 = \varsigma (0, 1)$ and $v_2 = \varsigma
(0, \iota)$ and define
  \begin{equation}\label{eq:su2_generators}
    M_1 \eqdef N_{v_1} =  \frac12\begin{pmatrix} 0 & 1 \\ -1 & 0\end{pmatrix} ,
    \qquad
    M_2 \eqdef N_{v_2} =  \frac{\iota}{2}\begin{pmatrix} 0 & 1 \\ 1 & 0\end{pmatrix} .
  \end{equation}
Then for any $\psi_0, \psi_\star \in \{ \psi \in \CC^2  \colon  | \psi| =1 \}$ there exist $t_1,t_2,t_3 \ge 0$ with $t_1 + t_2 + t_3 \le T_\star$ such that
  \begin{equation}\label{eq:euler_decomposition}
    e^{t_3 M_1} e^{t_2 M_2} e^{t_1 M_1} \psi_0 = \psi_\star .
  \end{equation}

Indeed, $M_1, M_2$ are skew-Hermitian and traceless, and therefore lie in the real Lie algebra $\mf{su}(2)$ of the special unitary Lie group $SU(2) = \{ X  \colon X^*X = I, \ \mathrm{det}(X) = 1\}$ where $X$ runs over all $2 \times 2$ complex-valued matrices. In particular, $e^{t M_i} \in SU(2)$. The group $SU(2)$ is $3$-dimensional, indeed it can be represented as follows:
\begin{equation*}
  SU(2) = \left\{ \begin{pmatrix} \bar a_0 & \bar b_0 \\ -b_0 & a_0\end{pmatrix},  \ \ \  a_0, b_0 \in \CC, \ \ \ |a_0|^2+|b_0|^2=1 \right\}.
\end{equation*}
Moreover, $SU(2)$ acts transitively on $\Sph^3 = \{\psi \in \CC^2 \colon |\psi|=1\}$. For $\psi_0 = (a_0, b_0) \in \Sph^3$ the matrix
\begin{equation*}
  U_0 = \begin{pmatrix} \bar a_0 & \bar b_0 \\ -b_0 & a_0\end{pmatrix}\in SU(2)
\end{equation*}
satisfies $U_0\psi_0 = (1,0)$. Building $U_\star$ in the same way from a second element $\psi_\star$, the element $U \eqdef U_\star^{-1} U_0 \in SU(2)$ satisfies
\begin{equation*}
  U\psi_0 = U_\star^{-1}(1,0) = \psi_\star .
\end{equation*}
The question therefore becomes whether any element $U \in SU(2)$ admits a representation
\begin{equation}\label{eq:claim_su2}
  U= e^{t_3 M_1} e^{t_2 M_2} e^{t_1 M_1}, \qquad \text{ for some } \qquad t_1 + t_2 + t_3 \le T_\star.
\end{equation}
The proof of this claim reduces the problem to rotations of three
dimensional real space through the two-to-one homomorphism $\mathrm{Ad} \colon
SU(2) \to SO(3)$, $\mathrm{Ad}_g X = g X g^{-1}$, whose kernel is $\{\pm I\}$,
so that $SO(3) \cong SU(2)/\{\pm I\}$ (cf. \cite{Hall2015}*{Proposition~1.19}
and the surrounding discussion). It identifies $\mf{su}(2)$ with $(\RR^3,
\times)$ through $\Phi(\vec n) \eqdef -\frac{\iota}{2} \vec n \cdot \vec\sigma$,
with $\vec\sigma$ the Pauli matrices, and carries $\mathrm{Ad}_{e^{tM}}$ to the
rotation about the axis $\Phi^{-1}(M)$ of angle $t$. Our two generators are $M_i
= \Phi(\vec n_i)$ with $\vec n_1 = (0,-1,0)$ and $\vec n_2 = (-1,0,0)$. Now we use the Euler angle representation of
rotations of $\RR^3$ (in which any rotation can be represented as the
composition of three rotations around two orthogonal axes). We find that for any
$U \in SO(3)$ there exist $t_1, t_2, t_3 \in [0, 2\pi]$ such that
$$\mathrm{Ad}_{e^{t_3M_1}}\mathrm{Ad}_{e^{t_2M_2}}\mathrm{Ad}_{e^{t_1M_1}} =
\mathrm{Ad}_U .$$ Since $SO(3) \cong SU(2)/\{\pm I\}$, we find that necessarily 
\begin{equation*}
   e^{t_3 M_1} e^{t_2 M_2} e^{t_1 M_1} \in \{\pm U\}.
\end{equation*}
Since $e^{2 \pi M_1} = -I$, the claim follows up to allowing $t_3 \in [0, 4\pi]$.
Finally we pass to smooth compactly supported controls by mollifying.
\end{proof}

Now we use the controllability on the shell to obtain the weak irreducibility of the full projective process. In this subsection we assume the initial condition lives almost entirely in the first shell, meaning $\pi_0 \in \mK_{\ve} $ for some $\ve \in (0,1)$. Then we combine the spectral gap of the Laplacian between the first and second shells with the control problem solved above. Indeed, we take the control $q$ from Corollary~\ref{cor:control} and define the vorticity control $w_{\mathrm{c}}$ through its Fourier transform on $[0, T_\star]$:
\begin{equation*}
  \mF w_{\mathrm{c}} (t, k) \eqdef \begin{cases}
    0 & \text{ if } k \not\in \{ (1, -1), (-1, 1) \} ,\\
    q_t & \text{ if } k = (1, -1), \\
    \overline{q}_t & \text{ if } k = (-1, 1) .
  \end{cases}
\end{equation*}
This control is supported on the conjugate pair $\pm(1,-1)$, and for any such
$w$ we have
\begin{equation}\label{eq:shell_operator_bound}
  \|B(w, \cdot) P_1\|_{\mH \to \mH} + \|P_1 L^0_w\|_{\mH \to \mH}
  + \|B(\cdot, w)\|_{\mH \to \mH} \ \lesssim\ \|w\|_{C^1} .
\end{equation} Then define
the slow-down
\begin{equation}\label{eq:wespc}
  w^\ve_{\mathrm{c}}(t, x) = \ve w_{\mathrm{c}}(\ve t, x) .
\end{equation}
With this definition, by construction the solution to
\begin{equation}\label{eq:pi-eps-proj}
  \partial_t \pi = P_1 G_{w^\ve_\mathrm{c}} \pi -  \langle P_1 G_{w^\ve_\mathrm{c}} \pi, \pi  \rangle \pi, \qquad \pi_0 \in V_1, \qquad \mF_1 \pi_0 = (a_0, b_0),
\end{equation}
satisfies
\begin{equation*}
  \| \pi_{\ve^{-1} T_\star} - \mF_1^{-1} (a_{\star}, b_{\star}) \| \leq  \delta.
\end{equation*}
This is because the evolution above is independent of the Laplacian and simply a time rescaling of the evolution in Corollary~\ref{cor:control}.

It remains to show that the evolution of the true solution
\begin{equation} \label{eq:pi-eps}
  \partial_t \pi =  G_{w^\ve_\mathrm{c}} \pi -  \langle  G_{w^\ve_\mathrm{c}} \pi, \pi  \rangle \pi,
\end{equation}
stays close to \eqref{eq:pi-eps-proj}, if the initial condition is close to $\mF_1^{-1}(a_0, b_0)$
\begin{equation}\label{eq:ic-assu}
  \pi_0 \in \mK_\ve , \qquad \| P_1 \pi_0 -
  \mF_1^{-1} (a_0, b_0) \| \leq \ve.
\end{equation}

\begin{lemma}\label{lem:shell_approximation}
Consider $\ve \in (0, 1)$ and $\pi_0 \in \Sph, (a_0, b_0),
(a_\star, b_\star)\in\{ \psi \in \CC^2  \colon  |\psi|=1\}$ satisfying
\eqref{eq:ic-assu}. Then consider the control $w^\ve_\mathrm{c}$ from
\eqref{eq:wespc}, and the solution $(a_{t}, b_{t})$ from
Corollary~\ref{cor:control} and $\pi$ to \eqref{eq:pi-eps}. Then it holds that
uniformly over $\ve \in (0, 1)$:
  \begin{equation*}
    \| \pi_{\ve^{-1} T_\star} - \mF_1^{-1} (a_{T_\star}, b_{T_\star}) \| \lesssim \ve  . 
  \end{equation*}
\end{lemma}
\begin{proof}
Let us write $w=w_{\mathrm{c}}$ for simplicity and define $\varpi(t, x) = \pi(\ve^{-1} t, x)$. Then
  \begin{equation*}
    \partial_t \varpi = \ve^{-1} \Delta \varpi - L^0_{w} \varpi - \ve^{-1 } \langle \varpi, \Delta \varpi \rangle \varpi + \langle \varpi, L^0_{w} \varpi \rangle \varpi , 
  \end{equation*}
where $L^0_w$ is defined in \eqref{eq:two_operators_intro_NS} and \eqref{eq:two_operators_intro_PS} depending on whether we work with PSA or LNS. We start with $\varpi_0 \in \mK_{\ve}$, and our first objective is to prove that $\varpi_t$ remains close to the cone $\mK_{\ve}$ for all $t \in [0, T_\star]$. Define $\pi^{>} \eqdef Q_1 \varpi$. Then we find the following energy estimate for some $C>0$:
  \begin{equation}\label{eq:tail_ineq}
    \frac12\frac{\ud}{\ud t}\|\pi^>\|^2
    \le - \ve^{-1} (\lambda_{V_2} - \lambda_{V_1})\big(1 - \|\pi^>\|^2\big)\|\pi^>\|^2
    + C \|w\|_{C^1}\|\pi^>\| .
  \end{equation}
To see this,
  \begin{equation}\label{eq:tail_energy}
    \frac12\frac{\ud}{\ud t}\|\pi^>\|^2
      = \ve^{-1} \left( \langle \pi^>, \Delta\pi^>\rangle -  \langle \varpi, \Delta \varpi \rangle \| \pi^> \|^2 \right)
      - \langle \pi^>, L^0_w\varpi\rangle
      + \langle \varpi, L^0_{w} \varpi \rangle \|\pi^>\|^2 .
  \end{equation}
Now for the first term we use
  \begin{equation*}
    -  \langle \varpi, \Delta \varpi \rangle \leq \lambda_{V_1} \| P_1 \varpi \|^2 +  \| \nabla \pi^>\|^2 \leq \lambda_{V_1} (1 - \| \pi^> \|^2)+  \| \nabla \pi^>\|^2 ,
  \end{equation*} 
which bounds the bracket in~\eqref{eq:tail_energy} by $-\|\nabla\pi^>\|^2 (1 - \|\pi^>\|^2) + \lambda_{V_1}(1 - \|\pi^>\|^2)\|\pi^>\|^2$ and so explains the first term of~\eqref{eq:tail_ineq}. For the last one, split $\varpi = \pi^\le + \pi^>$ with $\pi^\le = P_1\varpi$. The transport pairings $\langle \pi^>, B(w,\pi^>)\rangle$ and $\langle \varpi, B(w,\varpi)\rangle$ vanish because $K*w$ is divergence free, and what is left is bounded through~\eqref{eq:shell_operator_bound},
\begin{equation*}
  \big|\langle \pi^>, L^0_w\varpi\rangle\big|
  + \big|\langle \varpi, L^0_w\varpi\rangle\big| \|\pi^>\|^2
  \ \lesssim\ \|w\|_{C^1}\|\pi^>\| ,
\end{equation*}
using $\|\varpi\| = 1$ and $\|\pi^>\| \le 1$. In the (PSA) case $L^0_w = B(w,\cdot)$ and the bound on $B(w,\cdot)P_1$ alone suffices.

We now use~\eqref{eq:tail_ineq}. The control obeys $\|w_{\mathrm
c}\|_{L^\infty([0,T_\star]; C^1)} \lesssim \|q\|_{L^\infty} \lesssim 1$. As long as $\|\pi^>\|
\le \frac12$ we therefore have $1 - \|\pi^>\|^2 \ge \frac34$, so that~\eqref{eq:tail_ineq}
reads
  \begin{equation*}
    \frac{\ud}{\ud t}\|\pi^>\| \le - \frac{3}{4\ve}\|\pi^>\| + C
  \end{equation*}
 It follows that for small enough $\ve$ we have:
  \begin{equation}\label{eq:tail_sup}
    \sup_{0 \le t \le T_\star} \|\pi^>_t\| \ \lesssim \ve .
  \end{equation}

We now show that the dynamics of $\varpi$ stay close to those of the reduced system. We write the evolution of $\pi^{\le}$ as
\begin{align}
  \partial_t\pi^{\le} &= \ve^{-1}\big(-\lambda_{V_1} - \alpha\big)\pi^{\le}
    - P_1 L^0_w\varpi + \beta\pi^{\le} , \label{eq:pileq_coupled}
\end{align}
with
\begin{equation*}
  \alpha \eqdef \langle\varpi, \Delta\varpi\rangle, \qquad
  \beta \eqdef \langle\varpi, L^0_w\varpi\rangle.
\end{equation*}
Now define $\rho \eqdef \|\pi^{\le}\| = \sqrt{1 - \|\pi^>\|^2}$, so that
\begin{equation}\label{eq:rho_close}
  \rho \ge \frac12 \qquad\text{and}\qquad 0 \le 1 - \rho \le \|\pi^>\|^2 \lesssim \ve^2 .
\end{equation}
Then we introduce the shell direction and the tangential projection
\begin{equation*}
  \psi \eqdef \frac{\pi^{\le}}{\rho} \in \Sph(V_1) , \qquad
  \Pi^\perp \eqdef \mathrm{Id}_{V_1} - \psi\otimes\psi, 
\end{equation*}
so that differentiating $\psi = \pi^{\le}/\rho$ and using $\dot\rho = \langle\psi, \partial_t\pi^{\le}\rangle$, the radial part cancels and
\begin{equation*}
  \partial_t\psi = \frac1\rho\Pi^\perp\partial_t\pi^{\le} .
\end{equation*}
In~\eqref{eq:pileq_coupled} the two terms proportional to $\pi^{\le}$, that is $\ve^{-1}(-\lambda_{V_1}-\alpha)\pi^{\le}$ and $\beta\pi^{\le}$, are parallel to $\psi$ and therefore they are killed by the projection $\Pi^\perp$.

The dynamics therefore reduces to
\begin{equation}\label{eq:psihat_eq}
  \partial_t\psi = -\Pi^\perp P_1 L^0_w\psi + r ,
  \qquad r \eqdef -\frac1\rho\Pi^\perp P_1 L^0_w\pi^> .
\end{equation}
By~\eqref{eq:reduced_undamped} and the discussion after it, the operator $P_1 L^0_w$ is skew-symmetric on $V_1$ (no matter whether we consider \eqref{eq:two_operators_intro_NS} or \eqref{eq:two_operators_intro_PS}), so $\langle\psi, P_1 L^0_w\psi\rangle = 0$. In particular, we have reduced the dynamics to
\begin{equation*}
    \partial_t\psi = - P_1 L^0_w\psi +  r ,
\end{equation*}
which is a perturbation of~\eqref{eq:shell-reduced}. By~\eqref{eq:shell_operator_bound}, $\rho \ge \frac12$ and~\eqref{eq:tail_sup}, the remainder obeys $\sup_{0\le t\le T_\star}\|r_t\| \lesssim \|w\|_{C^1}\|\pi^>\| \lesssim \ve$. The result now follows from Corollary~\ref{cor:control} and an application of Gr\"onwall.
\end{proof}
While the previous result guarantees that at the final time $\pi^>$ is small in
$L^2$, we need a slightly stronger version, which guarantees also smallness in
$H^1$. This is the content of the following lemma.

\begin{lemma}\label{lem:h1_tail}
In the setting of Lemma~\ref{lem:shell_approximation}
  \begin{equation*}
    \int_0^{T_\star} \|\nabla \pi^>_t\|^2 \ud t \lesssim \ve^2
    \qquad\text{and}\qquad
    \|\pi^>_{T_\star}\|_{H^1} \lesssim \ve .
  \end{equation*}
\end{lemma}

\begin{proof}
We integrate the tail energy identity~\eqref{eq:tail_energy},
this time keeping the gradient term that the proof
discards. By~\eqref{eq:tail_sup} we have
$\sup_{t \le T_\star} \|\pi^>_t\| \le C\ve
\le \tfrac12$, so that
$1 - \|\pi^>\|^2 \ge \tfrac12$. Hence the first term of~\eqref{eq:tail_energy}
is bounded by
$-\frac12 \ve^{-1} \|\nabla\pi^>\|^2 +
\ve^{-1}\lambda_{V_1} \|\pi^>\|^2$,
and we obtain
  \begin{equation*}
    \frac12 \|\pi^>_{T_\star}\|^2
    + \frac{1}{2\ve} \int_0^{T_\star} \|\nabla\pi^>\|^2 \ud t
    \le \frac12 \|\pi^>_0\|^2
    + \frac{\lambda_{V_1}}{\ve} \int_0^{T_\star} \|\pi^>\|^2 \ud t
    + C \int_0^{T_\star} \|w\|_{C^1} \|\pi^>\| \ud t
    \lesssim \ve ,
  \end{equation*}
where we used $\|\pi^>_0\| \le \ve$, $\|\pi^>_t\| \le C\ve$ and
$\|w\|_{C^1} \le C_0$. Multiplying by $2\ve$ we obtain
$\int_0^{T_\star}
\|\nabla\pi^>\|^2 \ud t \lesssim \ve^2$. Moreover, by the
mean value theorem for integrals there is a $t_0 \in [T_\star - 1, T_\star]$
with $\|\nabla\pi^>_{t_0}\|^2 \lesssim \ve^2$.
This proves the first estimate. The second one follows from Schauder theory
applied to~\eqref{eq:proj-spde}.
\end{proof}

We conclude this subsection with a result guaranteeing that the dynamics enters $\mK_\ve$ with positive probability from every initial condition.

\begin{lemma}\label{lem:coupling_nondeg}
Fix any $ \zeta \in \mH \setminus \{0 \}$. Then there exists a $j \in \Zstar$ such that
  \begin{equation}\label{eq:coupling_nonzero}
    P_1 L^0_{\bar w}\zeta \ne 0 ,
  \end{equation}
  where $\bar w = c e_j + \bar c e_{-j}$ for some $c \in \{1, \iota\}$.
\end{lemma}

\begin{proof}
Write $\zeta = \sum_{k \in \Zstar} \zeta_k e_k $. Since $\zeta \ne 0$ there
exists a $k$ with $\zeta_k \neq 0$. As $k \ne 0$, it must be that either $k_1
\ne 0$ or $k_2 \ne 0$ and without loss of generality suppose that $k_2 \ne 0$.
Then set $j = (1,0)-k$, which lies in
$\Zstar$ because $k_2 \ne 0$. For the fixed mode $j$, the only Fourier mode of
$\zeta$ that $L^0_{e_j}$ carries into $e_{(1,0)}$ is $k = (1,0)-j$, so $\langle
e_{(1,0)}, L^0_{e_j}\zeta\rangle = \zeta_k \gamma$, where, using $\langle
j^\perp, k\rangle = j_2 k_1 - j_1 k_2 = -k_2$ for this $j$, we have
  \begin{equation}\label{eq:gamma_value}
    \gamma =
    \begin{cases}
      \dfrac{k_2}{|j|^2}, & \quad \text{(PSA)},\\[2ex]
      k_2\Big(\dfrac{1}{|j|^2} - \dfrac{1}{|k|^2}\Big), & \quad\text{(LNS)},
    \end{cases}
  \end{equation}
the (LNS) value adding the contribution of $B(\zeta,e_j)$, whose coefficient is $-\langle k^\perp, j\rangle|k|^{-2} = -k_2|k|^{-2}$. In the (PSA) case $\gamma \ne 0$ since $k_2 \ne 0$. In the (LNS) case $\gamma = 0$ would force $|(1,0)-k|^2 = |k|^2$, hence $k_1 = \frac12$, which is impossible for $k \in \Zstar$.

Since $(1,0)$ lies in the first shell, $P_1 L^0_{e_j}\zeta \ne
0$. It remains to pass from $e_j$ to a real field $\bar w$, where we have $c$
free to choose in $\{1, \iota\}$. Since $L^0$ is linear in the field, the two
choices produce $P_1 L^0_{e_j}\zeta + P_1 L^0_{e_{-j}}\zeta$ and $\iota\bigl(P_1
L^0_{e_j}\zeta - P_1 L^0_{e_{-j}}\zeta\bigr)$. If both of these were to vanish,
then so would $P_1 L^0_{e_j}\zeta$. Hence one of the two choices
satisfies~\eqref{eq:coupling_nonzero}.
\end{proof}

\subsection{Proof of the accessibility proposition}\label{sec:proof_accessibility}

\begin{proof}[Proof of Proposition~\ref{prop:accessibility}]

  The proof follows from solving a control problem. We will fix the control
  through a series of time intervals. Let us set
  \begin{equation*}
    0 < T_c < T_1 < T_2 < T_3 < \infty .
  \end{equation*}
Then we are given an initial condition $(w_0, \pi_0) \in \mX$ and a parameter $\delta_\star \in (0,1)$. We design a vorticity path $w^{\mathrm{fin}} \in C([0, T_3]; \mH)$ with $w^{\mathrm{fin}}(0, \cdot) = w_0(\cdot)$, realizable through a control $g \in \CM$ in the sense of Section~\ref{sec:tube_controls}, such that the solution to
  \begin{equation*}
    \partial_t \pi = G_{w^{\mathrm{fin}}} \pi - \langle G_{w^\mathrm{fin}} \pi, \pi \rangle \pi, \qquad \pi(0, \cdot) = \pi_0(\cdot)
  \end{equation*}
satisfies
  \begin{equation}\label{eq:skeleton_targets}
    \| \pi_{T_3} - \pi_\star \| \leq \frac{\delta_\star}{2} , \qquad
    \|w^{\mathrm{fin}}_{T_3}\|_{H^{s_\star}} \le \frac{\delta_\star}{2} , \qquad
    \|\pi_{T_3}\|_{H^1} \le 2 .
  \end{equation}
The tube estimate of Lemma~\ref{lem:tube} (which applies also to $\zeta$
replaced by $\pi$) then gives the conclusion.

We first construct a $w^{\mathrm{fin}}$ that is not smooth, then we will mollify
it. Fix an $\ve \in (0, \delta_\star)$, to be chosen below, write
$w^{\mathrm f}$ for the unforced solution started from $w_0$, and set
\begin{equation*}
  w^{\mathrm{fin}}(t) \eqdef
  \begin{cases}
    w^{\mathrm f}_t , \quad& t \in [0, T_{\mathrm c}] ,\\[2pt]
    \bar w_{\mathrm{inj}} , & t \in (T_{\mathrm c}, T_1] ,\\[2pt]
    0 , & t \in (T_1, T_2] ,\\[2pt]
    w^\ve_{\mathrm c}(t - T_2) , & t \in (T_2, T_3] ,
  \end{cases}
\end{equation*}
with the latter three pieces defined as follows.

On $[T_{\mathrm c}, T_1]$ we apply Lemma~\ref{lem:coupling_nondeg} to $\zeta = \pi_{T_{\mathrm c}}$, the deterministic state at the end of the coasting phase, which is non-zero and smooth. The lemma produces a mode $j \in \Zstar$ and a real vorticity field $\bar w_{\mathrm{inj}}$ supported on $\{\pm j\}$ with $P_1 L^0_{\bar w_{\mathrm{inj}}}\pi_{T_{\mathrm c}} \ne 0$. Then for $T_1 - T_{\mathrm c}$ small enough the first shell carries nonzero mass,
\begin{equation*}
  \|P_1\pi_{T_1}\| \ge c_1 (z_0) > 0 .
\end{equation*}
Indeed, if $P_1\pi_{T_{\mathrm c}} \ne 0$, then this holds by continuity. If instead $P_1\pi_{T_{\mathrm c}} = 0$, then $\frac{\ud}{\ud t}\big|_{t=T_{\mathrm c}} P_1\pi_t = -P_1 L^0_{\bar w_{\mathrm{inj}}}\pi_{T_{\mathrm c}} \ne 0$, because the terms $P_1\Delta\pi_{T_{\mathrm c}}$ and $\langle \pi_{T_{\mathrm c}}, G_{\bar w_{\mathrm{inj}}}\pi_{T_{\mathrm c}}\rangle P_1\pi_{T_{\mathrm c}}$ both vanish. The constant $c_1$ is deterministic because the coasting phase is the unforced flow from the given initial data.

On $[T_1, T_2]$ we switch the control off, and $\pi$ solves the projective heat flow $\partial_t\pi = \Delta\pi - \langle\pi, \Delta\pi\rangle\pi$. The two shells are eigenspaces of $\Delta$, so the ratio of tail to shell mass contracts at the spectral-gap rate,
\begin{equation*}
  \frac{\|Q_1\pi_t\|}{\|P_1\pi_t\|}
    \le e^{-(\lambda_{V_2} - \lambda_{V_1})(t - T_1)}\frac{\|Q_1\pi_{T_1}\|}{\|P_1\pi_{T_1}\|}
    \le \frac{1}{c_1}e^{-(\lambda_{V_2} - \lambda_{V_1})(t - T_1)} .
\end{equation*}
Taking $T_2 - T_1$ of order $\log(1/\ve)$ brings this ratio below $\ve$, so that $\pi_{T_2} \in \mK_\ve$.

Finally, on $[T_2, T_3]$, $\pi_{T_2} \in \mK_\ve$ has shell direction $\hat\psi = P_1\pi_{T_2}/\|P_1\pi_{T_2}\|$ with coordinates $(a_0, b_0) = \mF_1\hat\psi$, and the pair $(\pi_{T_2}, (a_0, b_0))$ satisfies the hypothesis~\eqref{eq:ic-assu}. Let $q$ be the control of Corollary~\ref{cor:control} steering $(a_0, b_0)$ to $\mF_1\pi_\star = (1,0)$ in time $T_\star$, and let $w^\ve_{\mathrm c}$ be the associated slowed-down control~\eqref{eq:wespc}, so that $T_3 - T_2 = \ve^{-1}T_\star$. The approximation estimate (Lemma~\ref{lem:shell_approximation}, with the accuracy of Corollary~\ref{cor:control} set to $\delta = \ve$) then gives
\begin{equation}\label{e:estimate-controlled}
  \|\pi_{T_3} - \pi_\star\| = \big\|\pi_{T_3} - \mF_1^{-1}(1,0)\big\|
    \lesssim \ve + \delta \lesssim \ve .
\end{equation}
The control has size $\|w^\ve_{\mathrm c}\| \lesssim \ve$ throughout this
phase. It is supported on finitely many Fourier modes, so in
particular $\|w^{\mathrm{fin}}_{T_3}\|_{H^{s_\star}} \lesssim \ve$, and
$\|\pi_{T_3}\|_{H^1} \le 2$ by Lemma~\ref{lem:h1_tail}. If we choose $\ve$ small
enough in terms of $\delta_\star$ and $c_1$, then the phases together
give~\eqref{eq:skeleton_targets} with $\delta_\star/4$ in place of
$\delta_\star/2$ in the first two bounds. 

The path just built is discontinuous across the times $T_{\mathrm c}, T_1, T_2$.
However we can make it smooth in time, by mollifying at a scale $\kappa <
T_{\mathrm c}/4$. If $\kappa$ is sufficiently small, then we still have~\eqref{e:estimate-controlled}.
Denoting by $w^{\mathrm{fin}}$ also this smoothened path, we find that the
correct control $g$ for the equation is given by
\begin{equation}\label{eq:accessibility_control}
  Q^{1/2} g_t \eqdef \partial_t w^{\mathrm{fin}}_t - \Delta w^{\mathrm{fin}}_t
  + B(w^{\mathrm{fin}}_t, w^{\mathrm{fin}}_t) .
\end{equation}
Since the right side is smooth in space we have $g \in \CM$. To extend the conclusion from $T_3$ to every horizon $T \ge T_3$ we extend the control by zero on $[T_3, T]$.

\end{proof}

\section{Proof of the main theorem via asymptotic coupling}
\label{sec:proof_of_main_theorem}

The aim of this section is to complete the proof of
Theorem~\ref{thm:main_result_intro}. We do so by constructing an asymptotic coupling for
the joint process $z_t = (w_t, \pi_t)$. The existence of such a coupling,
together with the absolute continuity of the law of the coupled process to the
original one, is a
sufficient condition for uniqueness of the stationary measure, by a celebrated criterion
due to Hairer, Mattingly, and Scheutzow~\cite{HMS11}.

The criterion is stated for a Markov operator $P$ on a Polish space
$\mathcal{X}$ (which will correspond to our space $\mX$) with metric $d$, with $P^{[\infty]}\mu$ the law on the path space
$\mathcal{X}^\infty = \mathcal{X}^{\NN}$ of the chain started at $z_0 \sim \mu$
and run with $z_{n+1} \sim P(z_n, \cdot)$. We apply this to our continuous-time
process by taking $P = P^2$, the time-two step of the Markov semigroup
\[
  P^t \varphi(z) \eqdef \EE_z[\varphi(z_t)]
\]
of the joint process. Since every stationary measure of $z_t$ is invariant for $P^2$, uniqueness of $P^2$-invariant measures yields uniqueness for the continuous time process as well.

\begin{theorem}[\cite{HMS11}*{Theorem~1.1}]
  \label{thm:hms_coupling_criterion} Let $P$ be a Markov operator on a Polish space $\mathcal{X}$ admitting two ergodic invariant measures $\mu_1$ and $\mu_2$, and let $m_i = P^{[\infty]}\mu_i$ for $i = 1,2$ be the associated path-space measures on $\mathcal{X}^\infty$. Then $\mu_1 = \mu_2$ whenever there exists a measure $\Gamma$ on the product path space $\mathcal{X}^\infty \times \mathcal{X}^\infty$ such that:
  \begin{enumerate}
    \item The two marginals of $\Gamma$ are absolutely continuous with respect
      to $m_1$ and $m_2$, respectively.
    \item $\Gamma$ assigns positive measure to the diagonal
      \begin{equation}\label{eq:diagonal_at_infinity}
        D = \Big\{ (z, z') \in \mathcal{X}^\infty \times
        \mathcal{X}^\infty : \lim_{n\to\infty} d\big(z_n, z_n' \big) = 0 \Big\} ,
        \qquad \Gamma(D) > 0 .
      \end{equation}
  \end{enumerate}
\end{theorem}
Throughout this section, we view the space $\mX = \mH \times \Sph$ as a closed
subset of the separable Hilbert space with the induced metric.
Two distinct stationary measures of the flow are two distinct $P^2$-invariant
measures, whose ergodic decompositions cannot coincide, so they produce two
distinct $P^2$-ergodic measures. Hence we may assume without loss of generality
that $\mu_1$ and $\mu_2$ are $P^2$-ergodic.

The construction of $\Gamma$ in this section rests on two ingredients: the
construction of the coupling that leads to synchronization
(Proposition~\ref{prop:asymptotic_contraction}) and a change of measure that
guarantees the coupled process stays absolutely continuous to the original one
(Proposition~\ref{prop:absolute_continuity}). Subsection~\ref{sec:main_proof}
combines these to obtain the proof of the main result, and Subsection~\ref{sec:fk_consequences}
proves Corollaries~\ref{cor:FK_formula} and~\ref{cor:deterministic_datum}.

\subsection{Construction of the coupling and asymptotic synchronization}\label{sec:coupling_meeting}

Recall that we  denote by $\Phi_{s,t}(\cdot, W) \colon
\mathcal{X} \to \mathcal{X}$, for $0 \le s \le t$, the flow which
maps $z$ to the solution $z_t = (w_t, \pi_t)$
to~\eqref{eq:sns_vorticity_intro} and~\eqref{eq:proj-spde} started at time $s$
in $z$, and with $w$ driven by $W$. We abbreviate
$\Phi_t = \Phi_{0,t}$.

We will construct a coupled system $(z_t, z_t')$.
The first copy $z_t = \Phi_t(z_0, W)$ is driven by a cylindrical Wiener process
$W$, while the second, $z'_t = \Phi_t(z'_0, W')$, by a process $W'$ whose law
will be absolutely continuous with respect to that of the cylindrical Wiener
process. The driver $W'$ is constructed as follows:
\begin{itemize}
  \item It is chosen independently of $W$ up to a ``meeting'' time. Because we
  can force the two solutions to stay close, the time at which we check whether
  the processes have met will be chosen deterministic.
  \item After this time, $W'$ is $W$ shifted by a control
    \begin{equation*}
      W'_{t} - W'_{\sigma_\star} = W_t - W_{\sigma_\star}
      + \int_{\sigma_\star}^t h_s \ud s , \qquad \forall t \geq \sigma_\star,
    \end{equation*}
    where $h$ is constructed by inverting the Malliavin matrix, following the
    standard approach of asymptotic strong Feller proofs, only in our localized setting.
\end{itemize}
We start by proving that two processes started with independent noise have
positive probability of meeting. 
\begin{lemma}\label{lem:meeting_time_new}
  Let $R_\star$ be the constant of Proposition~\ref{prop:accessibility} and let
  $\mu_1, \mu_2$ be $P^2$-invariant probability measures on $\mX$. Let $z_t =
  \Phi_t(z_0, W)$ and $z'_t = \Phi_t(z'_0, W'')$ be two solutions starting from
  $z_0 \sim \mu_1$ and $z'_0 \sim \mu_2$, with $z_0, z'_0, W, W''$ all mutually
  independent. Then for every $\delta_\star \in (0,1)$ and every $R \ge R_\star$
  there exists a $T_\star(\delta_\star, R) >0$ deterministic such that
  \begin{equation*}
    \PP(  \| z_{T_\star} - z_{T_\star}' \| \le \delta_\star ,  \ \text{ and } \ z_{T_\star},z_{T_\star}' \in \mC_R  ) > 0 .
  \end{equation*}
\end{lemma}
\begin{proof}
  Consider the time $T(z_0, \delta_\star /2)$ of
  Proposition~\ref{prop:accessibility}. Then choose a deterministic time $ T_\star$
  such that $\PP_{\mu_i}(T(z, \delta_\star /2) <  T_\star )> 0$ for both $i \in
  \{1,2\}$ (here $z \sim \mu_i$). Then
  \begin{equation*}
    \PP(  \| z_{T_\star} - z_{T_\star}' \| \le \delta_\star ,  \ \text{ and } \ z_{T_\star},z_{T_\star}' \in \mC_R  ) \geq \prod_{i =1}^2 \PP_{\mu_i}( \|z_{ T_\star} - (0, \pi_\star) \| < \delta_\star/2 \text{ and } z_{T_\star} \in \mC_R ) >0
  \end{equation*}
  by that same proposition.
\end{proof}

We now build the control. Here we start after the meeting time $\sigma_\star$
but to lighten the notation we assume that we start at time zero
from a pair of data $z_0, z'_0$ such that
\begin{eqnarray*}
  \|z_0 - z'_0\| \leq \delta_\star, \qquad z_0, z'_0 \in \mC_R .
\end{eqnarray*}
For a given control $(h_t)_{t \geq 0}$, let $\vartheta_t \eqdef z_t - z'_t$
denote the difference of the two coupled copies, and write $\rho_t =
\|\vartheta_t\|$ for its norm. Define the stopping time
\begin{equation} \label{eq:stopping_time_tau}
  \tau = \inf\{ 2n \in 2\NN  : \rho_{2 n} >  \delta e^{- 2 n} \}, \qquad \delta \in (0, 1) .
\end{equation}
The contraction of the coupling that follows will be local, in
the sense that we stop as soon as the stopping time $\tau$ kicks in (which we
expect to happen with positive probability). The geometric threshold of the
stopping time is used both for synchronization and to guarantee the absolute
continuity of the final measures.
We also define~$A_{s,t}$, the
two-parameter version of the Malliavin derivative operator $A_t = A_{0,t}$ of
Section~\ref{sec:quantitative_nondegeneracy}. For $0 \le s \le t$,
\begin{equation*}
  A_{s,t}\colon \CM \to T_{z_t}\mX, \qquad A_{s,t} h = \int_s^t J_{r,t}(Q^{1/2} h(r), 0)\ud r,
\end{equation*}
its adjoint is $(A_{s,t}^* \eta)(r) = Q^{1/2}\Pi_w(J_{r,t})^* \eta$,
$r \in [s,t]$, where $\Pi_w$ is the projection onto the base component  (as in Section~\ref{sec:quantitative_nondegeneracy}), and
$M_{s,t} = A_{s,t} A_{s,t}^*$ is the Malliavin matrix over $[s,t]$, acting on
the state space.

Each cycle $[2n, 2n+2)$ splits into two halves. On the \emph{active phase} $[2n,
2n+1)$ the control acts, and on the \emph{free phase} $[2n+1, 2n+2)$ no control
is applied and small scales are allowed time to contract through dissipation. 

Throughout this section a single integer subscript denotes the cycle, so that
$J_n \eqdef J_{2n, 2n+1}$, $M_n \eqdef M_{2n, 2n+1}$ and $A_n \eqdef A_{2n,
2n+1}$ are the Jacobian, the Malliavin matrix and the Malliavin derivative over
$[2n, 2n+1]$. Likewise for the free decay phase, we denote $J^\flat_n \eqdef
J_{2n+1, 2n+2}$. 
The energy of a cycle is
\begin{equation}\label{eq:def-8En}
  \mathcal{E}_n \eqdef \int_{2n}^{2n+2}\| w_s \|_{H^1}^2 +  \| \pi_s \|_{H^1}^2  \ud s,
\end{equation}
where the latter are the components of $z_s = (w_s, \pi_s) = \Phi_{0, s}(z_0,
W)$.
Then for arbitrary $K>0$, we introduce a smooth, non-increasing cutoff
function $\psi_K: \R \to [0,1]$ such that $\psi_K(x) = 1$ for $x \le
2K^2$, $\psi_K(x) = 0$ for $x \ge 2(K+1)^2$, and
$|\psi'_K(x)| \le 2$.
On discrete intervals $[2n, 2n+2)$, the control $h$ is then defined by:
\begin{equation} \label{eq:control_h}
  h_s =
  \begin{cases}
    \psi_K(\mathcal{E}_n) \bigl( A^*_n (\beta  I + M_n)^{-1} J_n \vartheta_{2n}^\sharp \bigr)(s) & \text{ on } \mA_{2n}(s), \\
    0 & \text{ otherwise,}
  \end{cases}
\end{equation}
where the active set $\mA_{2n}(s)$ is the condition
\begin{equation*}
  s \in [2n, 2n+1), \quad z_{2n} \in \mathcal{C}_R, \quad \text{ and } \quad 2 n < \tau ,
\end{equation*}
$\beta > 0$ is a Tikhonov regularization parameter, and $\vartheta^\sharp$ is the
tangent projection
\begin{equation*}
  \vartheta^\sharp_s = (w_s - w'_s, P_{\pi_s^\perp} (\pi_s - \pi'_s)) = (w_s - w'_s, - P_{\pi_s^\perp} \pi'_s ).
\end{equation*}
While the control is not adapted at all times, it
is adapted at even integer times so that
\begin{equation*}
  z_{2n},\ z'_{2n},\ \vartheta_{2n},\ \{z_{2n} \in \mC_R\},\ \{ 2n < \tau\}, \ h\vert_{[0, 2n]}
  \quad \text{are } \mF_{2n}\text{-measurable} , \qquad \forall n \in \NN.
\end{equation*}

Note that the control is switched off when the energy
$\mathcal{E}_n$ is too large, because in that case we lose control on the bounds on $J_n$ and on $(\beta I + M_n)^{-1}$. The cutoff is a smooth function of
$\mathcal{E}_n$ rather than a sharp indicator, because $\mathcal{E}_n$ depends
on the noise over the whole cycle $[2n, 2n+2)$, and in
Section~\ref{sec:absolute_continuity} we need the control to be Malliavin
differentiable in order to show that it is an absolutely continuous
perturbation. The two restrictions $z_{2n} \in \mC_R$ and $2n < \tau$ imposed
above through indicators are of a different kind, because they are measurable at
the start of the cycle. For them Malliavin differentiability is not an issue.

\subsection{Asymptotic synchronization}\label{sec:one_step}

Now we pass to proving that the difference $\vartheta$ is contracting under the
effect of the control. To see the contraction we condition on the \lqm good\rqm event
\begin{equation}\label{eq:good_event_EKn}
  E_K^{n} \eqdef
  \left\{ \left( \int_{2n}^{2n+2} \|w_s\|_{H^1}^4 + \|\pi_s\|_{H^1}^4 \ud s \right)^{\frac{1}{4}}
  \le K \right\}.
\end{equation}
For technical reasons this differs slightly from the energy $\mE_n$ introduced
above. We note that in the following we will contract the tangent component
$\vartheta^\sharp$ of the difference $\vartheta$. This is sufficient because:
\begin{equation}\label{eq:radial_defect}
  \| \vartheta_t - \vartheta^\sharp_t  \|= \langle \pi_t, \pi_t - \pi'_t \rangle = 1 - \langle \pi_t, \pi'_t\rangle
  = \frac12 \|\pi_t - \pi'_t\|^2 \le \frac{1}{2}\|\vartheta_t\|^2,
\end{equation}
which is of lower order with respect to $\|\vartheta^\sharp_t\| \simeq
\|\vartheta_t\|$ if $\|\vartheta_t \|< 1$. The following result controls the
difference first on the time interval $[2n, 2n+1]$ where high modes are not yet
compressed by dissipation and then on the
interval $[2n+1, 2n+2]$, where the high modes are dissipated.

\begin{proposition}\label{prop:one_step_contraction}
  The following holds uniformly over $\delta \in (0, 1)$.
  For any $K, R>1, N \in \NN$ and $\delta_0 \in (0,1)$, there exists a
  $\beta(K,R,N,\delta_0) > 0$ such that for any $n \in \NN, z_{2n} \in
  \mathcal{C}_R$ and if $ 2n < \tau$, then the difference
  $\vartheta_{2n+1}$ decomposes as:
  \begin{equation}\label{eq:one_step_decomp}
    \vartheta_{2n +1} = L_n + H_n + R_n ,
  \end{equation}
  with
  \begin{equation}
    L_n = P_N \beta(\beta I + M_n)^{-1} J_n  \vartheta^\sharp_{2n}, \quad H_n = Q_N \beta(\beta I + M_n)^{-1} J_n \vartheta^\sharp_{2n},
  \end{equation}
  and $R_n = \vartheta_{2n+1}-L_n - H_n$. Then
  \begin{enumerate}[label=(\roman*)]
    \item On the event $E_K^n$
      \begin{equation}\label{eq:low_mode_bound}
        \EE\Big[\| L_n \|^2 \cdot \mathbf{1}_{E_K^n}\Big|\mF_{2n}\Big] \le \delta_0^2 \| \vartheta_{2n}\|^2 .
      \end{equation}
    \item  There exists a deterministic constant $C_K> 0$ such that on the event $E_K^n$
      \begin{equation}\label{eq:high_mode_bound}
        \| H_n \| \le \| J_n\vartheta^\sharp_{2n}\| \le C_K \|\vartheta_{2n}\| ,
      \end{equation}
    \item There exists a deterministic constant $C(R,
      K, \beta)>0$ such that on the event $E_K^n$
      \begin{equation}\label{eq:remainder_bound}
        \|R_n\| \le C(R, K, \beta) \|\vartheta_{2n}\|^2 .
      \end{equation}
  \end{enumerate}
  Moreover, for every $\delta_1 \in (0,1)$ there exist
  $N(\delta_1, K) > 0$ and a $\beta(\delta_1, K, R) > 0$, such that
    \begin{equation}\label{eq:one_step_second_moment}
    \EE [ \| \vartheta_{2n+2} \|^2 \mathbf{1}_{E_K^n} \mid \mF_{2n}] \leq \delta_1 \| \vartheta_{2n} \|^2 + C(\delta_1, R, K) \| \vartheta_{2n} \|^4 .
  \end{equation}
\end{proposition}

\begin{proof}
  The difference $\vartheta_{2n+1} = \Phi_{2n, 2n+1}(z_{2n}, W) - \Phi_{2n, 2n+1}(z'_{2n}, W + \int h)$ can be written as
  \begin{equation}\label{eq:taylor_expansion}
    \vartheta_{2n+1} = J_n\vartheta_{2n}^\sharp - A_n h + R_n ,
  \end{equation}
  where $J_n$ is the Jacobian of the flow $\Phi$ with respect to initial
  data and $A_n$ is the Malliavin derivative operator of
  the discussion after~\eqref{eq:control_h}, both along the first copy, and
  $R_n$ is defined by the identity.

  Now we use that on the event $E_K^n$ we have
  $\psi_K(\mE_n)=1$, since  on $E_K^n$ we have that $\mE_n \le 2K^2$, so that $h = A^*_n (\beta I + M_n)^{-1} J_n
  \vartheta^\sharp_{2n}$. Substituting this into \eqref{eq:taylor_expansion} and using
  $A_nA^*_n = M_n$, we obtain on $E_K^n$
  \begin{equation}\label{eq:tikhonov_identity}
    J_n\vartheta^\sharp_{2n} - A_nh
    = \bigl(I - M_n(\beta I + M_n)^{-1}\bigr)J_n\vartheta^\sharp_{2n}
    = \beta(\beta I + M_n)^{-1}J_n\vartheta^\sharp_{2n} .
  \end{equation}
  Splitting into high and low frequencies gives the decomposition \eqref{eq:one_step_decomp}.

  Since $M_n$ is self-adjoint and nonnegative, $\beta(\beta I + M_n)^{-1}$ has spectrum in $[0,1]$, so it is self-adjoint with operator norm at most one.
  The two leading terms of the decomposition read \[ L_n = P_N \beta(\beta I + M_n)^{-1} J_n
    \vartheta^\sharp_{2n}, \qquad H_n = Q_N \beta(\beta I + M_n)^{-1} J_n
  \vartheta^\sharp_{2n}. \]
  The low-mode term in part~(i) is a
  consequence of Lemma~\ref{lem:tikhonov_contraction} and the time homogeneity
  of $z_t$.

  For the high-mode term in part~(ii) we use that
  $\|Q_N \beta(\beta I + M_n)^{-1}\|_{\mathrm{op}} \le 1$ and therefore \[ \|H_n\| = \|Q_N \beta(\beta I + M_n)^{-1}
    J_n \vartheta^\sharp_{2n}\| \le \|J_n
  \vartheta^\sharp_{2n}\| \le C_K \|\vartheta_{2n}\| . \]
  The Jacobian is bounded pathwise on $E_K^n$ by
  Lemma~\ref{lem:fibre_dissipation}(iii).

  For part~(iii) we expand around $z_t$. Write the joint
  system~\eqref{eq:sns_vorticity_intro}--\eqref{eq:proj-spde} as
  \[
    \partial_t z = F(z) + (\curl\xi, 0) , \quad z = (w, \pi)
  \]
  where $F$ is a polynomial of degree at most four, and define the remainder of its linearization by
  \begin{equation}\label{eq:exact_remainder}
    \mQ_z(\vartheta) \eqdef F(z - \vartheta) - F(z) + DF(z)\vartheta .
  \end{equation}
  If we subtract the
  two equations, then the difference solves
  \begin{equation}\label{eq:difference_equation}
    \partial_t \vartheta_t = DF(z_t)\,\vartheta_t - \mQ_{z_t}(\vartheta_t)
    - (Q^{1/2} h_t, 0) ,
  \end{equation}
  and Duhamel's formula with the linearized flow $J_{r,2n+1}$ along $z$, compared with~\eqref{eq:taylor_expansion}, gives
  \begin{equation}\label{e:mild-rbar}
    R_n = J_n\bigl( \vartheta_{2n} - \vartheta^\sharp_{2n} \bigr)
    - \int_{2n}^{2n+1} J_{r, 2n+1}\, \mQ_{z_r}(\vartheta_r) \ud r ,
  \end{equation}
  where every coefficient is evaluated on the trajectory of $z$. The first term has norm at most $C_K\|\vartheta_{2n}\|^2$ by~\eqref{eq:radial_defect} and Lemma~\ref{lem:fibre_dissipation}(iii), so it suffices to bound the integral.
  Now following the same steps as in Lemma~\ref{lem:continuity_in_data} we obtain
  from~\eqref{eq:difference_equation} the energy estimate
  \begin{equation}\label{eq:difference_apriori}
    \sup_{r \in [2n, 2n+1]} \|\vartheta_r\|^2
    + \int_{2n}^{2n+1} \|\vartheta_r\|_{H^1}^2 \ud r
    \ \le\ C_K \bigl( \|\vartheta_{2n}\|^2 + \|h\|_{L^2}^2 \bigr)
    \ \le\ C(R, K, \beta)\, \|\vartheta_{2n}\|^2 ,
  \end{equation}
  where the norm of $h$ is controlled in Lemma~\ref{lem:cm_regularity} below. We claim that
    \begin{equation}\label{eq:remainder_integrand}
      \| \mQ_{z}(\vartheta) \|
      \ \le\ C\, \Gamma(z)^{1/2}\, \|\vartheta\| \|\vartheta\|_{H^1}
      + C \|\vartheta\|_{H^1}^2
    \end{equation}
  for every state $z = (w, \pi)$ and every $\vartheta \in \mH \times \mH$, with $\Gamma$ as in~\eqref{eq:gamma_r}.

  Let us write $\vartheta = (\tilde w,\tilde\pi)$ with $\tilde w = w - w'$ and
  $\tilde\pi = \pi - \pi'$. Then one can compute that
  \begin{equation}\label{eq:remainder_terms}
    \mQ_z(\vartheta) = \bigl( -B(\tilde w,\tilde w) , \ -L^0_{\tilde w}\tilde\pi + \mathrm{I} \tilde\pi
    - (\mathrm{II}+\mathrm{III})\pi' \bigr)
  \end{equation}
  with
  \begin{equation*}
    \begin{aligned}
      \mathrm{I} & = \langle\pi,L^0_{\tilde w}\pi\rangle - \langle \tilde\pi,G_w\pi\rangle
      - \langle\pi,G_w\tilde\pi\rangle , \\
      \mathrm{II} & = \langle \tilde\pi,G_w\tilde\pi\rangle - \langle \tilde\pi,L^0_{\tilde w}\pi\rangle
      - \langle\pi,L^0_{\tilde w}\tilde\pi\rangle , \\
      \mathrm{III} & = \langle \tilde\pi,L^0_{\tilde w}\tilde\pi\rangle .
    \end{aligned}
  \end{equation*}
  We estimate these terms with~\eqref{eq:Gw_scalars},~\eqref{eq:Gw_bound} and,
  for $a, f, g \in H^1$,
  \begin{equation}\label{eq:B_bounds}
    \begin{aligned}
      \|B(f,g)\| &\lesssim \|f\|^{1/2}\|f\|_{H^1}^{1/2}\|g\|_{H^1} ,\\
      |\langle a,B(f,g)\rangle| & \lesssim
      \|a\|^{1/2}\|a\|_{H^1}^{1/2}\|f\|\,\|g\|_{H^1} , \\
      \langle a,B(f,a)\rangle & = 0 .
    \end{aligned}
  \end{equation}
  Now we find
  \begin{equation}
    \begin{aligned}
      \|B(\tilde w,\tilde w)\| & \lesssim \|\tilde w\|^{1/2}\|\tilde w\|_{H^1}^{3/2} \lesssim \|\vartheta\|_{H^1}^2 ,\\
      \|L^0_{\tilde w}\tilde\pi\| & \lesssim \|\tilde w\|^{1/2}\|\tilde w\|_{H^1}^{1/2}\|\tilde\pi\|_{H^1}
      + \|\tilde\pi\|^{1/2}\|\tilde\pi\|_{H^1}^{1/2}\|\tilde w\|_{H^1} \lesssim \|\vartheta\|_{H^1}^2 .
    \end{aligned}
  \end{equation}
  This leaves us with estimating only the terms $\mathrm{I}$, $\mathrm{II}$ and $\mathrm{III}$, where we
  write $L^0_f g = B(f,g) + B(g,f)$. Here we find:

  The term $\mathrm{I}$ enters $\mQ_z(\vartheta)$ through $\mathrm{I} \tilde\pi$, whose norm
  is $|\mathrm{I}|\,\|\tilde\pi\| \le |\mathrm{I}|\,\|\vartheta\|$. We find $\langle\pi,B(\tilde w,\pi)\rangle =
  0$ and the remaining terms are (by~\eqref{eq:B_bounds} and~\eqref{eq:Gw_scalars}):
  \begin{equation*}
    \begin{aligned}
      |\langle\pi,B(\pi,\tilde w)\rangle| &\lesssim \|\pi\|_{H^1}^{1/2}\|\tilde w\|_{H^1}
      \lesssim \Gamma^{1/2}\|\vartheta\|_{H^1} , \\
      |\langle \tilde\pi,G_w\pi\rangle| + |\langle\pi,G_w\tilde\pi\rangle|
      &\lesssim \bigl(\|\pi\|_{H^1}+\|w\|_{H^1}\bigr)\|\tilde\pi\|_{H^1}
      \lesssim \Gamma^{1/2}\|\vartheta\|_{H^1} ,
    \end{aligned}
  \end{equation*}
  Hence
  $\|\mathrm{I} \tilde\pi\| \lesssim \Gamma^{1/2}\|\vartheta\|\,\|\vartheta\|_{H^1}$.
  The term $\mathrm{II}$ enters through $\mathrm{II}\pi'$, whose norm is $|\mathrm{II}|$.
  We find
  \begin{equation*}
    \begin{aligned}
      |\langle \tilde\pi,G_w\tilde\pi\rangle| &\lesssim \|\tilde\pi\|_{H^1}^2
      + \|w\|_{H^1}\|\tilde\pi\|\,\|\tilde\pi\|_{H^1}
      \lesssim \|\vartheta\|_{H^1}^2 + \Gamma^{1/2}\|\vartheta\|\,\|\vartheta\|_{H^1} , \\
      |\langle \tilde\pi,B(\tilde w,\pi)\rangle| &\lesssim \|\tilde\pi\|^{1/2}\|\tilde\pi\|_{H^1}^{1/2}
      \|\tilde w\|\,\|\pi\|_{H^1} \lesssim \Gamma^{1/2}\|\vartheta\|\,\|\vartheta\|_{H^1} , \\
      |\langle \tilde\pi,B(\pi,\tilde w)\rangle| &\lesssim \|\tilde\pi\|^{1/2}\|\tilde\pi\|_{H^1}^{1/2}
      \|\tilde w\|_{H^1} \lesssim \|\vartheta\|_{H^1}^2 , \\
      |\langle\pi,B(\tilde w,\tilde\pi)\rangle| + |\langle\pi,B(\tilde\pi,\tilde w)\rangle|
      &\lesssim \|\pi\|_{H^1}^{1/2}\bigl(\|\tilde w\|\,\|\tilde\pi\|_{H^1}
      + \|\tilde\pi\|\,\|\tilde w\|_{H^1}\bigr) \lesssim \Gamma^{1/2}\|\vartheta\|\,\|\vartheta\|_{H^1} ,
    \end{aligned}
  \end{equation*}
  where the first estimate comes from~\eqref{eq:Gw_bound} at $a = b = \tilde\pi$ and the
  others follow from~\eqref{eq:B_bounds}, together with
  $\|\vartheta\|^{3/2}\|\vartheta\|_{H^1}^{1/2} \le \|\vartheta\|\,\|\vartheta\|_{H^1}$ and
  $\|\vartheta\|^{1/2}\|\vartheta\|_{H^1}^{3/2} \le \|\vartheta\|_{H^1}^2$. Hence

  \[
    |\mathrm{II}| \lesssim \Gamma^{1/2}\|\vartheta\|\,\|\vartheta\|_{H^1} + \|\vartheta\|_{H^1}^2.
  \]

  The term $\mathrm{III}$ is cubic and enters through $\mathrm{III}\pi'$. Since $\langle
  \tilde\pi,B(\tilde w,\tilde\pi)\rangle=0$ we are left with
  \begin{equation*}
    |\langle \tilde\pi,B(\tilde\pi,\tilde w)\rangle| \lesssim \|\tilde\pi\|^{1/2}\|\tilde\pi\|_{H^1}^{1/2}
    \|\tilde\pi\|\,\|\tilde w\|_{H^1} \lesssim \|\vartheta\|_{H^1}^2 ,
  \end{equation*}
  where we used $\|\tilde\pi\| \le 2$. Hence $|\mathrm{III}| \lesssim \|\vartheta\|_{H^1}^2$, and we have proven~\eqref{eq:remainder_integrand}.

  Now we estimate the integral in~\eqref{e:mild-rbar}
  with~\eqref{eq:remainder_integrand}, and use
  $\|J_{r,2n+1}\|_{\mathrm{op}} \le C_K$ from
  Lemma~\ref{lem:fibre_dissipation}(iii), then use Cauchy--Schwarz in time
  and $\int \Gamma_r \ud r \le C(1+\mE_n)$ and~\eqref{eq:difference_apriori}, so we
  obtain $\|R_n\| \ \le\ C(R, K, \beta)\|\vartheta_{2n}\|^2$, which
  is~\eqref{eq:remainder_bound}.

  This completes the proof of (iii). We now proceed with the last step, which is to
  control the difference on $[2n+1, 2n+2]$.
  Through first-order expansion over the free phase $[2n+1, 2n+2]$ we rewrite $\vartheta_{2n+2}$ as
  $J^\flat_n L_n + J^\flat_n H_n$ plus a remainder, which collects $J^\flat_n
  R_n$ and the second-order term of the expansion. We bound the three
  contributions
  conditionally on $\mathcal F_{2n}$ and on $E^n_K$, using
  Lemma~\ref{lem:fibre_dissipation}(iii), which implies
  $\|J^\flat_n\|_{\mathrm{op}}\le C_K$, so that together with part~(i) we obtain
  \begin{equation*} \EE\bigl[ \|J^\flat_n L_n\|^2 \mathbf 1_{E_K^n} \bigm| \mathcal F_{2n} \bigr] \ \le\ C_K^2\delta_0^2\|\vartheta_{2n}\|^2 .
  \end{equation*}
  For the high modes, we apply Lemma~\ref{lem:jacobian_moments} to the vector $H_n = Q_N \beta(\beta I + M_n)^{-1} J_n\vartheta^\sharp_{2n}$ with $T = 2$, and $\tau_0 = 1$ and $s = 1 \geq \tau_0/2$
  \begin{equation*} \|J^\flat_n H_n\| \ \le\ \delta_0 \|\beta(\beta I + M_n)^{-1} J_n\vartheta^\sharp_{2n}\| \ \le\ \delta_0 C_K\|\vartheta_{2n}\| ,
  \end{equation*}
  provided that $N$ is chosen sufficiently large depending on $\delta_0$.
  Finally, by the estimates of the previous step the remainder is bounded by
    \begin{equation*}
      C_K\|R_n\| + C(R,K)\|\vartheta_{2n+1}\|^2 \ \le\ C(R,K,\beta)\|\vartheta_{2n}\|^2 .
  \end{equation*}
  To conclude, we obtain
  \begin{equation*}
    \EE\big[\|\vartheta_{2n+2}\|^2\mathbf 1_{E_K^n}\big|\mathcal F_{2n}\big]
    \leq C_K \delta_0^2 \|\vartheta_{2n}\|^2
    + C(R,K,\beta)^2\|\vartheta_{2n}\|^4.
  \end{equation*}
  This completes the proof.
\end{proof}
Now we will build on Proposition~\ref{prop:one_step_contraction} in order to
obtain asymptotic synchronization, iterating the contraction over numerous
steps. Before we prove the synchronization, we will prove some supporting
statements. The first one proves a bound on how much the difference $\rho$ can grow.

\begin{lemma}\label{lem:active_step}
  For every $R, K, \beta > 0$ there is a deterministic index
  $n_0 = n_0(R, K, \beta)$ such that
  for every $n \ge n_0$ with $2n < \tau$, $z_{2n} \in \mC_R$ and
  $\psi_K(\mE_n) > 0$,
  \begin{equation}\label{eq:active_step}
    \rho_{2n+2} \ \le\ e^{C(1 + K^2)}\rho_{2n},
  \end{equation}
  for a constant $C$ that does not depend on any parameter of the problem.
\end{lemma}

\begin{proof}
  Write $\psi = \psi_K(\mE_n) \in (0,1]$. The Taylor identity~\eqref{eq:taylor_expansion} with
  the control $h = \psi\, A^*_n(\beta I + M_n)^{-1}J_n\vartheta^\sharp_{2n}$
  gives
  \begin{equation*}
    \vartheta_{2n+1}
    \ =\ \bigl( I - \psi M_n(\beta I + M_n)^{-1} \bigr) J_n \vartheta^\sharp_{2n} + R_n
    \ =\ \Bigl( (1-\psi) I + \psi\, \beta(\beta I + M_n)^{-1} \Bigr)
    J_n \vartheta^\sharp_{2n} + R_n .
  \end{equation*}
  Since $M_n$ is self-adjoint and nonnegative 
  $\|\beta(\beta I + M_n)^{-1}\|_{\mathrm{op}} \le 1$ for every $\beta > 0$. Therefore
  \begin{equation}\label{eq:active_step_convex}
    \|\vartheta_{2n+1}\| \ \le\ \|J_n \vartheta^\sharp_{2n}\| + \|R_n\| .
  \end{equation}
  On $\{\psi_K(\mE_n) > 0\}$ the cutoff gives $\mE_n \le 2(K+1)^2$, and
  Lemma~\ref{lem:fibre_dissipation}  guarantees
  \begin{equation}\label{eq:active_step_jac}
    \|J_n\|_{\mathrm{op}}^2
    \ \le\ C \exp\Bigl( C\!\int_{2n}^{2n+2}\Gamma_r \ud r \Bigr)
    \ \le\ C \exp\bigl( 2 C \bigl( 1 + (K+1)^2 \bigr) \bigr) ,
  \end{equation}
  which is of the correct order.
  For the remainder,~\eqref{eq:remainder_bound} gives
  $\|R_n\| \le C(R,K,\beta)\|\vartheta_{2n}\|^2$. Since $2n < \tau$ we have
  $\rho_{2n} \le e^{-2n}$, so
  \begin{equation*}
    \|R_n\| \ \le\ C(R,K,\beta)\, e^{-2n}\, \rho_{2n}
    \ \le \rho_{2n}
    \qquad\text{as soon as}\qquad
    e^{-2n} \ \le\ \frac{1}{C(R,K,\beta)} ,
  \end{equation*}
  which holds for every $n \ge n_0$, with $n_0$ appropriate.
\end{proof}

Next we recall an estimate on martingale differences. If $m = 1$ this is
\cite{Kifer86}*{Chapter~III, Lemma~2.1}. The case for general $m$ follows identically.
  \begin{lemma}\label{lem:lln_lag}
    Let $m \ge 1$ be an integer and let $(X_k)_{k \ge 0}$ be random variables such that
    $X_k$ is $\mF_{2k+2}$-measurable and $\sup_k \EE[X_k^2] < \infty$. Then
    \begin{equation}\label{eq:lln_lag}
      \lim_{n \to \infty} \frac1n \sum_{k < n}
      \Bigl( X_k - \EE\bigl[ X_k \bigm| \mF_{2k - 2m} \bigr] \Bigr) \ =\ 0
      \qquad \text{almost surely.}
    \end{equation}
  \end{lemma}

  Next introduce $\widehat V \eqdef V_\rho^{1/2}$ for
  the square root of the functional~\eqref{eq:lyapunov_S4}, at some $\rho > 1$ fixed later.
  \begin{lemma}\label{lem:occupation}
    There exists a constant $C_\star >0 $ (independent of all other parameters) such that, for every $K, \beta > 1$, almost
    surely
    \begin{equation}\label{eq:occupation_bound}
      \limsup_{n \to \infty} \frac1n \sum_{k<n}
        \bigl( \widehat V(u_{2k}) + \widehat V(u'_{2k}) \bigr) \ \le\  C_\star .
    \end{equation}
  \end{lemma}
  \begin{proof}
    Fix the rate in~\eqref{eq:super_lyapunov} to be $\gamma \eqdef \log 2$ and set
    $C_\star \eqdef C_\gamma^{1/2}$. On $\{2n < \tau\}$
    Lemma~\ref{lem:cm_regularity} and $\rho_{2n} \le \delta e^{-2n}$ give us
    \[
      \|h_n\|_{H_{\mC\mM}} \ \le\ C_K \beta^{-1/2} \delta e^{-2n} \ \le\ 1 ,
      \qquad n \ge n_0 ,
    \]
    for a deterministic $n_0 = n_0(K, \beta, \delta)$. So for $t \ge 2n_0$ the
    shift $h$ is
    bounded uniformly in all parameters, and $u'$, driven by $W + \int h$, obeys
    the same estimates as $u$, up to slightly changing all the constants. By
    turns~\eqref{eq:super_lyapunov} and Jensen, we then obtain
    \begin{equation}\label{eq:bwei}
      \EE\bigl[ \widehat V(u_{t+2}) \bigm| \mF_t \bigr]
        \le \frac12 \widehat V(u_t) + C_\star ,
      \qquad
      \EE\bigl[ \widehat V(u'_{t+2}) \bigm| \mF_t \bigr]
        \le \frac12 \widehat V(u'_t) + C_\star ,
      \qquad \forall t \ge 2n_0 .
    \end{equation}
    Denoting
    \[
    b_k \eqdef \widehat V(u_{2k}) + \widehat V(u'_{2k}),
    \qquad
    d_k \eqdef b_k - \EE[b_k \mid \mF_{2k-2}],
    \]
    we have
    \[
    b_k \ \le\ \frac12 b_{k-1} + 2C_\star + d_k .
    \]
    Hence, summing over $n_0 < k \le n$ and discarding $b_n \ge 0$,
    \[
      \frac12 \sum_{n_0 \le k < n} b_k
      \ \le\ b_{n_0} + 2 C_\star n + \sum_{n_0 < k \le n} d_k .
    \]
    Since
    $b_k \in \mF_{2k}$ with $\sup_k \EE[b_k^2] < \infty$ by~\eqref{eq:super_lyapunov},
    Lemma~\ref{lem:lln_lag} at $m = 1$ gives $n^{-1}\sum_{k \le n} d_k \to 0$.
    This completes the proof.
  \end{proof}

Recall the stopping time $\tau$ in \eqref{eq:stopping_time_tau},
which we use to measure synchronization. The following is the main result of
this section.
\begin{proposition}
  \label{prop:asymptotic_contraction}
    There exist $R, K, \beta > 1$ with the following property. For every
    $\delta \in (0,1)$ there is a $\delta_\star \in (0, \delta)$ such that, with $\tau$ the
    stopping time~\eqref{eq:stopping_time_tau} at that $\delta$,
    \begin{equation}\label{eq:tau_infinite}
      \PP( \tau = \infty ) \ >\ 0 , \qquad
      \forall z_0, z'_0 \in \mC_R \ \text{ with } \ \|z_0 - z'_0\| \le \delta_\star .
    \end{equation}
    In particular the two copies synchronize with positive probability,
    \begin{equation}\label{eq:synchronisation}
      \PP\Bigl( \lim_{n \to \infty}\|z_{2n} - z'_{2n}\| = 0 \Bigr) \ >\ 0 .
    \end{equation}
\end{proposition}

\begin{proof}[Proof of Proposition~\ref{prop:asymptotic_contraction}]

    Fix a parameter $a \ge 4$ (the choice is irrelevant in the proof and the
    parameter will not be optimized over), and consider the energy of the forced process:
  \begin{equation*}
    \mE'_n \eqdef \int_{2n}^{2n+2} \|w'_s\|_{H^1}^2 + \|\pi'_s\|_{H^1}^2 \ud s.
  \end{equation*}
  Recall the good event $E^n_K$ of~\eqref{eq:good_event_EKn}, on which the energy is bounded
    $\mE_n \le 2K^2$. The proof balances the following two estimates (a
    contraction on a good event, and a uniform upper bound in any event):
\begin{equation*}
  \rho_{2n+2} \le e^{-a}\rho_{2n} \quad\text{likely on } \{z_{2n} \in \mC_R\} \cap E^n_K ,
  \qquad
  \rho_{2n+2} \le e^{C(1 + \mE_n + \mE'_n + K^2)}\rho_{2n} \quad\text{always,}
\end{equation*}
    the first by Proposition~\ref{prop:one_step_contraction} and the second by
    Lemmas~\ref{lem:base_difference} and~\ref{lem:fibre_difference}(ii), and
    Lemma~\ref{lem:active_step} (see also the discussion further down).

    In the following proof the parameters are fixed in the order
    \begin{equation}\label{eq:param_order}
      R \ \longrightarrow\ \eta \ \longrightarrow\ K, \delta_1
      \ \longrightarrow\ N, \beta ,
    \end{equation}
    each depending only on those before it (note that we do not need to fix $\delta$). For each
    $\delta \in (0,1)$ we then choose $n_0$ and $\delta_\star \ll \delta$, both depending on $\delta$. 
    Now, denote by $F^n_{K}$ the contracting event
\begin{equation}
      F^n_{K} \eqdef \{ \rho_{2n+2} \le e^{-a} \rho_{2n} \} \cap E^n_K ,
\end{equation}
    which, unlike $\{z_{2n} \in \mC_R\}$, is not $\mF_{2n}$-measurable. For any threshold
    $\eta \in (0,1)$, Lemma~\ref{lem:short_time_energy_proof} at $p = 8$ gives a constant
    $C_R$ with
\begin{equation}\label{eq:energy_decay_conditioning}
      \PP\bigl( (E^n_K)^c \bigm| \mF_{2n} \bigr)
      \ \le\ C_R K^{-8} \ \le\ \frac{1}{2}\eta
      \qquad \text{on } \{ z_{2n} \in \mC_R \} ,
\end{equation}
    where in the last inequality we set $K \eqdef \bigl( 2C_R/\eta \bigr)^{1/8}$ so that
    $C_R K^{-8} = \tfrac12\eta$. On $\{z_{2n} \in \mC_R\} \cap \{2n < \tau\}$ the
    hypotheses of Proposition~\ref{prop:one_step_contraction} hold. Therefore we apply the
    second moment estimate~\eqref{eq:one_step_second_moment} with
    $\delta_1 \eqdef \tfrac14 \eta e^{-2a}$, fixing $N=N(\delta_1, K)$ and
    $\beta=\beta(\delta_1, K, R)$ as functions of $R, \eta$ alone. Since
    \[
      \{\rho_{2n+2} > e^{-a}\rho_{2n}\} \cap E^n_K \subseteq \{ \rho_{2n+2}^2 \mathbf 1_{E^n_K} > e^{-2a}\rho_{2n}^2 \},
    \]
    Markov's inequality followed by~\eqref{eq:one_step_second_moment} gives on $\{ z_{2n} \in \mC_R \} \cap \{ 2n < \tau \}$
\begin{equation*}
      \begin{aligned}
        \PP\bigl( \{ \rho_{2n+2} > e^{-a}\rho_{2n} \} \cap E^n_K \bigm| \mF_{2n} \bigr)
        \ &\le\ \frac{e^{2a}}{\rho_{2n}^2}
        \EE\bigl[ \rho_{2n+2}^2 \mathbf 1_{E^n_K} \bigm| \mF_{2n} \bigr] \\
        \ &\le\ \frac\eta4 + e^{2a} C(\delta_1, R, K)\rho_{2n}^2.
      \end{aligned}
\end{equation*}
    On $\{ 2n < \tau \}$ we have $\rho_{2n} \le \delta e^{-2n}$, so the second term with
    $C(\delta_1, R, K)$ can be bounded by
    \[
      e^{2a} C(\delta_1, R, K)\rho_{2n}^2 \le e^{2a} C(\delta_1, R, K)\delta^2 e^{-4n} \le \frac{1}{4}\eta
    \]
    for every $n \ge n_0$, with $n_0 = n(R,\eta,\delta)$ suitably large and deterministic. Since
    \[
      F^{n,c}_{K} \subseteq \bigl( \{\rho_{2n+2} > e^{-a}\rho_{2n}\} \cap E^n_K \bigr) \cup (E^n_K)^c,
    \]
    adding this to the bound \eqref{eq:energy_decay_conditioning} gives
\begin{equation}\label{eq:bcon}
      \PP\bigl( F^n_{K} \bigm| \mF_{2n} \bigr) \ \ge\ 1-\eta
      \qquad \text{on } \{ z_{2n} \in \mC_R \} \cap \{ 2n < \tau \} .
\end{equation}

    Off $F^n_K$ we only upper bound the growth. For $n \ge n_0$,
\begin{equation*}
  \rho_{2n+2} \le e^{C(1+\mE_n+\mE'_n)}\rho_{2n} \quad\text{if } \psi_K(\mE_n) = 0 ,
  \qquad
  \rho_{2n+2} \le e^{C(1+K^2)}\rho_{2n} \quad\text{if } \psi_K(\mE_n) > 0 ,
\end{equation*}
    by Lemmas~\ref{lem:base_difference} and~\ref{lem:fibre_difference}(ii) in the first case
    and Lemma~\ref{lem:active_step} in the second. Since the control acts only on
    $\{z_{2n} \in \mC_R\}$, both are covered by
\begin{equation}\label{eq:bgro}
      \log\frac{\rho_{2n+2}}{\rho_{2n}} \ \le\ Z_n \eqdef
      C\bigl( 1 + \mE_n + \mE'_n \bigr) + C\bigl( 1 + K^2 \bigr)\mathbf 1_{z_{2n} \in \mC_R} , \qquad n \geq n_0,
\end{equation}
    with $C$ uniform over all parameters. With the contraction on $F^n_K$ this gives
    \begin{equation}\label{eq:binc}
      \log\frac{\rho_{2n+2}}{\rho_{2n}}
      \ \le\ -a\, \mathbf 1_{F^n_{K}} + Z_n \mathbf 1_{F^{n,c}_{K}}
      \ =\ -a + \zeta_n , \qquad
      \zeta_n \eqdef (a + Z_n)\, \mathbf 1_{F^{n,c}_{K}} \ \ge\ 0 ,
      \qquad \forall n \ge n_0 .
    \end{equation}

    For $n<n_0$, the bound~\eqref{eq:bgro} is unavailable, but $2n < \tau$ gives
    $\rho_{2n} \le \delta$, so $\|h_n\|^2 \le C_K\beta^{-1}\delta^2$ by
    Lemma~\ref{lem:cm_regularity} and Lemmas~\ref{lem:base_difference}
    and~\ref{lem:fibre_difference}(ii) give
\begin{equation*}
  \rho_{2n+2} \ \le\ \exp\Bigl( C(1+\mE_n+\mE'_n)
    + \frac12\log\bigl(1+C_K\beta^{-1}\bigr) \Bigr)\rho_{2n} .
\end{equation*}
    Since $n_0$ is deterministic these cycles are finitely many, so the total growth they
    produce,
    \begin{equation}\label{e:def-Sigma0}
      \Sigma_0 \ \eqdef\ \log\frac{\rho_{2n_0}}{\rho_0} \ +\ a n_0 ,
\end{equation}
    is almost surely finite, the offset $a n_0$ absorbing the shift from starting at $n_0$.
    Summing~\eqref{eq:binc} over $n_0 \le k < n$ and using $\zeta_k \ge 0$ to extend the sum
    to every $k < n$,
    \begin{equation}\label{eq:bwalk}
      \rho_{2n} \ \le\ \rho_0\, e^{\Sigma_0}\,
      \exp\Bigl( -a n + \sum_{k<n}\zeta_k \Bigr) ,
      \qquad 2n < \tau .
    \end{equation}
    It remains to bound the long time average of $\zeta$.

    We bound $\zeta_k$ by conditioning first on $\mF_{2k}$. On $\{z_{2k} \in \mC_R\}$,
    Cauchy--Schwarz with~\eqref{eq:bcon} and the moments of
    Lemma~\ref{lem:short_time_energy_proof} give
    $c_\eta \eqdef \bigl( a + C_R + C(1+K^2) \bigr)\eta^{1/2}$, while on
    $\{z_{2k} \notin \mC_R\}$, $\mathbf 1_{F^{k,c}_K} \le 1$ and
    $Z_k = C(1+\mE_k+\mE'_k)$, so
\begin{equation}\label{eq:bin}
      \EE\bigl[ \zeta_k \bigm| \mF_{2k} \bigr]
      \ \le\ c_\eta \mathbf 1_{z_{2k} \in \mC_R}
      + \EE\bigl[ a + C(1 + \mE_k + \mE'_k) \bigm| \mF_{2k} \bigr]
        \mathbf 1_{z_{2k} \notin \mC_R} .
\end{equation}
    Since $K = (2C_R/\eta)^{1/8}$ we have $K^2\eta^{1/2} = (2C_R)^{1/4}\eta^{1/4}$, so
    $c_\eta \to 0$ as $\eta \to 0$ at fixed $R$. For the second term,
\begin{equation*}
  \EE\bigl[ \|w_s\|_{H^1}^4 \bigm| \mF_{2k-2} \bigr] \lesssim \widehat V(u_{2k-2}) ,
  \qquad
  \EE\bigl[ \|\pi_s\|_{H^1}^4 \bigm| \mF_{2k-2} \bigr] \lesssim \widehat V(u_{2k-2}) ,
\end{equation*}
    by Proposition~\ref{prop:energy_enstrophy}(i) at $\rho$ large and by
    Lemma~\ref{lem:uniform_fibre_moment}, so that
\begin{equation}\label{eq:bmom}
  \EE\bigl[ (1 + \mE_k + \mE'_k)^2 \bigm| \mF_{2k-2} \bigr]
  \ \le\ C\bigl( \widehat V(u_{2k-2}) + \widehat V(u'_{2k-2}) \bigr) .
\end{equation}
    Throughout we assume that $t_0 \le 2$, for the $t_0$ of
    Lemma~\ref{lem:uniform_fibre_moment}, which only lightens the notation. The definition
    of $\mC_R$ and Chebyshev's inequality then guarantee
    \begin{equation*}
  \PP\bigl( z_{2k} \notin \mC_R \bigm| \mF_{2k-2} \bigr)
      \ \le\ \frac1{R^2}\EE\bigl[ ( \|w_{2k}\|_{H^{s_\star}} + \|\pi_{2k}\|_{H^1} )^2
  \bigm| \mF_{2k-2} \bigr]
      \ \le\ \frac C{R^2}\,\widehat V(u_{2k-2}) ,
    \end{equation*}
    so that taking $\EE[\,\cdot\mid\mF_{2k-2}]$ in~\eqref{eq:bin} and applying
    Cauchy--Schwarz,
\begin{equation}\label{eq:bout}
      \EE\bigl[ \zeta_k \bigm| \mF_{2k-2} \bigr]
      \ \le\ c_\eta + \frac CR\bigl( \widehat V(u_{2k-2}) + \widehat V(u'_{2k-2}) \bigr) .
\end{equation}

    Now $\zeta_k$ is $\mF_{2k+2}$-measurable with second moments bounded uniformly in $k$
    by~\eqref{eq:bmom}, so Lemma~\ref{lem:lln_lag} at $m = 1$ removes the conditional
    expectation. With~\eqref{eq:bout} and~\eqref{eq:occupation_bound},
    \begin{equation}\label{eq:blim}
      \limsup_{n \to \infty}\frac1n\sum_{k<n}\zeta_k \ \le\ c_\eta + \frac CR
      \qquad\text{almost surely.}
    \end{equation}
    Since $C$ is absolute we may choose $R$ with $C/R \le a/4$, and then $\eta$ with
    $c_\eta \le a/4$, so that the right side of~\eqref{eq:blim} is at most $a/2$.

    Finally, $S_n \eqdef -an + \sum_{k<n}\zeta_k$ satisfies $\limsup_n S_n/n \le -a/2 \le -2$
    by~\eqref{eq:blim}, so that
\begin{equation*}
      S^\star \eqdef \sup_{n \ge 0}\bigl( S_n + 2n \bigr)
\end{equation*}
    is almost surely finite. For $2n < \tau$~\eqref{eq:bwalk} reads
    $\rho_{2n} \le \rho_0 e^{\Sigma_0 + S^\star} e^{-2n}$, so
    $\{ \Sigma_0 + S^\star \le \log(\delta/\rho_0) \} \subseteq \{\tau = \infty\}$. As
    $\Sigma_0 + S^\star$ is almost surely finite there is a $\delta_\star \in (0, \delta)$ with
\begin{equation*}
      \PP( \tau = \infty )
      \ \ge\ \PP\bigl( \Sigma_0 + S^\star \le \log(\delta/\delta_\star) \bigr) \ >\ 0
\end{equation*}
    for every initial pair with $\rho_0 \le \delta_\star$, which is~\eqref{eq:tau_infinite}.

    Finally, on $\{\tau = \infty\}$ we have $\rho_{2n} \le \delta e^{-2n}$ for every $n$ by~\eqref{eq:stopping_time_tau}, so that $\{\tau = \infty\} \subseteq \{\rho_{2n} \to 0\}$ and~\eqref{eq:synchronisation} follows from~\eqref{eq:tau_infinite}.

\end{proof}

\subsection{Some moment bounds}
In this section we collect some moment estimates that are used in the proof of
the contraction.

\begin{lemma}\label{lem:short_time_energy_proof}
  For any $R > 0$ and $p \ge 1$, there exists a $ C_{p, R} > 0$ such that for any $z_{2n} = (w_{2n}, \pi_{2n}) \in \mC_R$:
  \begin{equation}
    \PP\left( \left(\int_{2n}^{2n+2}\|w_s\|_{H^1}^4 + \|\pi_s\|_{H^1}^4 \ud s \right)^{\frac{1}{4}} > K \Big| \mathcal{F}_{2n}\right) \le C_{p, R}  K^{-p}, \qquad \forall K>0.
  \end{equation}
\end{lemma}

\begin{proof}

  The result follows if we prove that there exists $C_{p, R}>0$ such that
  \begin{equation}\label{eq:algebraic_moment_pi}
    \EE\bigg[\bigg(\int_{2n}^{2n+2} \|w_s\|_{H^1}^4 + \|\pi_s\|_{H^1}^4 \ud s\bigg)^p \bigg| \mathcal{F}_{2n}\bigg] \le C_{p, R} .
  \end{equation}
  By Jensen we have
  \begin{equation}\label{eq:moment_reduction}
    \EE\bigg[\bigg(\int_{2n}^{2n+2}\|w_s\|_{H^1}^4+\|\pi_s\|_{H^1}^4\ud s\bigg)^{p}\bigg|\mathcal F_{2n}\bigg]
    \lesssim \int_{2n}^{2n+2}\Big(\EE\big[\|w_s\|_{H^1}^{4p}\big|\mathcal F_{2n}\big]+\EE\big[\|\pi_s\|_{H^1}^{4p}\big|\mathcal F_{2n}\big]\Big)\ud s ,
  \end{equation}
  so that it suffices to prove
    \begin{equation*}
      \sup_{s \in [2n, 2n+2]} \Bigl(
        \EE\bigl[\|w_s\|_{H^1}^{4p} \bigm| \mathcal F_{2n}\bigr]
      + \EE\bigl[\|\pi_s\|_{H^1}^{4p} \bigm| \mathcal F_{2n}\bigr] \Bigr)
      \ \le\ C_{p,R} , \qquad z_{2n} \in \mC_R .
  \end{equation*}
  For the vorticity, $z_{2n}\in\mC_R$ gives
  $\|w_{2n}\|_{H^{s_\star}}\le R$, so the moment estimates
  follow from the Lyapunov functional~\eqref{eq:super_lyapunov}.

  For the projective process, we use \cite{HPRY24}*{Lemma~6.4}, which guarantees that for any $p \ge1$ and $0\le s\le t$,
  \begin{equation}\label{eq:hpry64moment}
    \EE\big[\|\pi_t\|_{H^1}^p \big| \mathcal F_s\big]\le C(p,t-s)\big(\|\pi_s\|_{H^1}^{2p}+1\big)^{1/2}V_\rho(u_s)^{1/2},
    \qquad V_\rho(u)=(1+\|u\|_{H^{\beta_\star}}^2)^\rho e^{c_\star\|u\|_{H^1}^2},
  \end{equation}
  with $V_\rho$ the functional in~\eqref{eq:lyapunov_S4} at $\rho = p$, and
  $\beta_\star=s_\star+1$, as in~\eqref{eq:X_star}. 
\end{proof}
The next moment bound on the projective component is uniform in the initial data
of that component. This result follows from~\cite{HPRY24}.

  \begin{lemma}\label{lem:uniform_fibre_moment}
    There is a $t_0 > 0$ such that the following holds. For every
    $q \ge 1$,  and all $t_0 \le t_1 \le t_2 < \infty$ there are
    $C = C(q, t_1, t_2)$ and $\rho =
    \rho(q, t_1, t_2)$ such that
    \begin{equation}\label{eq:uniform_fibre_moment}
      \sup_{t \in [t_1, t_2]} \EE_{z}\bigl[ \|\pi_t\|_{H^1}^{q} \bigr]
      \ \le\ C V_\rho(u)^{1/2}.
    \end{equation}
  \end{lemma}
  \begin{proof}
From 
\cite{HPRY24}*{Lemma~3.7} there exists a deterministic time $t_0 >0$ and a
random time $\eta \le t_0$ such that by \cite{HPRY24}*{Proposition~3.8}, for
every $p \ge 1$ (and up to performing a Sobolev embedding) and some $\rho_0>0$
    \begin{equation}\label{eq:import_prop38}
      \EE_z\bigl[ \|\pi_\eta\|_{H^1}^{2p} \bigr]
      \ \le\ C_p \bigl( 1 + \|w\|_{H^{s_\star}} \bigr)^{\rho_0 p} ,
    \end{equation}
Next, \cite{HPRY24}*{Lemma~6.4}, guarantees that for $\zeta \in (2,3)$ and $0
\le s \le t $
    \begin{equation}\label{eq:import_transport}
      \|\pi_t\|_{H^1} \ \le\ C(\zeta)
      \exp\Bigl( C(\zeta) \int_s^t \|u_r\|_{H^{\zeta}} \ud r \Bigr)
      \Bigl( \|\pi_s\|_{H^1}
      + (t-s) \sup_{s \le r \le t} \|u_r\|_{H^{\zeta+1}} \Bigr).
    \end{equation}
Finally \cite{HPRY24}*{Lemma~7.4} bounds
$\EE_z \exp\bigl( C_0 \int_0^{t_2} \|u_r\|_{H^{\zeta}} \ud r \bigr)$ by a
constant multiple of $V^{1/2}_{\rho}$ for some $\rho>0$.
 Collecting all the estimates with $ s =\eta \leq t_0 = t$, we obtain
    \begin{equation*}
      \EE_z\bigl[ \|\pi_t\|_{H^1}^{q} \bigr]
      \lesssim
      V^{1/2}_{\rho}.
    \end{equation*}
  \end{proof}

\subsection{Absolute continuity of path measures}
\label{sec:absolute_continuity}

The second part of the proof establishes that the law of the perturbed noise
path is absolutely continuous with respect to the original Wiener measure. Note
that our construction works in two steps. On a first time interval $[0,
T_\star]$ we run the two noises independently. Then (if at that time the
processes meet), we activate a control. Absolute continuity on $[0, T_\star]$ is
obvious (the law is the same), so we only check the absolute continuity when we
perturb with a control. Therefore, throughout this section we are in a setting
in which $\|z_0 - z'_0\| \le \delta_\star$ and $z_0, z'_0 \in \mC_R$ and $W' = W
+ \int^\cdot h_s \ud s$ with $h$ as in~\eqref{eq:control_h}. The main result of
this section is the following.

\begin{proposition}\label{prop:absolute_continuity}
  The measure induced by the perturbed noise path, $\mathcal{L}(W')$, is
  mutually absolutely continuous with respect to the Wiener measure
  $\mathcal{L}(W)$.
\end{proposition}

We prove Proposition~\ref{prop:absolute_continuity} one interval $[2n, 2n+2)$ at a time and
then pass to the infinite horizon. The map
$W \mapsto W' = W + \int h$ is a non-adapted Cameron--Martin shift, and
absolute continuity will be determined via the Girsanov--Ramer--Kusuoka theorem. We
start this section by recalling this result and the setting of its statement.

Let us fix some time horizon $T>0$ and recall that the reference noise $W$ is space-time white, that is, the cylindrical Wiener
process $W = \sum_{k \in \ZZ^2_*} e_k W_k$ with $W_k$ as in~\eqref{eq:forcing},
which we consider with values in
$ H^{-\alpha}(\TT^2)$ for any $\alpha > d/2$. The Cameron--Martin space associated to $W$ is 
\[
  H_{\mC\mM} =  L^2\bigl([0,T]; \mH\bigr)
    = L^2\bigl([0,T]\times\TT^2\bigr).
\]
Next, for a smooth cylindrical functional
$F = f\bigl(W(g_1),\dots,W(g_m)\bigr)$, with $f \in C^\infty_b(\RR^m)$ and
$g_i \in H_{\mC\mM}$ and $W(g) = \int_0^T \langle g(s), \ud W_s\rangle$, the Malliavin derivative is the $H_{\mC\mM}$-valued random
variable
\[
  D_W F = \sum_{i=1}^m \partial_i f\bigl(W(g_1),\dots,W(g_m)\bigr) g_i .
\]
Equivalently $\langle D_W F, g\rangle_{H_{\mC\mM}}
= \tfrac{\ud}{\ud\epsilon}\big|_{\epsilon = 0} F\bigl(W + \epsilon\int g\bigr)$
is the derivative of $F$ along the Cameron--Martin shift by $\int g$.
In our setting, we will consider the control $h$ from~\eqref{eq:control_h} as a
Malliavin differentiable functional $h \colon H^{-\alpha} \to H_{\mC\mM}$, so
that the Malliavin derivative $D_W h$ is a random operator on $H_{\mC\mM}$.

In particular, we write the space $\mathbb{D}^{1,2}(H_{\mC \mM})$ of once Malliavin differentiable functionals as the completion of smooth
cylinder functionals under the norm
\[
  \|h\|_{\mathbb{D}^{1,2}(H_{\mC\mM})}^2 = \EE\|h\|_{H_{\mC\mM}}^2
+ \EE\|D_W h\|_{HS}^2
\]
so $h \in L^2(\Omega; H_{\mC \mM})$ with
Malliavin derivative being a Hilbert--Schmidt operator. 
Here, for a bounded operator $A$ on a separable Hilbert
space $\mathcal{K}$ and any orthonormal basis $\{f_j\}_j$ of $\mathcal{K}$, the
Hilbert--Schmidt norm is
\[
  \|A\|_{HS}^2 = \sum_j \|A f_j\|_{\mathcal{K}}^2 = \mathrm{Tr}(A^*A).
\]
For an
$H_{\mC\mM}$-valued functional $h$, the derivative $D_W h$ is an
element of the tensor product $H_{\mC\mM} \otimes H_{\mC\mM}$, that is, an operator
on $H_{\mC\mM}$. The
Hilbert--Schmidt norm of that operator reads
\[
  \|D_W h\|_{HS}^2 = \sum_j \|D_{f_j} h\|_{H_{\mC\mM}}^2 ,
  \qquad D_{f_j} h = \langle D_W h, f_j\rangle .
\]
Next, we recall that the Skorohod integral $\delta$ is defined as the adjoint of $D_W$:
\[
  \EE\bigl[F \delta(u)\bigr]
    = \EE\bigl[\langle D_W F, u\rangle_{H_{\mC\mM}}\bigr]
    \qquad \forall F \in \mathbb{D}^{1,2},
\]
with domain $\mathrm{Dom}(\delta) \subseteq L^2(\Omega; H_{\mC\mM})$ the set of
$u$ for which $F \mapsto \EE[\langle D_W F, u\rangle_{H_{\mC\mM}}]$ is bounded on
$L^2(\Omega)$. On adapted integrands $\delta(u)$ coincides with the It\^o integral
$\int_0^T \langle u_s, \ud W_s\rangle$, and on anticipating integrands it is its
extension. In particular $\EE[\delta(u)] = 0$.

Finally we recall the Carleman--Fredholm determinant, defined for an arbitrary
Hilbert--Schmidt operator $A$ on a separable Hilbert space $\mK$ (note that
Hilbert--Schmidt implies compact, so we may list its nonzero eigenvalues
$\{\lambda_k\}$ with algebraic multiplicity). We define the determinant
\[
  \det\nolimits_2(I + A) = \prod_k (1 + \lambda_k) e^{-\lambda_k} .
\]
The regularizing factor $e^{-\lambda_k}$ makes
$\log\bigl[(1 + \lambda_k) e^{-\lambda_k}\bigr] = O(|\lambda_k|^2)$, so the product
converges absolutely as soon as $\sum_k |\lambda_k|^2 < \infty$. This holds for
every Hilbert--Schmidt $A$, since by Weyl's inequality
\[
  \sum_k |\lambda_k|^2 \le \sum_k s_k^2 = \|A\|_{HS}^2 < \infty ,
\]
where $\{s_k\}$ are the singular values of $A$. By contrast the ordinary determinant $\det(I + A) = \prod_k (1 + \lambda_k)$
requires $A$ trace-class, and when $A$ is trace-class the two are related by
$\det\nolimits_2(I + A) = \det(I + A) e^{-\mathrm{Tr} A}$.
In this setting, we are ready to state the Girsanov--Ramer--Kusuoka theorem.

\begin{theorem}\label{thm:grkn}
Fix $T>0$ and let $W$ be space-time white noise as in the setting above. Let
$h \in \mathbb{D}^{1,2}(H_{\mC\mM})$ be a shift satisfying
almost surely $\|D_W h\|_{\mathrm{op}} < 1$ and $\|D_W h\|_{\mathrm{HS}} < \infty$. Then
$W \mapsto W' = W + h(W)$ is almost surely bijective, the laws $\mathcal{L}(W')$
and $\mathcal{L}(W)$ are mutually absolutely continuous, and the Radon--Nikodym
derivative $L = \ud \mathcal{L}(W')/ \ud \mathcal{L}(W)$ satisfies
\[
  \log L = \delta(h) - \frac{1}{2}\|h\|_{H_{\mC\mM}}^2 - \log|\det\nolimits_2(I + D_W h)| .
\]
\end{theorem}

The result is the anticipating Girsanov theorem due to Ramer~\cite{Ramer1974} and
Kusuoka~\cite{Kusuoka1982}. In the form stated here it can be found in~\cite{UstunelZakai00}*{Theorem~3.2.1}.

Now we apply the result to our choice of control. To do so, consider $h$ as
in~\eqref{eq:control_h} and fix $n \in \NN$. Then let us consider the control
restricted to $[2n, 2n+1]$, defined as:
\begin{equation*}
  h_n (s) = h(s) ,\qquad \forall s \in [2n, 2n+1].
\end{equation*}
Then $h_n$ is a control in $H_{\mC\mM} = L^2([2n, 2n+1] \times \TT^2)$ (with
some abuse of notation since the initial time is $2n$) as in the previous discussion, with time
horizon $T=1$.
To apply Theorem~\ref{thm:grkn} we must prove that $h_n \in
\mathbb{D}^{1,2}(H_{\mC\mM})$ together with some precise estimates on its size.

  \begin{lemma}\label{lem:cm_regularity}
    For any $n \geq 0$ and $K, \beta >0$, the control $h_n$ satisfies for some deterministic constant $C_K>0$
    \begin{equation}\label{eq:cm_norm_bound} 
      \begin{aligned}
              \|h_n\|_{H_{\mC\mM}}^2 & \le C_K  \beta^{-1} \|\vartheta_{2n}\|^2.
      \end{aligned}
    \end{equation}
  \end{lemma}
  \begin{proof}
    Let us assume $z_{2n} \in \mC_R$ and $2n < \tau$, otherwise the 
    statement is trivially true.
    For clarity, let us write $M = M_{2n,2n+1}$ and $J = J_{2n, 2n+1}$ for the
    Malliavin matrix and Jacobian (respectively) on the
    interval $[2n, 2n+1]$.
    Recall that $$h_n(s) = \psi_{K}(\mE_n) \mK^*_{s,2n+1}(\beta I +
    M)^{-1}J \vartheta^\sharp_{2n}$$ for $s \in [2n, 2n+1)$. Setting $\eta = \psi_{K}(\mE_n)
    (\beta I + M)^{-1}J \vartheta^\sharp_{2n}$, the Cameron--Martin norm of $h_n$
    is given by
    \begin{align}
      \|h_n\|_{H_{\mC\mM}}^2 &= \int_{2n}^{2n+1}\|\mK^*_{s,2n+1}\eta\|^2\ud s = \langle M \eta, \eta\rangle , \label{eq:cm_norm_computation}
    \end{align}
    since $M = \int \mK \mK^*\ud s$ by definition. Now:
    \begin{equation}
      \langle M\eta, \eta\rangle = \psi_{K}(\mE_n)^2 \langle M(\beta I + M)^{-2}J\vartheta^\sharp, J\vartheta^\sharp\rangle \le \psi_{K}(\mE_n)^2 \|M(\beta I+M)^{-2}\|_{\mathrm{op}} \|J\vartheta^\sharp\|^2 .
    \end{equation}
    By the spectral theorem, $\|M(\beta I+M)^{-2}\|_{\mathrm{op}} \le 
    \frac{1}{4\beta}$ since $\sup_{\lambda \ge 0} \lambda/(\beta + \lambda)^2 =
    \frac{1}{4\beta}$. Therefore $\|h_n\|_{H_{\mC\mM}}^2 \le
    \frac{\psi_{K}(\mE_n)^2}{4\beta}\|J\vartheta^\sharp_{2n}\|^2$.
    Now, if $\mE_n \leq 2(K+1)^2$ (which is the case when $\psi_{K}(\mE_n)$ is nonzero), then by Lemma~\ref{lem:fibre_dissipation}(iii) we have
    $\|J_{2n,2n+1}\vartheta^\sharp_{2n}\| \le C_K \|\vartheta^\sharp_{2n}\|$ for some
    constant $C_K>0$.
    Finally $\|\vartheta^\sharp_{2n}\| \le \|\vartheta_{2n}\|$, because
    $P_{\pi_{2n}^\perp}$ is an orthogonal projection, completing the proof.
  \end{proof}

  While the previous result guarantees that $h_n$ is square integrable, the next
  one guarantees the Malliavin differentiability.

   Before stating the lemma, we record the structure of $D_W h_n$ on the whole
  interval $[2n,2n+2)$, whose Cameron--Martin space splits into its two unit blocks,
  \begin{equation*}
    H_{\mC\mM} = L^2\big([2n,2n+2)\times\TT^2\big)
      = L^2\big([2n,2n+1)\big) \oplus L^2\big([2n+1,2n+2)\big).
  \end{equation*}
  The control takes values only on the first block: since $h_n(s)=0$ on the free
  phase $[2n+1,2n+2)$. However the Malliavin derivative does not vanish on the
  second half of the interval. Recall that $h$ contains a cut-off
  $\psi_K(\mE_n)$, where $\mE_n$ is the energy
  \begin{equation*}
    \mE_n = \int_{2n}^{2n+2} \|w_u\|_{H^1}^2 + \|\pi_u\|_{H^1}^2 \,\ud u .
  \end{equation*}
  Writing
  $\tilde h_n(s) = \mK^*_{s,2n+1}(\beta I + M)^{-1} J \vartheta^\sharp_{2n}$ for the
  control without the cutoff, the product rule gives
  \begin{equation*}
    D_W h_n = \psi_K(\mE_n)\, D_W \tilde h_n
      + \psi_K'(\mE_n)\,\big( D_W \mE_n \otimes \tilde h_n \big),
  \end{equation*}
  where $D_W \tilde h_n$ is supported on $[2n,2n+1)\times[2n,2n+1)$, whereas $D_W \mE_n$ has components on the whole cycle.
  Therefore $D_W h_n$ is block upper triangular
  \begin{equation*}
    D_W h_n = \begin{pmatrix} \ast & \ast \\ 0 & 0 \end{pmatrix},
  \end{equation*}
  and so also $I + D_W h_n$ is block upper triangular, so its Carleman--Fredholm determinant $\det\nolimits_2(I+D_W h_n)$
  reduces to that of the restriction to the active block $[2n,2n+1)$ alone. In
  the following result, we therefore note that $h_n$ is considered on the full
  interval $[2n, 2n+2]$.

  \begin{lemma}\label{lem:malliavin_control_bounds}
    For any $n \ge 0$ and $K, \beta >1$, the derivative $D_W h_n$ satisfies for some
    $C(K, \beta)>0$:
    \begin{equation}\label{eq:malliavin_derivative_bounds}
      \|D_W h_n\|_{HS} \le C(K,\beta) \|\vartheta_{2n}\|.
    \end{equation}
     In particular, $h_n \in \mathbb{D}^{1,2}(H_{\mC\mM})$.
  \end{lemma}
  \begin{proof}
    We will prove the result only for the LNS case, since the PSA case is
    strictly simpler.
    As above, we may assume $h_n \neq 0$, since otherwise $D_W h_n = 0$ and the bound is trivial.
    Then, 
    taking the Malliavin derivative yields:
    \begin{equation}\label{eq:control_product_rule}
      D_W h_n = \psi(\mathcal{E}_n) D_W \tilde{h}_n + \psi'(\mathcal{E}_n) (D_W \mathcal{E}_n \otimes \tilde{h}_n) ,
    \end{equation}
    where
    \begin{equation*}
      \tilde{h}_n (s)= \mK^*_{s,2n+1}(\beta I +
    M)^{-1}J \vartheta_{2n}.
    \end{equation*}
    Because $\psi$ and $\psi'$ vanish for $\mathcal{E}_n > (K+1)^2$, it suffices to
    bound each term on $\{\mathcal{E}_n \le (K+1)^2\}$.

By the chain rule, estimating $D_W \tilde{h}$ requires estimating the Malliavin
derivatives $D_W M, D_W J$ and $D_W \mK$.

We discuss the case of $D_W M$, since it implies all the other ones (see the end
of the proof). Note that $M = A A^*$
with $A = A_{2n,2n+1}$, and for $i \in \ZZ^2_*$ and $s \in [2n, 2n+1]$ let
$\mathcal{D}^i_s F \eqdef \langle (D_W F)_s, e_i \rangle_{\mH}$ be the component of the
Malliavin derivative $D_W F \in H_{\mC\mM} = L^2([2n,2n+2];\mH)$ in the mode
$e_i$ at time $s$. The operators $M$, $J$ and $A$ read the noise
only on the interval $[2n, 2n+1]$, so their Malliavin derivatives are supported
on $[2n,2n+1]$. Instead the noise in $[2n+1, 2n+2]$ enters $D_W h_n$ only through $D_W\mE_n$, giving the block-triangular
structure noted before the lemma. Since
$\mathcal{D}^i_s A^* = (\mathcal{D}^i_s A)^*$ we have
\begin{equation}\label{eq:leibniz_M}
  \mathcal{D}^i_s M = (\mathcal{D}^i_s A) A^* + A (\mathcal{D}^i_s A)^* ,
\end{equation}
so it suffices to bound $\|A\|_{\mathrm{op}}$ and $\|\mathcal{D}^i_s A\|_{\mathrm{op}}$
on $\{\mathcal{E}_n \le (K+1)^2\}$.
By Lemma~\ref{lem:jacobian_moments}(i) we have that on $\{\mathcal{E}_n \le (K+1)^2\}$,
\begin{equation}\label{eq:Jop_EK}
\begin{aligned}
    \|J_{u,v}\|_{\mathrm{op}} & \le C_K , \qquad & \forall 2n \le u \le v \le 2n+1 , \\
    \int_u^{2n+1} \|J_{u,v}\eta\|_{H^1}^2 \ud v & \le C_K \|\eta\|^2 , \qquad & \xi \in T_{z_u}\mX,\ u \in [2n,2n+1] .
\end{aligned}
\end{equation}
Now, the Malliavin derivative of the Jacobian solves the
second--variation equation
\begin{equation}\label{eq:secondvar_J}
  \mathcal{D}^i_s J_{r,t}\eta
    = \int_{s \vee r}^t J_{u,t} (D_z L_{z_u})\bigl[J_{s,u}(Q^{1/2}e_i,0)\bigr] J_{r,u}\eta \ud u ,
\end{equation}
where we have used the linearized operator $L_z$
of~\eqref{eq:grand_jacobian_system} together with the identity
\begin{equation*}
  \mD^i_s z_u = J_{s, u}(Q^{1/2}e_i,0) \mathbf{1}_{\{s \leq u\}}.
\end{equation*}
The operator $L_z$ is affine in $w$ and quadratic in $\pi$: the base
block $L^{ww}_w = \Delta - B(w,\cdot) - B(\cdot,w)$ is affine in $w$, while the coupling
and fiber blocks~\eqref{eq:coupling_operator}--\eqref{eq:fiber_operator} are quadratic in
$\pi$ and affine in $w$. More precisely:
\begin{equation*}
  (D_z L_z)[\zeta]\phi=
  \begin{pmatrix}
    -B(\zeta_w,\phi_w) - B(\phi_w,\zeta_w) \\[2pt]
    D_z L^{\pi w}_z [\zeta] \phi_w  + D_z L^{\pi \pi}_z [\zeta] \phi_\pi
  \end{pmatrix},
  \qquad \zeta=(\zeta_w,\zeta_\pi),\quad \phi=(\phi_w,\phi_\pi),
\end{equation*}
where the last two terms are defined as follows, with $P^\perp v \eqdef v - \langle\pi,v\rangle\pi$ denoting the projection onto $\pi^\perp$,
$(DL^0_w)[a]b \eqdef B(a,b) + B(b,a)$ is the derivative
of~\eqref{eq:two_operators_intro_NS},
$G_w$ is as in~\eqref{eq:proj-spde}, and the derivatives of the coupling
block~\eqref{eq:coupling_operator} and of the fiber block~\eqref{eq:fiber_operator} are
\begin{align*}
  D_z L^{\pi w}_z [\zeta] \phi_w  &= -P^\perp (DL^0_w)[\phi_w]\zeta_\pi
    + \langle\pi,(DL^0_w)[\phi_w]\pi\rangle\zeta_\pi
    + \langle\zeta_\pi,(DL^0_w)[\phi_w]\pi\rangle\pi , \\
  D_z L^{\pi \pi}_z [\zeta] \phi_\pi &= -P^\perp (DL^0_w)[\zeta_w]\phi_\pi
    + \langle\pi,(DL^0_w)[\zeta_w]\pi\rangle\phi_\pi
    + \langle\phi_\pi,(DL^0_w)[\zeta_w]\pi\rangle\pi \\
    &\quad - \bigl(\langle\pi,G_w\phi_\pi\rangle + \langle\phi_\pi,G_w\pi\rangle\bigr)\zeta_\pi
    - \bigl(\langle\zeta_\pi,G_w\phi_\pi\rangle + \langle\phi_\pi,G_w\zeta_\pi\rangle\bigr)\pi \\
    &\quad - \bigl(\langle\zeta_\pi,G_w\pi\rangle + \langle\pi,G_w\zeta_\pi\rangle\bigr)\phi_\pi .
\end{align*}
Here to obtain the second identity, we differentiate the four terms of
$L^{\pi\pi}_z\phi_\pi = G_w\phi_\pi - \langle\pi,G_w\phi_\pi\rangle\pi
- \langle\phi_\pi,G_w\pi\rangle\pi - \langle\pi,G_w\pi\rangle\phi_\pi$
of~\eqref{eq:fiber_operator} in $z$, with
$D_w G_w[\zeta_w] = -(DL^0_w)[\zeta_w]$. Therefore, term by term:
\begin{align*}
  D_z\bigl(G_w\phi_\pi\bigr) &= -(DL^0_w)[\zeta_w]\phi_\pi , \\
  D_z\bigl(-\langle\pi,G_w\phi_\pi\rangle\pi\bigr)
    &= +\langle\pi,(DL^0_w)[\zeta_w]\phi_\pi\rangle\pi
       - \langle\zeta_\pi,G_w\phi_\pi\rangle\pi
       - \langle\pi,G_w\phi_\pi\rangle\zeta_\pi , \\
  D_z\bigl(-\langle\phi_\pi,G_w\pi\rangle\pi\bigr)
    &= +\langle\phi_\pi,(DL^0_w)[\zeta_w]\pi\rangle\pi
       - \langle\phi_\pi,G_w\zeta_\pi\rangle\pi
       - \langle\phi_\pi,G_w\pi\rangle\zeta_\pi , \\
  D_z\bigl(-\langle\pi,G_w\pi\rangle\phi_\pi\bigr)
    &= +\langle\pi,(DL^0_w)[\zeta_w]\pi\rangle\phi_\pi
       - \bigl(\langle\zeta_\pi,G_w\pi\rangle + \langle\pi,G_w\zeta_\pi\rangle\bigr)\phi_\pi .
\end{align*}
Now we use that $\| D L^0_w [a]b \| \lesssim \|a \|_{H^1} \| b \|_{H^1} $ and
$\| \pi \|=1$, so
that 
\begin{equation}\label{eq:DzL_bound}
  \bigl\| (D_z L_z)[\zeta]\phi \bigr\|
    \lesssim \|\zeta\|_{H^1}\|\phi\|_{H^1}
      + \bigl(\|w\|_{H^1} + \|\pi\|_{H^1}\bigr)\bigl(\|\zeta\|_{H^1}\|\phi\| + \|\zeta\|\|\phi\|_{H^1}\bigr).
\end{equation}

Here to bound all the terms involving  $G_w$ we use the following estimate:
\begin{equation}\label{eq:Gw_bound}
  |\langle a, G_w b\rangle|
    \lesssim \|a\|_{H^1}\|b\|_{H^1}
      + \|w\|_{H^1}\bigl(\|a\|_{H^1}\|b\| \ \wedge\ \|a\|\|b\|_{H^1}\bigr),
\end{equation}
where $\wedge$ is the minimum. Indeed
$\langle a,\Delta b\rangle=-\langle\nabla a,\nabla b\rangle$ gives the first
term. For the transport part, with $u_f = K*f$, one 
has $$|\langle a, u_w \cdot \nabla b \rangle| \lesssim \|a \|\| u_w\|_\infty \|
b\|_{H^1} \lesssim \|a \| \|w\|_{H^1} \|
b\|_{H^1} $$
and the roles of $a$ and $b$ can be inverted by integration by parts. Instead,
for the stretching term we have
$|\langle a, B(b,w)\rangle| = |\langle a,u_b\cdot\nabla w\rangle| \le
\|a\|_{L^4}\|u_b\|_{L^4}\| w\|_{H^1} \wedge \|a\|_{L^2}\|u_b\|_{L^{\infty}}\| w\|_{H^1}$.
Using either $\|u_f\|_{L^4}\lesssim\|f\|$ or $\|u_f\|_{L^\infty}\lesssim\|f\|_{H^1}$ and
$\|f\|_{L^4}\lesssim\|f\|_{H^1}$ delivers the appropriate estimate.

Now we insert \eqref{eq:DzL_bound} into \eqref{eq:secondvar_J} with
$\zeta_u = J_{s,u}(Q^{1/2}e_i,0)$ and $\phi_u = J_{r,u}\eta$, and use the bound
$\|J_{u,t}\|_{\mathrm{op}} \le C_K$ by \eqref{eq:Jop_EK} to obtain by repeated
use of Lemma~\ref{lem:jacobian_moments}
\begin{equation*}
  \begin{aligned}
      \|\mathcal{D}^i_s J_{r,t}\eta\| & \lesssim_K \int_{s \vee r}^t \|\zeta_u\|_{H^1}\|\phi_u\|_{H^1}
      + \bigl(\|w_u\|_{H^1} + \|\pi_u\|_{H^1}\bigr)\bigl(\|\zeta_u\|_{H^1}\|\phi_u\| + \|\zeta_u\|\|\phi_u\|_{H^1}\bigr) \ud u \\
      & \lesssim_K \int_{s \vee r}^t \|\zeta_u\|_{H^1}\|\phi_u\|_{H^1}
      + \bigl(\|w_u\|_{H^1} + \|\pi_u\|_{H^1}\bigr)\bigl(\|\zeta_u\|_{H^1} \| \eta \| + \sigma_i \|\phi_u\|_{H^1}\bigr) \ud u \\
      & \lesssim_K \Bigl(\int_{s}^{2n+1} \|\zeta_u\|_{H^1}^2\ud u\Bigr)^{1/2}
        \Bigl(\int_{r}^{2n+1} \|\phi_u\|_{H^1}^2\ud u\Bigr)^{1/2} \\
        & \qquad + \Bigl(\int_{s}^{2n+1} \|\zeta_u\|_{H^1}^2\ud u\Bigr)^{1/2} \|\eta \| + \Bigl(\int_{r}^{2n+1} \|\phi_u\|_{H^1}^2\ud u\Bigr)^{1/2} \sigma_i \\
      & \lesssim_K \sigma_i  \| \eta  \|.
  \end{aligned}
\end{equation*}
It follows that on $\{\mE_n \leq (K+1)^2\}$ and as an operator on $T_{z_r} \mX$, we have
\[
  \|\mathcal{D}^i_s J_{r,t}\|_{\mathrm{op}} \le C_K \sigma_i,
\]
uniformly in $s,r,t \in [2n,2n+1]$.

Since
$Ah = \int_{2n}^{2n+1} J_{r,2n+1}(Q^{1/2}h(r),0)\ud r$ depends on the noise only
through $J$, differentiating under the integral and using the previous bound,
\[
  \|\mathcal{D}^i_s A\|_{\mathrm{op}}
    \le \int_{2n}^{2n+1} \|\mathcal{D}^i_s J_{r,2n+1}\|_{\mathrm{op}}\|Q^{1/2}\|_{\mathrm{op}}\ud r
    \le C_K \sigma_i ,
\]
while $\|A\|_{\mathrm{op}} \le C_K \|Q^{1/2}\|_{\mathrm{op}}$ by \eqref{eq:Jop_EK}.
Inserting both into the Leibniz identity \eqref{eq:leibniz_M},
\begin{equation}\label{eq:DM_mode}
  \|\mathcal{D}^i_s M\|_{\mathrm{op}}
    \le 2\|\mathcal{D}^i_s A\|_{\mathrm{op}}\|A\|_{\mathrm{op}}
    \le C_K \sigma_i , \qquad s \in [2n,2n+1],
\end{equation}
Summing over the noise modes and integrating in $s$, the
Malliavin derivative of $M$ is controlled by the trace of the noise covariance:
\[
  \|D_W M\|_{H_{\mC\mM}}^2
    = \sum_i \int_{2n}^{2n+1} \|\mathcal{D}^i_s M\|_{\mathrm{op}}^2 \ud s
    \le C_K \sum_i \sigma_i^2 < \infty ,
\]
which is finite because $\sigma_i \sim |i|^{-\alpha_\star}$ with $\alpha_\star > 10$
by Assumption~\ref{ass:noise}. The same arguments apply to $D_W A$ and $D_W
\mK^*$ (which is the kernel of $A$).

We are left with estimating the Malliavin derivative $D_W\mathcal{E}_n$, which
appears in the second term of~\eqref{eq:control_product_rule}. Recall that
\[
  \mathcal{E}_n = \int_{2n}^{2n+2} \|w_u\|_{H^1}^2 + \|\pi_u\|_{H^1}^2 \ud u ,
\]
and that the Malliavin derivative of the state in the mode $e_i$ at time $s$ is
$\mathcal{D}^i_s z_u = J_{s,u}(Q^{1/2}e_i,0) \mathbf{1}_{\{s \le u\}}$, with base and fiber
components $\mathcal{D}^i_s w_u$ and $\mathcal{D}^i_s \pi_u$.
Differentiating the two
quadratic integrands we obtain
\[
  \mathcal{D}^i_s \mathcal{E}_n
    = 2\int_{s}^{2n+2} \langle w_u, \mathcal{D}^i_s w_u\rangle_{H^1}
      + \langle \pi_u, \mathcal{D}^i_s \pi_u\rangle_{H^1} \ud u .
\]
Hence by Cauchy--Schwarz and since we are working on the event $\mE_n \leq (K+1)^2$:
\begin{equation*}
  \begin{aligned}
  |\mathcal{D}^i_s \mathcal{E}_n|
    & \le 2\int_{s}^{2n+2} \bigl(\|w_u\|_{H^1} + \|\pi_u\|_{H^1}\bigr)
      \|\mathcal{D}^i_s z_u\|_{H^1}\ud u \\
    & \lesssim_K \left( \int_{s}^{2n+2} 
      \|\mathcal{D}^i_s z_u\|_{H^1}^2\ud u \right)^{\frac{1}{2}} .
  \end{aligned}
\end{equation*}
Now we use the bound for the dissipative term of
Lemma~\ref{lem:jacobian_moments}(i) to obtain
\[
  \int_{s}^{2n+2} \|\mathcal{D}^i_s z_u\|_{H^1}^2 \ud u
    = \int_{s}^{2n+2} \|J_{s,u}(Q^{1/2}e_i,0)\|_{H^1}^2 \ud u
    \lesssim_K \sigma_i^2 .
\]
Therefore, on the event $\mE_n \leq (K+1)^2$ we have the following bound:
\[
  \|D_W \mathcal{E}_n\|_{H_{\mC\mM}}^2
    = \sum_i \int_{2n}^{2n+2} |\mathcal{D}^i_s \mathcal{E}_n|^2 \ud s
    \lesssim_K 1,
\]
by Assumption~\ref{ass:noise}. 

  Hence, overall we bound the second term of~\eqref{eq:control_product_rule} via
 the previous display and Lemma~\ref{lem:cm_regularity} for $\|\tilde{h}_n\|_{H_{\mC\mM}} \le
  C\beta^{-1/2}\|\vartheta_{2n}\|$:
\[
  \bigl\| \psi'(\mathcal{E}_n)(D_W \mathcal{E}_n \otimes \tilde{h}_n) \bigr\|_{HS}
    = |\psi'(\mathcal{E}_n)|\|D_W \mathcal{E}_n\|_{H_{\mC\mM}}\|\tilde{h}_n\|_{H_{\mC\mM}}
    \le C(K,\beta)\|\vartheta_{2n}\| .
\]
The same works for the first term $\psi(\mathcal{E}_n)D_W \tilde{h}_n$ of~\eqref{eq:control_product_rule}:
\[
  \bigl\| \psi(\mathcal{E}_n)D_W \tilde{h}_n \bigr\|_{HS}^2
    \le \psi(\mathcal{E}_n)^2  \sum_i \int_{2n}^{2n+1} \|\mathcal{D}^i_s \tilde{h}_n\|_{H_{\mC\mM}}^2 \ud s
    \lesssim_{K,\beta} \Bigl(\sum_i \sigma_i^2\Bigr)\|\vartheta_{2n}\|^2 ,
\]
as desired. The proof is complete.
  \end{proof}

  For the next result, recall that we introduced a smallness parameter
  $\delta$ in the definition of the stopping time $\tau$
  in~\eqref{eq:stopping_time_tau}. By that definition, $2n < \tau$ forces 
  $\|\vartheta_{2n}\| \le \delta e^{-2n} \le \delta$, or if $\tau
  \leq 2n$ then $h_n = 0$.
  \begin{lemma}\label{lem:path_space_invertibility}
    There exists a $\delta \in (0, 1)$ such that $(I + D_W h_n): H_{\mC\mM} \to H_{\mC\mM}$ is invertible for all $n \ge 0$.
  \end{lemma}
  \begin{proof}
    By Lemma~\ref{lem:malliavin_control_bounds}, $D_W h_n$ is
    Hilbert--Schmidt, so it is a bounded operator and it suffices to show $\|D_W h_n\|_{\mathrm{op}} < 1$, since
    then $(I + D_W h_n)$ is invertible by a Neumann series. By~\eqref{eq:malliavin_derivative_bounds} we have $\|D_W h_n\|_{\mathrm{op}}
    \le C(K,\beta)\|\vartheta_{2n}\|$, so it suffices to take $\delta \le \min\{ 1/2,\ (2 C(K, \beta))^{-1} \}$. 
  \end{proof}
The three lemmas above prove by Theorem~\ref{thm:grkn} that on each interval
$[2n, 2n+2)$ the laws of $W$ and $W'$ are equivalent. 
To be precise, for $n \ge 0$ write $$W_n(t) = W(t) - W(2n), \qquad \forall t \in
[2n, 2n+2]$$ for the increment of the cylindrical
Wiener process $W$ on
$[2n, 2n+2] \times \TT^2$ and let
\begin{equation*}
  \mathcal{B}_n \eqdef \sigma\big( W_j : j < n \big) ,
\end{equation*}
so that $(\mathcal{B}_n)_{n \ge 0}$ is a filtration with $\mathcal{B}_\infty = \sigma(W)$.
Then with the control $h_n$ of~\eqref{eq:control_h} we have $$W'_n (t)  = W_n(t) 
+ \int_{2n}^t h_n(s) \ud s , \qquad \forall t \in [2n, 2n+2] .$$
Conditionally on $\mathcal{B}_n$, because $h_n$ depends on the past, Theorem~\ref{thm:grkn} applied on $[2n, 2n+2)$ (its
hypotheses hold by Lemmas~\ref{lem:cm_regularity}--\ref{lem:path_space_invertibility})
makes the conditional laws of $W'_n$ and $W_n$ equivalent, with density
\begin{equation}\label{e:ln-def}
  L_n \eqdef \frac{\ud \mathcal{L}(W'_n \mid \mathcal{B}_n)}{\ud \mathcal{L}(W_n)} ,
  \qquad
  \log L_n = \delta(h_n) - \frac{1}{2}\|h_n\|_{H_{\mC\mM}}^2
    - \log\big|\det\nolimits_2(I + D_W h_n)\big| ,
\end{equation}
where we used that $\mathcal{L}(W_n \mid \mathcal{B}_n) = \mathcal{L}(W_n)$, the reference
noise having independent blocks.
To pass to the infinite
horizon we use a Markovian equivalent of Kakutani's
equivalence criterion that can be found in \cite{JacodShiryaev2003}*{Theorem IV.2.36(a), p. 253}. To verify the conditions of that theorem we need the
following bound on the density $L_n$.

\begin{lemma}\label{lem:summability}
  Fix any $K, \beta > 1$ and define $h_n$ as in~\eqref{eq:control_h}. For
  $\delta = \delta(K, \beta) \in (0, 1)$ sufficiently small, as in
  Lemma~\ref{lem:path_space_invertibility}, the densities $L_n$ of~\eqref{e:ln-def} satisfy
  \begin{equation}\label{eq:conditional_hellinger_bound}
    \EE\big[ (\sqrt{L_n} - 1)^2 \,\big|\, \mathcal{B}_n \big] \ \le\ C e^{-2  n} ,
  \end{equation}
  with a deterministic $C = C(K, \beta)$. In particular 
  $\sum_{n \ge 0} \EE[(\sqrt{L_n}-1)^2 \mid \mathcal{B}_n]
  \le C \sum_{n \ge 0} e^{-2 n} < \infty$.
\end{lemma}
\begin{proof}
  For $x > 0$ we use the (somewhat classical in this context) inequality $(\sqrt{x} - 1)^2 \le x - 1 - \log x$. Conditionally on $\mathcal{B}_n$
  the variable $L_n$ is a Radon--Nikodym density, so $\EE[L_n \mid \mathcal{B}_n] = 1$ and
  \begin{equation*}
    \EE\big[(\sqrt{L_n} - 1)^2 \,\big|\, \mathcal{B}_n\big]
      \ \le\ -\EE\big[\log L_n \,\big|\, \mathcal{B}_n\big]
      \ =\ \frac{1}{2}\EE\big[\|h_n\|_{H_{\mC\mM}}^2 \,\big|\, \mathcal{B}_n\big]
        + \EE\big[\log\big|\det\nolimits_2(I + D_W h_n)\big| \,\big|\, \mathcal{B}_n\big] ,
  \end{equation*}
  by~\eqref{e:ln-def} and since the Skorokhod integral satisfies
  $\EE[\delta(h_n) \mid \mathcal{B}_n] = 0$. Next we apply the estimate
  $\log|\det_2(I+U)| \le C\|U\|_{HS}^2$ of \cite[Theorem~9.2]{simon2005trace}
  together with Lemma~\ref{lem:path_space_invertibility}, which give us the
  uniform bound
  $\|D_W h_n\|_{\mathrm{op}} \le 1/2$. By Lemmas~\ref{lem:cm_regularity}
  and~\ref{lem:malliavin_control_bounds} and the definition of $\tau$
  in~\eqref{eq:stopping_time_tau}, we deduce
  \begin{equation*}
    \frac{1}{2}\|h_n\|_{H_{\mC\mM}}^2 + C\|D_W h_n\|_{HS}^2
      \ \le\ C (2\delta)^2 e^{-2 n}\, \mathbf{1}_{\{2n < \tau\}}
      \ \le\ C e^{-2 n} ,
  \end{equation*}
  and~\eqref{eq:conditional_hellinger_bound} follows.
\end{proof}

With these results in hand, we are ready to prove the main result of this subsection.
\begin{proof}[Proof of Proposition~\ref{prop:absolute_continuity}]
  The measures $\mathcal{L}(W_n)$ and $\mathcal{L}(W'_n)$ are equivalent for
  every $n$
  by Theorem~\ref{thm:grkn}. By Lemma~\ref{lem:summability}
  the assumption of~\cite{JacodShiryaev2003}*{Theorem IV.2.36(a), p. 253} holds under either measure. Applying the theorem once with
  $(P, Q) = (\mathcal{L}(W), \mathcal{L}(W'))$ and once with the roles exchanged yields both
  inclusions, so the two path measures are mutually absolutely continuous.
\end{proof}

\subsection{Proof of the main theorem}\label{sec:main_proof}

We now combine all the results obtained so far.

\begin{proof}[Proof of the uniqueness of the stationary measures in Theorem~\ref{thm:main_result_intro}]
  We consider two $P^2$-ergodic measures $\mu_1 \ne \mu_2$ and prove that
  necessarily $\mu_1 =\mu_2$.

  Next we fix the parameters of the asymptotic coupling. Note that the coupling
  depends on $\delta_\star, R$ through the time $T_\star(\delta_\star, R)$ of
  the initial independent coupling. Then the control $h$ depends additionally on parameters
  $K, \beta >1$ and $\delta \in (0, 1)$, where $\delta$ appears in the
  definition of the stopping time $\tau$.
  The parameters $R, K, \beta > 1$ are fixed in
  Proposition~\ref{prop:asymptotic_contraction} and do not depend on $\delta$. This 
  leaves $\delta$ free to be fixed in
  Lemma~\ref{lem:path_space_invertibility} as a function of $K$ and $\beta$.
  Then we fix
  $\delta_\star$ in Proposition~\ref{prop:asymptotic_contraction} as a function of
  $\delta$. Finally, we fix $T_\star(\delta_\star, R)$ in
  Lemma~\ref{lem:meeting_time_new}.

  In this setting Proposition~\ref{prop:asymptotic_contraction} and
  Proposition~\ref{prop:absolute_continuity} verify the assumption of
  Theorem~\ref{thm:hms_coupling_criterion} so that the claim is proven.
\end{proof}

\subsection{The Furstenberg--Khasminskii formula}\label{sec:fk_consequences}

In this section we prove some simple consequences of the uniqueness of stationary
measures that was the main aim of this work. In particular, we prove a
Furstenberg--Khasminskii formula, and that the
Lyapunov exponent is attained for any deterministic initial condition.

  \begin{proof}[Proof of Corollaries~\ref{cor:FK_formula} and~\ref{cor:deterministic_datum}]
    Write $\Lambda(w,\pi) \eqdef \langle \pi, G_w \pi \rangle$. The chain
    rule gives
    \begin{equation}\label{eq:log_derivative_integrated}
      \log\|\zeta_T\| - \log\|\zeta_0\| = \int_0^T \Lambda(w_t, \pi_t) \ud t , \qquad
      \Lambda(w, \pi) = -\|\nabla \pi\|^2 - \langle \pi, (K*\pi)\cdot\nabla w \rangle ,
    \end{equation}
    the second term being absent in the PSA case~\eqref{eq:two_operators_intro_PS}, and
    H\"older with interpolation give
    \begin{equation}\label{eq:Lambda_pointwise}
      |\Lambda(w,\pi)| \ \le\ \frac32 \|\pi\|_{H^1}^2 + C\bigl( 1 + \|w\|_{H^1}^2 \bigr) .
    \end{equation}
    So $\Lambda \in L^1(\nu)$ by \cite[Corollary~2.3]{HPRY24}, with $\nu$
    the now unique stationary measure of $z = (w,\pi)$. Setting
    $\lambda_\nu \eqdef \int_\mX \Lambda \ud\nu$, we claim
    \begin{equation}\label{eq:time_average_Lambda}
      \frac1T \int_0^T \Lambda(z_t) \ud t \ \to\ \lambda_\nu
      \qquad \text{almost surely, for every deterministic } z_0 \in \mX .
    \end{equation}
    This is \cite{Kifer86}*{Chapter~III, Corollary~2.1}, after
    \cite{FurstenbergKifer83}, but proved there through Birkhoff and hence only for
    $\nu$-almost every $z_0$, on a compact fiber and for $\Lambda$ of finite semi-norm. We
    prove it as stated. By
    Remark~\ref{rem:smoothing} and the Markov property we may fix initial data
    $z_0 \in \mX_\star$, hence in a central region $\mC_{R_0}$ since
    $\beta_\star > 1$, at the cost of an almost surely finite additive constant
    in~\eqref{eq:log_derivative_integrated}. 
    We write $\mE_n$ for the cycle
    energy~\eqref{eq:def-8En}, $\widehat V (u)= V_\rho^{1/2} (u)$ with the
    Lyapunov functional as in~\eqref{eq:lyapunov_S4}, and we consider $\rho$
    sufficiently large for all our purposes.

   We will use the following. If $\{X_k\}_{k \ge 0}$ is nonnegative with $X_k$
    measurable for $\mF_{2k+2}$ and
    \begin{equation}\label{eq:cycle_average_hyp}
      \sup_{k \ge 0} \EE_{z_0}\bigl[ X_k^2 \bigr] < \infty , \qquad
      \EE\bigl[ X_k \bigm| \mF_{2k-2} \bigr]
      \le C_\star \bigl( 1 + \widehat V(u_{2k-2}) \bigr) \quad (k \ge 1) ,
    \end{equation}
    then almost surely
    \begin{equation}\label{eq:cycle_average_principle}
      \limsup_{n \to \infty} \frac1n \sum_{k<n} X_k < \infty , \qquad \frac{X_n}{n} \to 0 .
    \end{equation}
    The second assertion follows from Chebyshev and Borel--Cantelli. For the
    first, split $X_k = \EE[X_k \mid \mF_{2k-2}] + d_k$. Lemma~\ref{lem:lln_lag} at $m = 1$
    gives $n^{-1}\sum_{k<n} d_k \to 0$, and~\eqref{eq:cycle_average_hyp}
    with~\eqref{eq:occupation_bound} bound the rest.

    Now, we will apply~\eqref{eq:cycle_average_hyp} to the sequences
    \begin{equation*}
      X_k = 1 + \mE_k , \qquad
      X_k = \|\pi_{2k}\|_{H^1}^2 + \widehat V(u_{2k}) .
    \end{equation*}
    The fact that both satisfy~\eqref{eq:cycle_average_hyp} follows, in the
    first case, by~\eqref{eq:bmom}, and in the second case by
    Lemma~\ref{lem:uniform_fibre_moment}.
    Hence, we find that almost surely,
    \begin{equation}\label{eq:pathwise_averages}
      \Sigma \eqdef \sup_{n \ge 1} \frac1n \sum_{k<n}
      \bigl( 1 + \mE_k + \|\pi_{2k}\|_{H^1}^2 + \widehat V(u_{2k}) \bigr) < \infty , \qquad
      \frac{\mE_n}{n} \to 0 .
    \end{equation}

        Now let $\nu_n \eqdef \frac1n \sum_{j<n} \delta_{z_{2j}}$ be the occupation measures of $\{z_{2j}\}$. Then
    by~\eqref{eq:pathwise_averages} the sequence $\nu_n$ is almost surely
    tight. Fix a countable convergence-determining family. For $f$ in it
    the chain is Feller \cite[proof of Corollary~2.3]{HPRY24}, so $P^2 f$ is bounded
    continuous on $\mX$ and
    \begin{equation*}
      \frac1n \sum_{j<n}\bigl[ f(z_{2j+2}) - (P^2f)(z_{2j}) \bigr] \to 0 ,
      \qquad
      \frac1n \sum_{j<n}\bigl[ f(z_{2j+2}) - f(z_{2j}) \bigr]
      = \frac{f(z_{2n}) - f(z_0)}{n} \to 0 ,
    \end{equation*}
    the first by Lemma~\ref{lem:lln_lag} at $m = 0$, so $\nu_n(P^2 f - f) \to 0$. The
    family being countable this holds off one null set, where every limit point of $\nu_n$
    is therefore $P^2$-invariant. Such a measure is $\nu$ by
    Theorem~\ref{thm:main_result_intro}, so $\nu_n \Rightarrow \nu$ almost surely.
    Since $\Lambda$ is neither continuous (on $\mX$) nor bounded, we smooth it. We set
    \begin{equation*}
      Y_j \eqdef \int_{2j}^{2j+2}\Lambda(z_t) \ud t , \qquad
      F \eqdef \int_2^4 P^r\Lambda \ud r , \qquad
      \EE\bigl[ Y_{j+1} \bigm| \mF_{2j} \bigr] = F(z_{2j}) ,
    \end{equation*}
    the last identity by the Markov property.

    In the following discussion, we will prove two properties of $F$. The first is the bound
    \begin{equation}\label{eq:F_square_bound}
      F^2 \ \le\ C \widehat V ,
    \end{equation}
    which replaces the boundedness of the functional.
    The second is that $F$ is continuous on sublevel sets $\{\widehat V \le K\}$ for every
    $K > 0$ (both bounds together imply that on sublevel sets of the Lyapunov
    function $F$ is continuous and bounded). 

    For the first bound, by~\eqref{eq:Lambda_pointwise} we have
    \begin{equation*}
      \Lambda(z_r)^2 \ \le\ C \bigl( \|\pi_r\|_{H^1}^4
      +\|w_r\|_{H^1}^4 \bigr) .
    \end{equation*}
    The first term is bounded by~\eqref{eq:uniform_fibre_moment} at $q = 4$. For
    the second we find
    \begin{equation*}
      \|w\|_{H^1}^4 \ \le\  C V_\rho(u) ,
    \end{equation*}
    for $\rho$ sufficiently large.
 Collecting the two terms we obtain
    \begin{equation}\label{eq:lambda_second_moment}
      \EE_z\bigl[ \Lambda(z_r)^2 \bigr] \ \le\ C \widehat V(u) ,
      \qquad r \in [2, 4] ,
      \quad z = (w, \pi) \in \mX , \ w \in H^{s_\star} ,
    \end{equation}
    uniformly over $\pi$. Jensen then implies~\eqref{eq:F_square_bound}. 
    Since
    $\int \widehat V \ud\nu_n \le \Sigma$ by~\eqref{eq:pathwise_averages}, we
    also have that almost surely 
    \begin{equation}\label{eq:F_uniform_bound}
      \sup_{n \ge 1} \int F^2 \ud\nu_n \ \le\ C \Sigma.
    \end{equation}

    We turn to the continuity. The functional $\widehat V$ is lower
    semicontinuous, so $\{\widehat V \le K\}$ is closed. On it
    $\|u\|_{H^{\beta_\star}}$ is bounded, and so is $\|w\|_{H^{s_\star}}$
    because $\beta_\star - 1 = s_\star$. The Gr\"onwall constants of
    Lemma~\ref{lem:base_difference} and Lemma~\ref{lem:fibre_difference}(ii) are
    therefore uniform there, so $z \mapsto (w, \pi)$ is continuous from $\mX$ into
    $L^2([2,4]; H^1)$. In that norm $\Lambda$ is quadratic, so
    $z \mapsto \int_2^4 \Lambda(z_r) \ud r$ is continuous pathwise.
    By~\eqref{eq:lambda_second_moment} these integrals are bounded in $L^2$
    uniformly on $\{\widehat V \le K\}$, hence uniformly integrable, and Vitali's
    theorem implies that $F$ is
    continuous on $\{\widehat V \le K\}$.

    Next, the differences $d_j \eqdef Y_{j+1} - F(z_{2j})$ satisfy
    $\EE[d_j \mid \mF_{2j}] = 0$, by the identity defining $F$, and
    $\EE_{z_0}[d_j^2] \le 4c_\Lambda$, so Lemma~\ref{lem:lln_lag} with $m = 0$ gives
    $n^{-1}\sum_{j<n} d_j \to 0$ almost surely. The complements of $\{\widehat
    V \le K\}$ carry at most $\Sigma/K$ mass with respect to $\nu_n$ and $\nu$.
    Then, since $F \in L^1(\nu)$, we have
    \begin{equation*}
      \int |F| \mathbf{1}_{\{|F| > L\}} \ud\nu_n \ \le\ \frac{C}{L} ,
    \end{equation*}
    by $F^2 \le C\widehat V$ and $\int \widehat V^{1/2} \ud\nu \lesssim 1$.
    Truncating first in $K$, then in $L$, and sending $n \to \infty$, we obtain
    \begin{equation*}
      \int F \ud\nu_n \ \to\ \int F \ud\nu = 2\lambda_\nu ,
      \qquad \text{hence} \qquad
      \frac1n \sum_{j<n} F(z_{2j}) \ \to\ 2\lambda_\nu
      \quad \text{almost surely,}
    \end{equation*}
    where the middle identity holds by invariance of $\nu$ and Fubini. From here
    we deduce that
    $T^{-1}\int_0^T \Lambda(z_t) \ud t \to \lambda_\nu$
    almost surely, which with~\eqref{eq:log_derivative_integrated} reads
    \begin{equation}\label{eq:as_exponent}
      \lim_{T \to \infty} \frac1T \log\|\zeta_T\| = \lambda_\nu \qquad \text{almost surely,}
    \end{equation}
    for every deterministic initial pair, hence for every pair independent of
    the noise.
    
    It remains to identify $\lambda_\nu$ with $\lambda_1$. By the multiplicative ergodic theorem (see e.g. \cite{BlumenthalPunshonSmith23}*{Theorem~2.2} for the LNS and PSA cocycles), for $\mu \otimes \PP$-almost every $(w_0,\omega)$ there is a closed subspace $L = L(w_0,\omega) \subset \mH$ with
    \begin{equation}\label{eq:cocycle_exponent}
      \lim_{T \to \infty} \frac1T \log\|\zeta_T\| \ =\ \lambda_1 ,
      \qquad \zeta_0 \notin L .
    \end{equation}
    Since $\lambda_1 > -\infty$ by \cite{HPRY24}, we have that $L$ is proper.

    Taking $w_0 \sim \mu$ and comparing with~\eqref{eq:lambda1-def} gives $\lambda_\nu \le \lambda_1$. Suppose $\lambda_\nu < \lambda_1$, and fix $w_0$ for which the above holds for almost every $\omega$. Since~\eqref{eq:as_exponent} holds from every deterministic initial state, every $\zeta_0 \ne 0$ has exponent $\lambda_\nu$, so $\zeta_0 \in L$ almost surely. Fixing a countable dense $\{\zeta_i\} \subset \mH \setminus \{0\}$, this holds for all $i$ on a full $\PP$-measure set, where $L \supseteq \overline{\{\zeta_i\}} = \mH$ because $L$ is closed, which goes against properness. Hence $\lambda_\nu = \lambda_1$, which is Corollary~\ref{cor:FK_formula}, and~\eqref{eq:as_exponent} is Corollary~\ref{cor:deterministic_datum}.

  \end{proof}

\appendix

\section{Moment bounds on the Jacobian}\label{app:jacobian_moments}

This appendix proves
Lemma~\ref{lem:fibre_dissipation}, Lemma~\ref{lem:fibre_smoothing} and
Lemma~\ref{lem:jacobian_moments}, all by means of parabolic energy estimates or
similar parabolic regularity tools. Throughout, $\Gamma = \Gamma(z)$ denotes~\eqref{eq:gamma_r} at the state $z$ under consideration.

\begin{proof}[Proof of Lemma~\ref{lem:fibre_dissipation}]
We argue only for LNS (the PSA case is strictly simpler), and treat separately
the blocks $J^{ww}$, $J^{\pi w}$ and $J^{\pi\pi}$ of~\eqref{eq:J-matrix}. We
start with $J^{ww}$. For fixed $\delta w$ the curve $\delta w_r \eqdef
J^{ww}_{s,r}\delta w$ solves the random linear equation
  \begin{equation}\label{eq:lin_vort_blue}
    \partial_r \delta w_r = L_{w_r}\delta w_r
    = \Delta\delta w_r - B(w_r,\delta w_r) - B(\delta w_r,w_r), \qquad \delta w_s = \delta w ,
  \end{equation}
in which the noise enters only through the coefficient field $w_r$.
We start with a standard $L^2$ estimate for~\eqref{eq:lin_vort_blue}. For any
$\ve \in (0,1)$ and $p_\ve \eqdef \frac{2}{2-\ve} \in (1,2)$, via~\eqref{eq:l4_toolkit}, interpolation and Young with exponents $2/\ve$
and $p_\ve$ we find 
  \begin{equation}\label{eq:gronwall_blue}
    \sup_{r\in[s,t]}\|\delta w_r\|^{2} + \int_s^t \|\delta w_r\|_{H^1}^2 \ud r
    \lesssim_{t-s} \exp\Big( 2C\int_s^t \|w_r\|_{H^1}^{p_\ve}\ud r \Big) \| \delta w \|^{2} .
  \end{equation}

  Next, we prove~\eqref{eq:fiber_dissipation} for an arbitrary $\delta\pi \in \mH$.
Then
\begin{equation}\label{eq:fd_terms}
  \langle\delta\pi, L^{\pi\pi}_z\delta\pi\rangle
  \ =\ \underbrace{-\|\nabla\delta\pi\|^2}_{\rm(I)}
  \ \underbrace{-\ \langle\delta\pi, B(\delta\pi,w)\rangle}_{\rm(II)}
  \ \underbrace{-\ \langle\pi,\delta\pi\rangle
    \bigl( \langle\pi,G_w\delta\pi\rangle
    + \langle\delta\pi,G_w\pi\rangle \bigr)}_{\rm(III)}
  \ \underbrace{-\ \langle\pi,G_w\pi\rangle\,\|\delta\pi\|^2}_{\rm(IV)} .
\end{equation}
Term~(I) is the dissipation and is kept. The other three are estimated as
follows.
\begin{enumerate}[label=(\roman*)]
  \item[(II)] Writing $u_\pi \eqdef K * \pi$ for the divergence-free velocity
    attached to $\pi$, so that $\|\nabla u_\pi\| = \|\pi\|$ and
    $\|u_\pi\|_{L^4} \lesssim \|\pi\|$, we find
    \begin{equation}\label{eq:fibre_stretch_absorb}
      |\langle\delta\pi,B(\delta\pi,w)\rangle|
      \ \le\ \frac12\|\nabla\delta\pi\|^2 + C\Gamma\|\delta\pi\|^2 .
    \end{equation}
  \item[(III)] For every $\varphi \in H^1$, we have
    \begin{equation}\label{eq:Gw_scalars}
      |\langle\pi,G_w\varphi\rangle| + |\langle\varphi,G_w\pi\rangle|
      \ \le\ C\bigl( \|\pi\|_{H^1} + \|w\|_{H^1} \bigr)\|\varphi\|_{H^1} .
    \end{equation}
  \item[(IV)] By antisymmetry of the transport term,
    $-\langle\pi,G_w\pi\rangle = \|\nabla\pi\|^2 + \langle\pi,B(\pi,w)\rangle$,
    and the two inequalities of~(II) give
    $|\langle\pi,B(\pi,w)\rangle| \lesssim \|\nabla\pi\|^{1/2}\|\nabla w\|
    \lesssim \Gamma$. Hence
    $-\langle\pi,G_w\pi\rangle \le C\Gamma$, which is~\eqref{eq:rayleigh_bound}.
\end{enumerate}
Adding the four contributions and using Young's inequality gives~\eqref{eq:fiber_dissipation}.

For the coupling term~\eqref{eq:coupling_operator}, in which
$(DL^0_w)[\delta w]\pi = B(\delta w,\pi) + B(\pi,\delta w)$, we have
\begin{equation}\label{eq:coupling_split}
  \langle\delta\pi,L^{\pi w}_z\delta w\rangle
  \ =\ \underbrace{-\,\bigl\langle \delta\pi,\, B(\delta w,\pi) + B(\pi,\delta w) \bigr\rangle}_{\rm(I)}
  \ +\ \underbrace{\langle\pi,\delta\pi\rangle
    \bigl\langle \pi,\, B(\delta w,\pi) + B(\pi,\delta w) \bigr\rangle}_{\rm(II)} .
\end{equation}
\begin{enumerate}
  \item[(I)] We use~\eqref{eq:l4_toolkit} together with $\|\pi\| = 1$ to obtain
    $|{\rm(I)}| \le C\|\pi\|_{H^1}\|\delta\pi\|\,\|\delta w\|_{H^1}$.
  \item[(II)] First,
    $\langle\pi,B(\delta w,\pi)\rangle = 0$. Then, the second term obeys
    $|\langle\pi,B(\pi,\delta w)\rangle| \lesssim
    \|\pi\|_{H^1}^{1/2}\|\delta w\|_{H^1}$ after one integration by parts
    and~\eqref{eq:l4_toolkit}.
\end{enumerate}
Since $\|\pi\|_{H^1} + \|\pi\|_{H^1}^2 \le C\Gamma$, applying Young and adding the two bounds gives
\begin{equation}\label{eq:coupling_bound}
  \bigl| \langle\delta\pi,L^{\pi w}_z\delta w\rangle \bigr|
  \ \le\ C\Gamma\,\|\delta\pi\|^2 + C\|\delta w\|_{H^1}^2 .
\end{equation}
Combining these last estimates with~\eqref{eq:fiber_dissipation} we obtain
    \begin{equation}\label{eq:fibre_energy_ineq}
      \frac{\ud}{\ud r}\|\delta\pi_r\|^2 + \frac12\|\nabla\delta\pi_r\|^2
      \le C\Gamma_r\|\delta\pi_r\|^2 + C\|\delta w_r\|_{H^1}^2 .
    \end{equation}
Then the claim~\eqref{eq:jacobian_pathwise} follows via 
Gr\"onwall together with~\eqref{eq:fibre_energy_ineq}   and~\eqref{eq:gronwall_blue}.
\end{proof}

Now we work toward the proof of Lemma~\ref{lem:jacobian_moments} with two
preliminary results. First we find the following.

\begin{lemma}\label{lem:fibre_smoothing}
  For every $T, K > 0$ there is a constant $C_{K,T} > 0$ such that, on
  $E_K$, $\forall s \in (0, T]$:
  \begin{equation}\label{eq:fibre_smoothing}
    \|\pi_s\|_{H^1} \le C_{K, T}  s^{-1/4}  .
  \end{equation}
\end{lemma}

\begin{proof}
  Fix $s \in (0,T]$ and work on $E_K$. Since
  $\int_{s/2}^{s}\|\pi_r\|_{H^1}^4 \ud r \le K^4$ and the interval has length
  $s/2$, there is an $r \in [s/2, s]$ with
  \begin{equation}\label{eq:fs_mvt}
    \|\pi_r\|_{H^1}^4 \ \le\ \frac{2K^4}{s} ,
    \qquad\text{that is}\qquad
    \|\pi_r\|_{H^1}^2 \ \le\ \frac{\sqrt2\, K^2}{s^{1/2}} .
  \end{equation}
  We propagate this from $r$ to $s$. By~\eqref{eq:fd_dirichlet} the Dirichlet
  quotient obeys
  $\frac{\ud}{\ud t}\|\pi_t\|_{H^1}^2 \le C\|w_t\|_{H^1}^2 \|\pi_t\|_{H^1}^2$,
  so Gr\"onwall on $[r,s]$ gives
  \begin{equation*}
     \|\pi_s\|_{H^1}^2
    \ \le\  \|\pi_r\|_{H^1}^2   
    \exp\Bigl( C\!\int_r^s \|w_t\|_{H^1}^2 \ud t \Bigr) .
  \end{equation*}
  Hence
  $\|\pi_s\|_{H^1} \le C_{K,T}\, s^{-1/4} $, which
  implies the desired result.
\end{proof}

The following is another supporting statement.
  \begin{lemma}\label{lem:heat_decay}
Fix $T, K > 0$, write $\ve \eqdef \frac18$, and let $\Gamma_r$ be as in~\eqref{eq:gamma_r}. For every $m > 0$ there is a constant $ C(m, K, T) > 0$ such that the following holds for every $N \ge 1$, every $\varrho \in [0,1]$ and all $0 \le s \le t \le T$. Let $\phi \colon [s,t] \to \mH$ solve
    \begin{equation}\label{eq:heat_decay_eq}
      \partial_r \phi_r = \Delta\phi_r + F_r , \qquad \phi_s = Q_N\phi_s ,
    \end{equation}
and suppose that on the event $E_K$ of~\eqref{eq:good_event_EK}:
    \begin{enumerate}[label=(\alph*)]
      \item The flow obeys the a priori bound
        \begin{equation}\label{eq:heat_decay_apriori}
          \sup_{r\in[s,t]}\|\phi_r\|^2 + \int_s^t \|\phi_r\|_{H^1}^2 \ud r
          \le m \|\phi_s\|^2 .
        \end{equation}
      \item The forcing splits as $F_r = a_r\phi_r + G_r + R_r$, where $a_r \le m\Gamma_r$, and
        \begin{equation}\label{eq:heat_decay_forcing}
          \|G_r\|_{H^{-1}} \le m\,\Gamma_r^{1/2}\|\phi_r\|_{H^\ve} ,
        \end{equation}
and the remainder satisfies for some $\varsigma_r \ge 0$,
        \begin{equation}\label{eq:heat_decay_source}
          |\langle R_r, Q_N\phi_r\rangle| + |\langle R_r, P_N\phi_r\rangle|
          \le \varsigma_r , \qquad
          \int_s^t \varsigma_r \ud r \le \varrho \, \|\phi_s\|^2 .
        \end{equation}
    \end{enumerate}
Then, pathwise on $E_K$,
    \begin{equation}\label{eq:heat_decay_concl}
      \|\phi_t\| \le C(m, K, T)\bigl( e^{-N^2(t-s)/2} + N^{-3/8} + \varrho^{1/2} \bigr)
      \|\phi_s\| .
    \end{equation}
  \end{lemma}

  \begin{proof}
Throughout we write $C_K$ for a constant that depends only on $m$, $K$ and $T$,
and we use twice that on $E_K$
\begin{equation}\label{eq:hd_gamma_moments}
  \int_s^t \Gamma_r \ud r \le C_K
  \qquad\text{and}\qquad
  \int_s^t \Gamma_r^2 \ud r \le C_K ,
\end{equation}
where the second holds because
$\Gamma_r^2 \le 2(\|w_r\|_{H^1}^4 + \|\pi_r\|_{H^1}^4)$. We decompose
$\phi_r = Q_N\phi_r + P_N\phi_r$, we close an energy estimate on each half, and
then we add the two.

\emph{High modes.} We apply $Q_N$ to~\eqref{eq:heat_decay_eq} and pair with
$Q_N\phi_r$, so that
$\tfrac12\frac{\ud}{\ud r}\|Q_N\phi_r\|^2
= -\|\nabla Q_N\phi_r\|^2 + \langle Q_N\phi_r, F_r\rangle$. Every mode of
$Q_N\phi_r$ has $|k| > N$, so we may extract the spectral gap and still keep
half of the dissipation for the transport term,
\[
  -\|\nabla Q_N\phi_r\|^2 \ \le\ -\frac12 N^2\|Q_N\phi_r\|^2
  - \frac12\|\nabla Q_N\phi_r\|^2 .
\]
We use~(b) to split $\langle Q_N\phi_r, F_r\rangle$ into three terms and we
estimate each. The first two are
\[
  \langle Q_N\phi_r, a_r\phi_r\rangle = a_r\|Q_N\phi_r\|^2
  \ \le\ m\Gamma_r\|Q_N\phi_r\|^2 ,
  \qquad
  |\langle Q_N\phi_r, R_r\rangle| \ \le\ \varsigma_r ,
\]
the second by~\eqref{eq:heat_decay_source}. For the third, duality,
\eqref{eq:heat_decay_forcing} and Young give
\[
  |\langle Q_N\phi_r, G_r\rangle|
  \ \le\ \|G_r\|_{H^{-1}}\|Q_N\phi_r\|_{H^1}
  \ \le\ m\Gamma_r^{1/2}\|\phi_r\|_{H^\ve}\|Q_N\phi_r\|_{H^1}
  \ \le\ \frac14\|\nabla Q_N\phi_r\|^2 + C\Gamma_r\|\phi_r\|_{H^\ve}^2 ,
\]
where the last step uses $\|Q_N\phi_r\|_{H^1}^2 \le 2\|\nabla Q_N\phi_r\|^2$ for
$N \ge 1$. We collect these and multiply by $2$,
\[
  \frac{\ud}{\ud r}\|Q_N\phi_r\|^2
  \ \le\ -N^2\|Q_N\phi_r\|^2 + C\Gamma_r\|Q_N\phi_r\|^2
  + C\Gamma_r\|\phi_r\|_{H^\ve}^2 + 2\varsigma_r .
\]
Since $\|Q_N\phi_s\| = \|\phi_s\|$, Gr\"onwall and~\eqref{eq:hd_gamma_moments}
give
\[
  \|Q_N\phi_t\|^2 \ \le\ C_K e^{-N^2(t-s)}\|\phi_s\|^2
  + C_K\!\int_s^t e^{-N^2(t-r)}\Gamma_r\|\phi_r\|_{H^\ve}^2\ud r
  + C_K\varrho\|\phi_s\|^2 .
\]
It remains to estimate the middle integral. We interpolate
$\|\phi_r\|_{H^\ve}^2 \lesssim \|\phi_r\|^{7/4}\|\phi_r\|_{H^1}^{1/4}$ at
$\ve = \tfrac18$, and we apply H\"older in time with the exponents
$(2,\tfrac83,8)$, admissible because $\tfrac12+\tfrac38+\tfrac18 = 1$. We place
$\Gamma_r$ in $L^2$, the weight in $L^{8/3}$ and $\|\phi_r\|_{H^1}^{1/4}$ in
$L^8$, while $\|\phi_r\|^{7/4}$ goes in $L^\infty$. The weight has mass
\[
  \Bigl(\int_s^t e^{-\frac83 N^2(t-r)}\ud r\Bigr)^{3/8}
  \ \le\ \Bigl(\frac{3}{8N^2}\Bigr)^{3/8} \ =\ C N^{-3/4} ,
\]
and the two members of~\eqref{eq:heat_decay_apriori} give
$\sup_r\|\phi_r\|^{7/4} \le m^{7/8}\|\phi_s\|^{7/4}$ and
$\bigl(\int_s^t\|\phi_r\|_{H^1}^2\ud r\bigr)^{1/8} \le m^{1/8}\|\phi_s\|^{1/4}$.
The middle integral is therefore at most $C_K N^{-3/4}\|\phi_s\|^2$, and
\begin{equation}\label{eq:hd_high}
  \|Q_N\phi_t\|^2
  \ \le\ C_K\bigl( e^{-N^2(t-s)} + N^{-3/4} + \varrho \bigr)\|\phi_s\|^2 .
\end{equation}

\emph{Low modes.} For the low modes we have
$\tfrac12\frac{\ud}{\ud r}\|P_N\phi_r\|^2
= -\|\nabla P_N\phi_r\|^2 + \langle P_N\phi_r, F_r\rangle$. The terms
$\langle P_N\phi_r, a_r\phi_r\rangle$ and $\langle P_N\phi_r, R_r\rangle$ are
estimated as above. For $\langle P_N\phi_r, G_r\rangle$, duality
and~\eqref{eq:heat_decay_forcing} give
\[
  |\langle P_N\phi_r, G_r\rangle|
  \ \le\ \|G_r\|_{H^{-1}}\|P_N\phi_r\|_{H^1}
  \ \le\ m\Gamma_r^{1/2}\|\phi_r\|_{H^\ve}\|P_N\phi_r\|_{H^1} ,
\]
and on the right side we split
$\|\phi_r\|_{H^\ve} \le \|Q_N\phi_r\|_{H^\ve} + \|P_N\phi_r\|_{H^\ve}$. We
use Young with exponents
$\bigl(\tfrac{2}{1+\ve},\tfrac{2}{1-\ve}\bigr)$:
\[
  m\Gamma_r^{1/2}\|P_N\phi_r\|_{H^\ve}\|P_N\phi_r\|_{H^1}
  \ \lesssim\ m\Gamma_r^{1/2}\|P_N\phi_r\|^{1-\ve}\|P_N\phi_r\|_{H^1}^{1+\ve}
  \ \le\ \frac14\|P_N\phi_r\|_{H^1}^2
  + C\Gamma_r^{\frac{1}{1-\ve}}\|P_N\phi_r\|^2 ,
\]
where $\Gamma_r^{1/(1-\ve)} = \Gamma_r^{8/7} \le \Gamma_r^2$ because
$\Gamma_r \ge 1$. The high-mode half we estimate instead by
\[
  m\Gamma_r^{1/2}\|Q_N\phi_r\|_{H^\ve}\|P_N\phi_r\|_{H^1}
  \ \le\ \frac14\|P_N\phi_r\|_{H^1}^2 + C\Gamma_r\|Q_N\phi_r\|_{H^\ve}^2 ,
\]
so that, collecting and multiplying by $2$,
\[
  \frac{\ud}{\ud r}\|P_N\phi_r\|^2
  \ \le\ C\Gamma_r^2\|P_N\phi_r\|^2 + C\Gamma_r\|Q_N\phi_r\|_{H^\ve}^2
  + 2\varsigma_r ,
\]
where the source carries $Q_N\phi_r$ alone. The data vanishes, so Gr\"onwall
from $P_N\phi_s = 0$ gives
\[
  \|P_N\phi_t\|^2 \ \le\ C_K\!\int_s^t \Gamma_r\|Q_N\phi_r\|_{H^\ve}^2\ud r
  + C_K\varrho\|\phi_s\|^2 .
\]
We estimate this integral using only that $Q_N\phi_r$ is supported on the
modes $|k| > N$, where $1 + |k|^2 > N^2$ and therefore
$(1+|k|^2)^{\ve} \le N^{-3/4}(1+|k|^2)^{1/2}$. Summing this over the modes, and
applying Cauchy--Schwarz to the sum that results,
\[
  \|Q_N\phi_r\|_{H^\ve}^2 \ \le\ N^{-3/4}\|Q_N\phi_r\|_{H^{1/2}}^2
  \ \le\ N^{-3/4}\|Q_N\phi_r\|\,\|Q_N\phi_r\|_{H^1} .
\]
Hence, by Cauchy--Schwarz in time, \eqref{eq:hd_gamma_moments} and the two members
of~\eqref{eq:heat_decay_apriori} give
\[
  \int_s^t \Gamma_r\|Q_N\phi_r\|_{H^\ve}^2\ud r
  \ \le\ N^{-3/4}\Bigl(\sup_{r\in[s,t]}\|\phi_r\|\Bigr)
    \Bigl(\int_s^t\Gamma_r^2\ud r\Bigr)^{1/2}
    \Bigl(\int_s^t\|\phi_r\|_{H^1}^2\ud r\Bigr)^{1/2}
  \ \le\ C_K N^{-3/4}\|\phi_s\|^2 .
\]
Hence
\begin{equation}\label{eq:hd_low}
  \|P_N\phi_t\|^2 \ \le\ C_K\bigl( N^{-3/4} + \varrho \bigr)\|\phi_s\|^2 .
\end{equation}
Finally, since
$\|\phi_t\|^2 = \|Q_N\phi_t\|^2 + \|P_N\phi_t\|^2$, the two
bounds~\eqref{eq:hd_high} and~\eqref{eq:hd_low} give
\[
  \|\phi_t\|^2 \ \le\ C_K\bigl( e^{-N^2(t-s)} + N^{-3/4} + \varrho \bigr)
  \|\phi_s\|^2.
\]

  \end{proof}

\begin{proof}[Proof of Lemma~\ref{lem:jacobian_moments}]
We argue again only for LNS, and write $\ve = \frac18$ as in
Lemma~\ref{lem:heat_decay} ($\ve$ need only be sufficiently small). We
must show $\|J_{s,t}Q_N\eta\|\le\delta\|\eta\|$ on $E_K$. Writing $\eta=(\delta
w,\delta\pi)$, the block structure~\eqref{eq:grand_jacobian_system} gives us the decomposition
$$J_{s,t}Q_N\eta=\bigl(J^{ww}_{s,t}Q_N\delta w,\ J^{\pi\pi}_{s,t}Q_N\delta\pi +
J^{\pi w}_{s,t}Q_N\delta w\bigr), $$ 
so it suffices to prove high-mode decay
for each of the three blocks separately.

For the first block we set
$v_r = J^{ww}_{s,r}Q_N\delta w$, which solves~\eqref{eq:heat_decay_eq}
with initial data $v_s = Q_N\delta w$ and forcing
$F_r = -B(w_r,v_r) - B(v_r,w_r)$, so that we can set $a_r = 0$ and $R_r = 0$, and hence
$\varrho = 0$, and $G_r =F_r$ which satisfies the forcing
bound~\eqref{eq:heat_decay_forcing} because
$B(a,b) = \div(u_ab)$ with $\|u_a\|_{L^\infty}\lesssim\|a\|_{H^\ve}$
by~\eqref{eq:l4_toolkit}, so that
$\|F_r\|_{H^{-1}} \lesssim \|w_r\|_{H^1}\|v_r\|_{H^\ve} \le C\Gamma_r^{1/2}\|v_r\|_{H^\ve}$.

Moreover, from Lemma~\ref{lem:fibre_dissipation}(iii) applied
to the initial data $Q_N\delta w$, we obtain the a priori
bound~\eqref{eq:heat_decay_apriori} as follows for some $C_K>0$
    \begin{equation}\label{eq:apriori_v}
      \sup_{r\in[s,T]}\|v_r\|^2 \le C_K\|\delta w\|^2 ,
      \qquad \int_s^t\|v_r\|_{H^1}^2\ud r \le C_K\|\delta w\|^2 .
    \end{equation}
Here and in the next block Lemma~\ref{lem:heat_decay} is applied with $m = C_K$.
Therefore, we obtain
    \begin{equation}\label{eq:jww_decay}
      \|J^{ww}_{s,t}Q_N\delta w\|
      \le C_K\bigl(e^{-N^2(t-s)/2} + N^{-3/8}\bigr)\|\delta w\| ,
    \end{equation}
which is at most $\delta\|\delta w\|$ on $E_K$ as soon as $t - s \ge \tau_0$ and $N \ge N_0(\delta,\tau_0,K,T)$.

Next, we set $\phi_r \eqdef J^{\pi\pi}_{s,r}Q_N\delta\pi$, which solves
$\partial_r\phi_r = L^{\pi\pi}_{z_r}\phi_r$ with initial data
$\phi_s = Q_N\delta\pi$. We write $G_w = \Delta - L^0_w$ and recall
from~\eqref{eq:fiber_operator} that 
\[ L^{\pi\pi}_z\delta\pi = G_w\delta\pi -
\langle\pi,G_w\delta\pi\rangle\pi - \langle\delta\pi,G_w\pi\rangle\pi -
\langle\pi,G_w\pi\rangle\delta\pi . \]
Hence this is~\eqref{eq:heat_decay_eq}
with $a_r = -\langle\pi_r,G_{w_r}\pi_r\rangle$, with $G_r = -L^0_{w_r}\phi_r$,
and with the rank-one remainder
$R_r = -\bigl(\langle\pi_r,G_{w_r}\phi_r\rangle + \langle\phi_r,G_{w_r}\pi_r\rangle\bigr)\pi_r$.
The scalar obeys the one-sided bound
$a_r = -\langle\pi_r,G_{w_r}\pi_r\rangle \le C\Gamma_r$, which is
precisely~\eqref{eq:rayleigh_bound}. The transport part
obeys~\eqref{eq:heat_decay_forcing} exactly as for $J^{ww}$. 

As for the
remainder, scalars in $R_r$ are at most
$C\bigl(\|\pi_r\|_{H^1} + \|w_r\|_{H^1}\bigr)\|\phi_r\|_{H^1}$
by~\eqref{eq:Gw_scalars}. There remains the pairing of $\pi_r$ with the two projections of
$\phi_r$. Since $Q_N$ is self-adjoint and $\|Q_N\pi\|\le N^{-1}\|\pi\|_{H^1}$,
\[ |\langle\pi_r,Q_N\phi_r\rangle| = |\langle Q_N\pi_r,Q_N\phi_r\rangle| \le
N^{-1}\|\pi_r\|_{H^1}\|\phi_r\| , \] 
while for the pairing
$\langle\pi_r,P_N\phi_r\rangle = \langle\pi_r,\phi_r\rangle -
\langle\pi_r,Q_N\phi_r\rangle$, we find:
\[ \frac{\ud}{\ud
r}\langle\pi_r,\phi_r\rangle = \langle G_{w_r}\pi_r - \langle
G_{w_r}\pi_r,\pi_r\rangle\pi_r,\phi_r\rangle +
\langle\pi_r,L^{\pi\pi}_{z_r}\phi_r\rangle =
-2\langle\pi_r,G_{w_r}\pi_r\rangle\langle\pi_r,\phi_r\rangle , \] so that
$\langle\pi_r,\phi_r\rangle = \langle Q_N\pi_s,\phi_s\rangle
\exp\bigl(-2\int_s^r\langle\pi_u,G_{w_u}\pi_u\rangle\ud u\bigr)$ because the
initial condition is projected onto high frequencies.
Next we observe that $-\langle\pi,G_w\pi\rangle \le C\Gamma$ by~\eqref{eq:rayleigh_bound} and
$|\langle Q_N\pi_s,\phi_s\rangle| \le N^{-1}\|\pi_s\|_{H^1}\|\phi_s\|$, and
since
$s \ge \tau_0/2$, by
Lemma~\ref{lem:fibre_smoothing} we obtain
$\|\pi_s\|_{H^1} \le C_K (\tau_0/2)^{-1/4}  = C_{K,\tau_0}$
on $E_K$. 

Now we have
\begin{equation}\label{e:aprior-phip}
  \sup_{r \in [s, t]}\|\phi_r\|^2 + \int_s^t \| \phi_r \|_{H^1}^2
\ud s \le C_K\|\phi_s\|^2
\end{equation}
which is Lemma~\ref{lem:fibre_dissipation}(iii) applied to
$\eta = (0, Q_N\delta\pi)$.
Hence we
conclude that 
\[ |\langle\pi_r,Q_N\phi_r\rangle| +
|\langle\pi_r,P_N\phi_r\rangle| \le C_{K,\tau_0}N^{-1} 
\|\pi_r\|_{H^1} \|\phi_s\| , \] 
so that~\eqref{eq:heat_decay_source} holds with 
\[ \varsigma_r \eqdef C_{K,\tau_0}N^{-1}
\bigl(\|\pi_r\|_{H^1}+\|w_r\|_{H^1}\bigr)
\|\pi_r\|_{H^1} \|\phi_r\|_{H^1}\|\phi_s\| . \] 
H\"older in time
with exponents $(4,4,2)$, the fourth moments of $E_K$ and the a priori bound
\eqref{e:aprior-phip} give the estimate $\int_s^t\varsigma_r\ud r \le
C_{K,\tau_0}N^{-1}\|\phi_s\|^2$, which implies
$\varrho = C_{K,\tau_0}N^{-1}$. Since $N^{-1/2} \le N^{-3/8}$ for $N \ge 1$,
Lemma~\ref{lem:heat_decay} now gives \[
\|J^{\pi\pi}_{s,t}Q_N\delta\pi\| \le C_{K,\tau_0}\bigl(e^{-N^2(t-s)/2}
+ N^{-3/8}\bigr)\|\delta\pi\| . \] 
Now the right side tends to
zero as $N \to \infty$, so
$\|J^{\pi\pi}_{s,t}Q_N\delta\pi\|\le\delta\|\delta\pi\|$ for
$N \ge N_0(\delta,\tau_0,K,T)$.

Now we pass to the last block, which starts from zero initial datum.
We use the notation $v_r = J^{ww}_{s,r} Q_N \delta w$ and
$\xi_r = J^{\pi w}_{s,r} Q_N \delta w$, so that
$(v_r,\xi_r) = J_{s,r}(Q_N\delta w,0) \in T_{z_r}\mathcal{X}$ and hence
$\xi_r \in \pi_r^\perp$, and $\xi_r$ solves
$\partial_r \xi_r = L^{\pi\pi}_{z_r}\xi_r + L^{\pi w}_{z_r} v_r$ with
$\xi_s = 0$. 
We have
$\tfrac12\frac{\ud}{\ud r}\|\xi_r\|^2
= \langle\xi_r, L^{\pi\pi}_{z_r}\xi_r\rangle + \langle\xi_r, L^{\pi w}_{z_r}v_r\rangle$.
The first pairing obeys~\eqref{eq:fiber_dissipation}, that is
$\langle\xi_r,L^{\pi\pi}_{z_r}\xi_r\rangle \le -\tfrac14\|\nabla\xi_r\|^2 + C\Gamma_r\|\xi_r\|^2$.
In the second, the rank one term of~\eqref{eq:coupling_operator} has range $\RR\pi_r$
and therefore drops, because $\xi_r \in \pi_r^\perp$, which leaves
\[
  \langle\xi_r, L^{\pi w}_{z_r}v_r\rangle
  \ =\ -\langle \xi_r, B(v_r,\pi_r)\rangle - \langle\xi_r, B(\pi_r,v_r)\rangle .
\]
In both terms we move the derivative onto $\xi_r$, which is legitimate because the
transporting fields $K*v_r$ and $K*\pi_r$ are divergence free, and we then use
$\|\pi_r\| = 1$ with the third estimate of~\eqref{eq:l4_toolkit}. This gives
\[
  \begin{aligned}
    |\langle\xi_r, B(v_r,\pi_r)\rangle|
      &= |\langle \pi_r, (K*v_r)\cdot\nabla\xi_r\rangle|
       \ \le\ \|K*v_r\|_{L^\infty}\|\nabla\xi_r\|
       \ \lesssim\ \|v_r\|_{H^\ve}\|\nabla\xi_r\| , \\
    |\langle\xi_r, B(\pi_r,v_r)\rangle|
      &= |\langle v_r, (K*\pi_r)\cdot\nabla\xi_r\rangle|
       \ \le\ \|K*\pi_r\|_{L^\infty}\|v_r\|\|\nabla\xi_r\| \\
      &\ \lesssim\ \|\pi_r\|_{H^1}\|v_r\|\|\nabla\xi_r\| ,
  \end{aligned}
\]
because $\|\pi_r\|_{H^\ve} \le \|\pi_r\|_{H^1}$. Young's inequality leaves us
then with
    \begin{equation*}
      \frac{\ud}{\ud r}\|\xi_r\|^2 \le C\Gamma_r\|\xi_r\|^2
      + C \|\pi_r\|_{H^1}^2 \|v_r\|^2 + C \|v_r\|_{H^\ve}^2 ,
    \end{equation*}

Then Gr\"onwall and $\xi_s = 0$ give
\[
  \|\xi_t\|^2 \ \le\ C\exp\Bigl(C\!\int_s^t\Gamma_r\ud r\Bigr)
  \int_s^t \Bigl( \|\pi_r\|_{H^1}^2\|v_r\|^2 + \|v_r\|_{H^\ve}^2 \Bigr)\ud r
  \ \le\ C_K\bigl( \mathrm{I} + \mathrm{II} \bigr) ,
\]
because $\int_s^t\Gamma_r\ud r \le C(1+K^2)$ on $E_K$, where $\mathrm{I}$ and
$\mathrm{II}$ denote the integrals of the two forcing terms. We start by
observing that if we square~\eqref{eq:jww_decay} and integrate in $r$, we obtain
\[
  \begin{aligned}
    \int_s^t\|v_r\|^2\ud r
      &\ \le\ C_K\int_s^t\bigl( e^{-N^2(r-s)} + N^{-3/4} \bigr)\ud r\,\|\delta w\|^2 \\
      &\ \le\ C_K\bigl( N^{-2} + TN^{-3/4} \bigr)\|\delta w\|^2
       \ \le\ C_K N^{-3/4}\|\delta w\|^2 .
  \end{aligned}
\]
For the first term we must estimate, Cauchy--Schwarz in time and the fourth moments of $E_K$ give
\[
  \begin{aligned}
    \mathrm{I} &\ \le\ \Bigl(\int_s^t\|\pi_r\|_{H^1}^4\ud r\Bigr)^{1/2}
      \Bigl(\int_s^t\|v_r\|^4\ud r\Bigr)^{1/2} , \\
    \int_s^t\|v_r\|^4\ud r &\ \le\ \Bigl(\sup_{r\in[s,t]}\|v_r\|^2\Bigr)
      \int_s^t\|v_r\|^2\ud r \ \le\ C_K N^{-3/4}\|\delta w\|^4 ,
  \end{aligned}
\]
by the supremum bound of~\eqref{eq:apriori_v}, so that
$\mathrm{I} \le C_K N^{-3/8}\|\delta w\|^2$. For the second term, the
interpolation $\|v_r\|_{H^\ve} \lesssim \|v_r\|^{1-\ve}\|v_r\|_{H^1}^{\ve}$ and
H\"older in time with exponents $\bigl(\tfrac{1}{1-\ve},\tfrac1\ve\bigr)$ give
\[
  \mathrm{II} \ \lesssim\ \int_s^t\|v_r\|^{2(1-\ve)}\|v_r\|_{H^1}^{2\ve}\ud r
  \ \le\ \Bigl(\int_s^t\|v_r\|^2\ud r\Bigr)^{1-\ve}
  \Bigl(\int_s^t\|v_r\|_{H^1}^2\ud r\Bigr)^{\ve}
  \ \le\ C_K N^{-3(1-\ve)/4}\|\delta w\|^2 ,
\]
by the integral bound of~\eqref{eq:apriori_v}. Taking $N$ sufficiently large
leads to the desired result.

\end{proof}

\section{Analyticity of the passive scalar with respect to Cameron--Martin shifts}\label{app:analyticity}

This appendix proves Lemma~\ref{lem:analytic}, the Cameron--Martin analyticity
of the passive scalar solution map. This rests on a Weierstrass
convergence theorem \cite{ABHN2011}*{Theorem~A.5} and the fact that
Navier--Stokes is holomorphic in the forcing, which was already somewhat known
(analyticity is proven in~\cite[Theorem~2.1]{Kuksin82}), so we only sketch the
argument of the proof.

\begin{proof}[Proof of Lemma~\ref{lem:analytic}]

Write $h^x \eqdef \sum_{i=1}^n x_i h_i$ for a Cameron--Martin shift, and $w^x$
for the associated vorticity, which solves:
  \begin{equation*}
    \partial_t w^x + (K*w^x) \cdot \nabla w^x = \Delta w^x + \curl \xi + Q^{1/2}\dot{h}^x
  \end{equation*}
and $u^x \eqdef K * w^x$ for the associated velocity, $\zeta^x$ for the passive
scalar transported by $u^x$ from $\zeta_0$. Since $D_B$ is a polynomial it carries
holomorphic maps of $z$ to holomorphic maps.

The problem is thereby reduced to the holomorphy of
$z \mapsto \zeta^z_T \in H^1(\TT^2,\CC)$, where the maps above are extended
from $\RR^n$ to $\CC^n$. To this end, we write $\zeta^z$ in
its mild form:
  \begin{equation}\label{eq:duhamel_scalar}
    \zeta^z_t = e^{t\Delta}\zeta_0
    - \int_0^t e^{(t-r)\Delta}\bigl(u^z_r\cdot\nabla\zeta^z_r\bigr)\ud r.
  \end{equation}
and iterate, to build an iterated expansion:
  \begin{equation*}
    \zeta^z_t = \sum_{m\geq0}\zeta^{z,(m)}_t, \qquad \zeta^{z,(m)}_t = - \int_0^t e^{(t-s) \Delta} \Big[ u^{z}_s \cdot \nabla \zeta^{z, (m-1)}_s \Big] \ud s, \qquad \zeta^{z, (0)}_t = e^{t \Delta } \zeta_0.
  \end{equation*}
In particular $\zeta^{z,(m)}_t$ is a fixed bounded $m$-linear form evaluated on the diagonal $(u^z,\dots,u^z)$.

First, we will find a measurable set $\Omega_0$ with $\PP(\Omega_0)=1$ and a warm-up time $\eta \in (0, T \wedge 1)$ such that for every $\omega\in\Omega_0$ and every real point $x_0\in\RR^n$ there is an $r=r(x_0)>0$ for which, writing
  \begin{equation}\label{eq:polydisc}
    \mathcal{P} \eqdef \{z\in\CC^n \colon |z_j - x_{0,j}| < r \text{ for } j=1,\dots,n\}
  \end{equation}
for the polydisc of radius $r$ about $x_0$, the map
  \begin{equation}\label{eq:vel_holo}
    \mathcal{P} \ni z \mapsto u^z \in C\bigl([\eta,T];C^1(\TT^2,\CC^2)\bigr) \quad \text{ is holomorphic on $\mathcal{P}$},
  \end{equation}
and satisfies
  \begin{equation}\label{eq:vel_bound}
    M_{\mathcal{P}} \eqdef \sup_{z\in\mathcal{P}}\ \sup_{\eta \le \tau\leq T}\ \|u^z_\tau\|_{C^1} \ <\ \infty .
  \end{equation}
Second, we will prove that each term obeys, uniformly over $m\geq 0$ and $z\in\mathcal{P}$ (for an arbitrary polydisc $\mathcal{P}$):
  \begin{equation}\label{eq:neumann_bound}
    \bigl\|\zeta^{z,(m)}_T\bigr\|_{H^1}
    \ \leq\ \|\zeta_0\|_{H^1}\ \frac{(C M_{\mathcal{P}} \sqrt{T})^{m}}{\Gamma\bigl(1+\frac m2\bigr)}.
  \end{equation}
Granting these, $z\mapsto\zeta^{z,(m)}_T$ is a continuous polynomial in $u^z$
composed with a holomorphic map, hence holomorphic on
$\mathcal{P}$, and the series converges uniformly there, so
the Weierstrass convergence theorem \cite{ABHN2011}*{Theorem~A.5}, applied in each variable separately,
shows that $z\mapsto\zeta^z_T$ is holomorphic on $\mathcal{P}$
with values in $H^1(\TT^2,\CC)$.

We only sketch the proof of~\eqref{eq:vel_holo} and~\eqref{eq:vel_bound} since
this is somewhat standard (and analyticity was already proved in~\cite[Theorem~2.1]{Kuksin82}).
We start with the first claim.
We define $Z$ to be the $z$-independent stochastic convolution $\partial_\tau Z=\Delta Z+\curl\xi$ with $Z_0=0$, and fix any $s \in (2, 3)$, then set
  \begin{equation}\label{eq:Omega0}
    \Omega_0\eqdef\bigl\{\omega:\ Z\in C([0,T];H^s(\TT^2))\bigr\},
  \end{equation}
so that $\PP(\Omega_0)=1$ by parabolic regularity.
On $\Omega_0$ we define $v^z\eqdef w^z-Z$, which solves the random PDE
  \begin{equation}\label{eq:v_eq}
    \partial_\tau v^z=\Delta v^z-B(Z+v^z,Z+v^z)+g^z,\qquad
    g^z\eqdef Q^{1/2}\dot h^z,\qquad v^z_0=w_0,
  \end{equation}
and it suffices to prove~\eqref{eq:vel_holo} and~\eqref{eq:vel_bound} for $v^z$.

One observation to make is that for complex-valued forcing the usual
cancellation in the non-linearity which is used for the global well-posedness of
the vorticity equation $\langle B(a,b), b\rangle = 0$ does not vanish, so we are
not able to construct the solution $v^z$ from scratch. 
Instead,  we can consider $v^z$ as a perturbation of the global real solution
$v^{x_0}$ (with $x_0$ real) over small intervals. Since $w_0$ lies only in $\mH$
and so is irregular (not in $C^1$) we measure in two regimes separated
by $\eta$. Before $\eta$ we use the norm
    \begin{equation*}
      \|v\|_{\mathcal K_\eta} \eqdef \sup_{\tau \le \eta} \bigl( \|v_\tau\|
      + \tau^{\gamma}\|v_\tau\|_{H^{1+\delta}} \bigr) ,
      \qquad \delta \in \bigl(0, \frac14\bigr)  ,
      \qquad \gamma = \frac{1+\delta}{2} < \frac34 ,
    \end{equation*}
and after $\eta$ we consider the uniform norm in $H^s$. We write $\|\cdot\|_{\mathcal K(I)}$ for the same
norm restarted at the left endpoint $a$ of an interval $I$. Namely for any $I = [a, a+t]\subseteq [0,\eta]$ or $ I \subseteq [\eta, T]$, write
$\|\cdot\|_I$ for the norm 
\[
  \|d\|_{I} \eqdef \|d\|_{\mathcal K(I)} \quad\text{for}\quad I \subseteq [0,\eta] ,
  \qquad
  \|d\|_{I} \eqdef \sup_{\tau\in I}\|d_\tau\|_{H^s}
  \quad\text{for}\quad I \subseteq [\eta,T] ,
\]
Now let us set
    \begin{equation}\label{eq:polydisc_R}
      R \eqdef M_Z + \mathcal M_0 + 1 ,
      \qquad M_Z \eqdef \sup_{\tau\le T}\|Z_\tau\|_{H^s} ,
      \qquad \mathcal M_0 \eqdef \|v^{x_0}\|_{\mathcal K_\eta}
      + \sup_{\eta/2 \le \tau\le T}\|v^{x_0}_\tau\|_{H^s} ,
    \end{equation} 
Then set $V' = Z + v^{x_0}$. Then $d=v^z - v^{x_0}$ on $I$ if and only if it is a fixed
point $d = \Phi_I d$ of
\[
  (\Phi_I d)_\tau \eqdef e^{(\tau-a)\Delta} d_a
  - \int_a^\tau e^{(\tau-r)\Delta}
    \Bigl[ B\bigl(d_r, V'_r + d_r\bigr) + B\bigl(V'_r, d_r\bigr) \Bigr] \ud r
  + \int_a^\tau e^{(\tau-r)\Delta}\bigl(g^z_r - g^{x_0}_r\bigr) \ud r .
\]
On the unit ball
$\mathcal B_I \eqdef \{ d \colon \|d\|_I \le 1 \}$, since $\|d\|_I \le 1$ we
find a $\mu >0$ such that
    \begin{equation}\label{eq:contraction_modulus}
      \bigl\|\Phi_I d - \Phi_I \tilde d\bigr\|_I
      \ \le\ C R\, t^{1-\mu}\, \bigl\|d - \tilde d\bigr\|_I ,
      \qquad d, \tilde d \in \mathcal B_I .
    \end{equation}
We therefore choose a $t_R$ such that for $t <
t_R$~\eqref{eq:contraction_modulus} gives a contraction.

    Now we partition
    $[0,\eta]$ and $[\eta, T]$ separately into $N$ intervals of length at most
    $t_R$, where $\eta$ is the $j_0$-th interval starting point, and let us
    denote by $\ve$ the maximal radius
    \begin{equation}\label{eq:polydisc_eps}
      \ve \eqdef \sup_{z \in \mathcal{P}}
      \bigl\|g^z - g^{x_0}\bigr\|_{L^2([0,T];H^s)}.
    \end{equation}
  Then, provided $\ve$ is sufficiently small (which corresponds to $r$ small in
  the definition of the polydisc), we can successively solve for $d$
  on each small interval of length $t_R$, thus obtaining~\eqref{eq:vel_bound}.
  Moreover, the fixed point guarantees an analytic expansion in the forcing, so
  that also~\eqref{eq:vel_holo} is satisfied.

Finally, for~\eqref{eq:neumann_bound} write $F_m(t) \eqdef \bigl\|\zeta^{z,(m)}_t\bigr\|_{H^1}$ for
$t \in [\eta, T]$. Parabolic smoothing and H\"older give
$\bigl\|e^{(t-r)\Delta}f\bigr\|_{H^1} \le C(t-r)^{-1/2}\|f\|$ and
$\|u^z_r \cdot \nabla \zeta\| \le \|u^z_r\|_{L^\infty}\|\zeta\|_{H^1}
\le M_{\mathcal{P}}\|\zeta\|_{H^1}$ by~\eqref{eq:vel_bound}, so recursion collapses to the scalar inequality
\[
  F_m(t) \ \le\ C M_{\mathcal{P}} \int_\eta^t (t-r)^{-1/2} F_{m-1}(r) \ud r ,
  \qquad
  F_0(t) \ \le\ \bigl\|\zeta^z_\eta\bigr\|_{H^1} .
\]
Iterating it needs one integral only, the Euler Beta integral
    \begin{equation}\label{eq:beta_iteration}
      \int_\eta^t (t-r)^{-1/2}\,(r-\eta)^{\frac{m-1}{2}} \ud r
      \ =\ \frac{\Gamma\bigl(\frac12\bigr)\,\Gamma\bigl(\frac{m+1}{2}\bigr)}
                {\Gamma\bigl(\frac m2 + 1\bigr)}\ (t-\eta)^{\frac m2} .
    \end{equation}
Feeding the recursion the bound at $m-1$ returns the very same expression at
$m$, because the numerator $\Gamma\bigl(\frac{m+1}{2}\bigr)$
of~\eqref{eq:beta_iteration} is the denominator
$\Gamma\bigl(1+\frac{m-1}{2}\bigr)$ carried over from the previous step. Hence
induction from $F_0$ gives
\[
  F_m(t) \ \le\ \bigl\|\zeta^z_\eta\bigr\|_{H^1}\,
  \frac{\bigl( C M_{\mathcal{P}}\, \Gamma\bigl(\frac12\bigr)\, \sqrt{t-\eta} \bigr)^{m}}
       {\Gamma\bigl(1 + \frac m2\bigr)} ,
\]
which is~\eqref{eq:neumann_bound} at $t = T$, once $\Gamma(\tfrac12) = \sqrt\pi$
is absorbed into $C$. The Gamma factor grows faster than any geometric
sequence, so the series $\sum_m \zeta^{z,(m)}_T$ converges uniformly on
$\mathcal{P}$.

\end{proof}

\end{document}